\documentclass[11pt,reqno]{amsart}

\RequirePackage{amsthm,amsmath,amsfonts}

\usepackage{color}
\usepackage{amsmath}
\usepackage{amssymb}
\usepackage{amsthm}
\usepackage{amsfonts}
\usepackage{bbm}
\usepackage{cancel}
\usepackage{dsfont}
\usepackage{verbatim}
\usepackage{mathtools}
\usepackage{enumerate}
\usepackage{enumitem}
\usepackage{mathrsfs}
\usepackage{stmaryrd}
\usepackage{graphicx}
\usepackage[left=2.6cm, right=2.6cm, top=2.6cm, bottom= 2.6cm]{geometry}

\newtheorem{assumption}{Assumption}

\definecolor{darkgreen}{rgb}{0,0.5,0}

\usepackage{setspace}
\numberwithin{equation}{section}

\usepackage[colorlinks=true, pdfstartview=FitV, linkcolor=blue, citecolor=blue, urlcolor=blue,pagebackref=false]{hyperref}

\newcommand{\bP}{\mathbb{P}}

\newcommand{\Q}{\mathbb{Q}}

\newcommand{\E}{\mathbb{E}}
\newcommand{\N}{\mathbb{N}}

\newcommand{\Z}{\mathbb{Z}}

\newcommand{\R}{\mathbb{R}}

\newcommand{\eps}{\epsilon}

\newcommand{\cA}{\mathcal{A}}

\newcommand{\cL}{\mathcal{L}}

\newcommand{\cB}{\mathcal{B}}
\newcommand{\cU}{\mathcal{U}}
\newcommand{\cZ}{\mathcal{Z}}

\def\restrict#1{\raise-.5ex\hbox{\ensuremath|}_{#1}}

\newcommand{\Xrap}{\mathbb{X}_{\mathrm{rap}}}
\newcommand{\Xtemp}{\mathbb{X}_{\mathrm{temp}}}

\newcommand{\indc}{\mathds{1}}

\newcommand{\op}{\bar{p}}

\newcommand{\sL}{\mathscr{L}}
\newcommand{\sM}{\mathscr{M}}

\newcommand{\sF}{\mathscr{F}}

\definecolor{darkgreen}{rgb}{0,0.5,0}

\newcommand\restr[2]{{
  \left.\kern-\nulldelimiterspace 
  #1 
  \vphantom{\big|} 
  \right|_{#2} 
  }}

\newtheorem{theorem}{Theorem}[section]
\newtheorem{lemma}[theorem]{Lemma}
\newtheorem{proposition}[theorem]{Proposition}
\newtheorem{corollary}[theorem]{Corollary}

\theoremstyle{definition}
\newtheorem{remark}[theorem]{Remark}
\newtheorem{definition}[theorem]{Definition}

\usepackage[foot]{amsaddr}

\begin{document}

\title[Compact support for infinite-dimensional SDEs]{A compact support property for infinite-dimensional SDEs with Hölder continuous coefficients}

\author{Thomas Hughes \and Marcel Ortgiese}
\address{Department of Mathematical Sciences, University of Bath, UK}




\maketitle

\begin{abstract}
We consider non-negative solutions to some infinite-dimensional SDEs on $\Z^d$ with 
H\"older continuous noise coefficients. We prove that if the H\"older exponent is less than $1/2$, solutions 
are compactly supported for almost all times, a variant of the classical compact support property for SPDEs.
Our results imply that the instantaneous propagation of supports for superprocesses associated to discontinuous spatial motions is effectively sharp. 
We also show in a special case that the support is unbounded on a dense set of times.

The proof uses a general approach which we expect can be applied to prove similar results for non-local SPDEs. It is based on an analysis of the excursions and zero sets of semimartingales whose quadratic variation satisfies a certain lower bound. As a corollary of our method, we show that the zero sets of non-negative solutions to some simple one-dimensional SDEs have positive Lebesgue measure, despite the absence of ``sticky'' dynamics.
\end{abstract}

\tableofcontents
\section{Introduction}

How quickly a random spatial process spreads mass is determined by the competition between local fluctuations, which may locally lead to extinction, and the underlying spatial motion, which transports mass back in from neighbouring regions. The behaviour of the support of the process is often determined by the interplay and relative strength of these effects.
This paper concerns the propagation of supports of spatial stochastic models, notably non-negative solutions to infinite-dimensional {\it stochastic differential equations} (SDEs) and  {\it stochastic partial differential equations} (SPDEs). Our main result establishes compact support-type behaviour for a family of stochastic models in which the underlying spatial motion is discontinuous. The behaviour established is distinct from the classical compact support property exhibited by solutions to certain SPDEs, while it is also, to our knowledge, the first result concerning compact supports for a spatial process associated to a discontinuous motion.

Our results concern solutions to infinite-dimensional SDEs, but they are better understood in the context of the literature on SPDEs, so we begin our discussion there. Consider the multiplicative stochastic heat equation on the real line given formally by
\begin{equation} \label{eq_SHE}
	\partial_t u(t,x) = \frac 1 2 \Delta u(t,x) + u(t,x)^\gamma \dot{W}(t,x),\quad t >0,  \,x \in \R,
\end{equation}
where $\dot{W}$ is a space-time white Gaussian noise and $\gamma >0$. We will always consider non-negative solutions in this work, and $u(t,x)^\gamma$ is therefore a suitable shorthand for $|u(t,x)|^\gamma \indc (u(t,x) \geq 0)$. It is well-known that the behaviour of the supports of solutions to \eqref{eq_SHE}, and of more general parabolic SPDEs for which the noise term has the same form, exhibits a transition when the parameter $\gamma$ is varied. In particular, if $\gamma \in (0,1)$, a solution to \eqref{eq_SHE} with compactly supported initial data has the {\it compact support property} (CSP): for every $T>0$, 
\begin{equation} \label{CSP}
	\bP\left( \mathrm{cl} \left( \cup_{t \in [0,T]} \, \mathrm{supp}(u(t,\cdot)) \right) \text{ is compact} \right) = 1, 
\end{equation}
where $\mathrm{cl}(A)$ denotes the closure of $A$. On the other hand, if $\gamma \geq 1$, solutions a.s. satisfy $u(t,x) > 0$ for all $(t,x) \in (0,\xi) \times \R$, where $\xi \in (0,+\infty]$ is a blow-up time until which solutions are defined, which may or may not be finite with positive probability depending on $\gamma$ \cite{M2000}. The strict positivity result for $\gamma \geq 1$ is originally due to Mueller \cite{M1990}, see also \cite{MF2014} in the case $\gamma  = 1$. Strict positivity in the linear case has important implications, as it allows one to take the logarithm of a solution, which defines the Cole-Hopf solution of the KPZ equation \cite{BG1997}.

The present paper is concerned with the case of (strictly) Hölder continuous coefficients, i.e.\ $\gamma \in (0,1)$. Early work on the CSP for \eqref{eq_SHE} is closely related to the study of {\it super-Brownian motion} (SBM), which is the universal measure-valued scaling limit of critical branching Markov processes under finite variance conditions on the spatial motion and the branching mechanism (see e.g. \cite{Perkins}). In the case $\gamma = 1/2$, a solution to \eqref{eq_SHE} is the density of a one-dimensional SBM \cite{KS1988, R1989}, and the compact support property is due to Iscoe \cite{I1988}. Early work on the CSP for general $\gamma$, including the first proof for all $\gamma \in (0,1)$, which is due to Mueller and Perkins \cite{MP1992}, used techniques from superprocess theory; see also \cite{S1994}. Later, Krylov gave a proof using only stochastic calculus \cite{K1997}. The CSP and its close relative the compact interface property often play an important role in the analysis of SPDEs of reaction-diffusion type; for a recent reference see \cite{MMR2021}. Recently, the compact support property has also been shown for SPDEs with correlated noise \cite{HKY2023} and stable (non-Gaussian) noise \cite{H2025}, both using generalizations of Krylov's method. Sub-polynomial perturbations of the noise coefficient at the critical parameter $\gamma=1$ were considered in \cite{HKY2024}.



In order to situate our results, we return now to the super-Brownian case, i.e.\ $\gamma =1/2$. If the Brownian motion describing the spatial motion in the SBM is replaced with a general Lévy process, one obtains a measure-valued process called a super-Lévy process. Under certain conditions on the associated Lévy measure, the super-Lévy process has a density. For concreteness, we take the motion to be a symmetric $\alpha$-stable process with $\alpha \in (1,2)$. In this case, as shown in \cite{KS1988}, the density of a super-$\alpha$-stable process in dimension one exists and is a solution to the non-local SPDE
	\begin{equation} \label{eq_SHE_superstable}
	\partial_t u(t,x) =-(-\Delta)^{\alpha/2} u(t,x) + u(t,x)^{1/2} \dot{W}(t,x),\quad t >0, \, x \in \R.
\end{equation}
Superprocesses associated to discontinuous motions do not have the compact support property. Rather, they exhibit {\it instantaneous propagation of supports}: specialized to \eqref{eq_SHE_superstable}, this means that an integrable non-negative solution to \eqref{eq_SHE_superstable}, which is the density of a super-$\alpha$-stable process with finite mass, satisfies
\begin{equation} \label{eq_inst}
	\bP\left(\mathrm{supp}(u(t,\cdot)) = \R  \, | \, u(t,\cdot) \neq 0 \right) = 1
\end{equation}
for every $t \geq 0$, as shown in \cite{P1990} and generalized to other spatial motions in \cite{EP1991}. 

More generally, one may consider solutions to the SPDE
\begin{equation} \label{eq_SHE_stable}
	\partial_t u(t,x) = -(-\Delta)^{\alpha/2} u(t,x) + u(t,x)^\gamma \dot{W}(t,x),\quad t >0, \, x \in \R,
\end{equation}
for $\gamma >0$. For the corresponding equation with diffusion (i.e.\ continuous motion) \eqref{eq_SHE}, solutions have the compact support property for $\gamma \in (0,1)$, but by the discussion above, the compact support property fails for \eqref{eq_SHE_stable} at $\gamma =1/2$ by \eqref{eq_inst}. It is natural to ask how supports of solutions to \eqref{eq_SHE_stable} behave for other values of $\gamma \in (0,1)$, but this has not been studied before. Here, we investigate the behaviour of a family of related models for $\gamma \in (0,1/2)$, and prove that they exhibit a variant of the compact support property, indicating that $\gamma = 1/2$ is the critical parameter in this setting. One way to see that smaller values of $\gamma$ lead to stronger compact support-type behaviour is as follows: one decomposes the noise term as
\[u(t,x)^\gamma \dot{W}(t,x) = u(t,x)^{\gamma - \frac 12} \sqrt{u(t,x)} \dot{W}(t,x). \]
The term $\sqrt{u(t,x)} \dot{W}(t,x)$ can be interpreted as continuous state branching, 
and the residual coefficient represents a density-dependent branching rate which, if $\gamma < 1/2$, is large when solutions are small, and therefore there are large fluctuations which may cause the solution to touch zero. 



In this work, we adapt this question to the setting of countable systems of SDEs, which may be viewed as spatially discretized SPDEs. This allows us to initiate the study of support behaviour for equations similar to \eqref{eq_SHE_stable} while keeping technicalities at a minimum. We remark, however, that the framework developed here is general, and we expect that large parts of our argument will be applicable in the setting of SPDEs. In particular, the majority of the paper is spent proving a result which does not depend on the discreteness of the underlying space. 

Consider the system of SDEs on the $d$-dimensional integer lattice $\Z^d$ given by
\begin{equation} \label{eq_sys_intro}
	dX_t(i) = \cL X_t(i) dt + X_t(i)^\gamma dB_t(i), \quad i \in \Z^d,
\end{equation}
where $\{(B_t(i))_{t \geq 0} : i \in \Z^d\}$ is a family of independent Brownian motions and $\cL$ denotes the generator of a Markov process on $\Z^d$, that is, it generates some continuous-time random walk. Because the underlying motion is discontinuous, one expects the propagation of supports for solutions to \eqref{eq_sys_intro} to be similar to that for solutions to \eqref{eq_SHE_stable}. Indeed, for $\gamma = 1/2$, the solution is a super-random walk, and therefore exhibits instantaneous propagation of supports \cite{EP1991}. We prove that this is sharp. Our main result, Theorem~\ref{thm_main_compact}, establishes that for solutions to \eqref{eq_sys_intro} with $\gamma \in (0,1/2)$—in fact we treat a more general equation—the set of times at which the solution has compact support has full measure almost surely. At the same time, we show (in a special case) that \eqref{CSP} cannot hold, because there exists a dense set of times on which the support is unbounded. 
We conjecture that the same results hold for solutions to \eqref{eq_SHE_stable}; as we discuss in the next section, we believe that some of the methods developed in this paper should be applicable to \eqref{eq_SHE_stable} and other non-local SPDEs.

We conclude this introduction with a discussion on uniqueness of solutions. For $\gamma \in (0,1/2)$, it is an open problem to prove or disprove weak uniqueness for solutions to \eqref{eq_SHE}, \eqref{eq_sys_intro}, or the more general spatially discrete system which we introduce in the next section. Thus, it is unknown if the solutions we consider are unique in law. For $\gamma < 1/2$, the noise coefficient $\sigma(u) = u^\gamma$ is not $\frac 1 2$-Hölder at $0$, placing it below the threshold of the Yamada-Watanabe criterion for pathwise uniqueness. In the case of signed solutions, it is known that weak uniqueness fails for $\gamma \in (0,1/2)$. This is already true in the one-dimensional setting. A classical example of an SDE for which weak uniqueness fails is the Girsanov SDE
\begin{equation}\label{eq_Girsanov}
	dX_t = |X_t|^\gamma dB_t
\end{equation}
with $\gamma \in (0,1/2)$. One can construct many solutions with different laws by time-changing a Brownian motion (see Section V.26 of \cite{RogersWilliams2}); however, if one restricts to non-negative solutions, i.e.\ the equation $dX_t = X_t^\gamma dB_t$, then $0$ is an absorbing state and solutions are pathwise unique. 
The restriction to non-negative solutions removes the issue of solutions spontaneously leaving zero, which is the sole mechanism for generating a family of solutions with different laws \cite{ES1985}, and hence it is at least plausible that non-negative solutions to \eqref{eq_sys_intro} are unique in law. However, proving or disproving this is an open problem. 

In the case of SPDEs, it is proved in \cite{MMP2014} that for $\gamma \in (0,3/4)$, weak uniqueness fails for the equation
\begin{equation} \label{eq_SHE_signed}
		\partial_t u(t,x) = \frac 1 2 \Delta u(t,x) + |u(t,x)|^\gamma \dot{W}(t,x),\quad t >0, \, x \in \R.
\end{equation}
Once again, restriction to non-negative solutions re-opens the problem. Indeed, non-negative solutions to \eqref{eq_SHE_signed} (i.e.~solutions to \eqref{eq_SHE}) with $\gamma \geq 1/2$ are unique-in-law \cite{M1998}. It is less clear what should be expected for non-negative solutions when $\gamma \in (0,1/2)$. For example, it is known that pathwise uniqueness fails for \eqref{eq_SHE} with $\gamma \in (0,1/2]$ if one adds a non-negative, non-trivial immigration term \cite{BMP2010,C2015}. 


\section{Main results}
Let $d \in \N$. We will consider an infinite system of SDEs on $\Z^d$ in which the inter-site drift corresponds to the motion of a continuous time random walk. Let $\ell^1 = \ell^1(\Z^d)$ denote the space of integrable functions on $\Z^d$ and let $\ell^1_+ = \ell^1_+(\Z^d)$ denote the elements of $\ell^1$ with non-negative entries. We fix $q\in \ell^1_+$ satisfying $q(0) = 0$, 
which specifies the jump rates of a continuous-time random walk on $\Z^d$ with generator
\begin{equation} \label{def_gen}
	\mathcal{L} \phi (i) := \sum_{j \in \Z^d} q(i-j) (\phi(j) - \phi(i)).
\end{equation}
We now introduce a system of SDEs on $\Z^d$ generalizing the system \eqref{eq_sys_intro} from the Introduction. For $q$ as above, will consider solutions to 
\begin{equation} \label{e_sdesystem}
dX_t(i) =   \bigg(\sum_{j \in \Z^d} q(i-j) \left( X_t(j) - X_t(i)\right) \bigg)dt + f(X_t(i))dt +  \sigma(X_t(i)) dB_t(i) , \quad i \in \Z^d.
\end{equation}
In the above, $\{ (B_t(i))_{t \geq 0} : i \in \Z^d\}$ is a family of independent Brownian motions. The first drift term is simply the generator $\cL$ applied to $X_t$, which we express in terms of the jump rates for clarity's sake. In addition to migration of mass between sites governed by $q$, the solution is subject to on-site drift and diffusion, corresponding respectively to the functions $f$ and $\sigma$. The assumptions imposed on $q$, $f$ and $\sigma$ are discussed shortly. We always consider non-negative solutions, and as such will consider $f$ and $\sigma$ as functions on $\R_+ = [0,\infty)$.

Our chief interest in this paper is in the support properties of solutions to \eqref{e_sdesystem}, in particular when $\sigma(x) \approx x^{\gamma}$ for small values of $x$ for some $\gamma \in (0,1/2)$ (this is made precise in Assumption~\ref{assumption2}). For the compact support-type behaviour we are interested in, it is natural to consider solutions with finite mass, and indeed we prove our main result, Theorem~\ref{thm_main_compact}, for solutions to \eqref{e_sdesystem} taking values in $\ell^1_+$. In order to focus on our analysis of supports, we wish to keep our discussion on the existence of solutions brief, and hence to construct solutions with a minimum of difficulty. For this reason we prove existence of solutions taking values in a space of rapidly decaying functions (which we introduce shortly) instead of $\ell^1_+$. Proving existence of rapidly decaying solutions requires an additional assumption on the tails of the jump rates $q(\cdot)$. 


We will always assume the following on $f$ and $\sigma$.

\begin{assumption} \label{assumption1}  
	(a) $f : \R_+ \to \R$ is globally Lipschitz and satisfies $f(0) = 0$. \\
	(b) $\sigma: \R_+ \to \R_+$ is continuous, $\sigma(0) = 0$, $\sigma(x) > 0$ for $x >0$, and there exist constants $C \geq 1$ and $\theta \in (0,1)$ such that 
	\begin{equation} \label{eq_sigma_assumption1}
		\sigma(x) \leq C(x^\theta + x) \,\, \text{ for all } x \geq 0.
	\end{equation}
\end{assumption}

As noted above, our existence theorem is stated for solutions taking values in the space of rapidly decaying functions. We define
\begin{equation} \label{def_Xrap}
    \Xrap = \Xrap(\Z^d) := \left\{ f : \Z^d \to \R : \langle f, \Phi_\lambda \rangle < \infty \,\,\, \forall \, \lambda > 0 \right\}
\end{equation}
where $\Phi_\lambda(i) := e^{\lambda |i|}$ for $\lambda \in \R$, $|\cdot|$ denotes some fixed norm on $\Z^d$, and $\langle f,g\rangle := \sum_{j \in \Z^d} |f(i) g(i)|$. We denote by $\Xrap^+$ the elements of $\Xrap$ with non-negative entries. Note that $\Xrap$ embeds continuously into $\ell^1$ by inclusion, so that continuous $\Xrap^+$-valued solutions to \eqref{e_sdesystem} are also continuous $\ell^1_+$-valued solutions.


In order to construct solutions taking values in $\Xrap^+$, we impose the following conditions on $q(\cdot)$.

\begin{assumption} \label{assumption1.5} (a) $q$ is symmetric, i.e.\ $q(i) = q(-i)$ for all $i \in \Z^d$. \\
(b) For all $\lambda >0$, $\langle q , \Phi_\lambda \rangle < \infty$.
\end{assumption}
Note that part (b) implies that the random walk with jump rates $q(\cdot)$ has exponential moments of all orders.

\begin{proposition} \label{prop_exist}
    Under Assumptions~\ref{assumption1} and~\ref{assumption1.5}, for any $X_0 \in \Xrap^+$, there exists a continuous $\,\Xrap^+$-valued solution to \eqref{e_sdesystem} started from $X_0$.
\end{proposition}
A more thorough discussion of the existence of solutions and some elementary properties is given in Section~\ref{s_sol}, see in particular Theorem~\ref{thm_sol}.


Our main results concern solutions to \eqref{e_sdesystem} when $\sigma(x)$ behaves like $x^\gamma$ near zero for some $\gamma \in (0,1/2)$. This is encoded in the following assumption.

\begin{assumption} \label{assumption2} For some $\gamma \in (0,1/2)$,
	\begin{equation} \label{eq_sigma_assumption2}
		\liminf_{x \downarrow 0} \frac{\sigma(x)}{x^\gamma} >0.
	\end{equation}
\end{assumption}

Note that for the central example of interest $\sigma(x) = x^\gamma$ with $\gamma \in(0,1/2)$, \eqref{eq_sigma_assumption1} and \eqref{eq_sigma_assumption2} are satisfied (with $\theta = \gamma$). Our main result is that under Assumption~\ref{assumption2}, solutions to \eqref{e_sdesystem} with compactly supported initial data have compact support for almost all times. Let $\mathrm{Leb}(A)$ denote the $1$-dimensional Lebesgue measure of $A \subseteq \R$. 

\begin{theorem} \label{thm_main_compact}
	Suppose Assumptions~\ref{assumption1} and \ref{assumption2} hold and let $X = (X_t)_{t \geq 0}$ be a continuous $\ell^1_+$-valued solution to \eqref{e_sdesystem} with compactly supported initial state. Then for every $T>0$, with probability one, 
	\begin{equation*}
		\mathrm{Leb}\left( \left\{ t \in [0,T] : \mathrm{supp}(X_t) \text{ is compact}\, \right\} \right) = T.
	\end{equation*}
\end{theorem}

This is the first compact support-type result for solutions to an infinite-dimensional SDE or a SPDE in which the underlying spatial motion is discontinuous. Indeed, for lattice SDEs such as \eqref{e_sdesystem} the discrete spatial domain means that any spatial motion discontinuous, which makes this, so far as we are aware, the first compact support result for an infinite-dimensional SDE. In general, Theorem~\ref{thm_main_compact} cannot be improved to state that the support is compact for all times; we prove in Theorem~\ref{thm_unbounded} that (in a special case) there exist exceptional times when the support is unbounded, which distinguishes the support behaviour of solutions to \eqref{e_sdesystem} under Assumption~\ref{assumption2} from that of solutions to the stochastic heat equation \eqref{eq_SHE} with $\gamma \in (0,1)$. 



As discussed in the Introduction, Theorem~\ref{thm_main_compact} should be contrasted with the case of the corresponding superprocess, or the super-random walk, which exhibits instantaneous propagation of supports. We recall that one obtains a superprocess from \eqref{e_sdesystem} by setting $\sigma(x) = \sqrt{x}$, which corresponds to $\gamma = 1/2$, and $f \equiv a$ for $a \in \R$; for general Lipschitz $f$ the superprocess with drift $f$ can be realized and studied via Dawson's Girsanov theorem \cite{D1978} (see also \cite[Theorem~IV.1.6]{Perkins}). In either case, we have \[\bP(X_t(i) >0 \,\, \forall \, i \in \Z^d \, | \,  X_t \neq 0) = 1\] for every $t \geq 0$ \cite[Theorem~5.1]{EP1991}. (This assumes the irreducibility of the underlying random walk kernel; without this assumption, the natural analogous statements hold.) In particular, until the extinction time, the set of times at which the solution has full support is a.s.\ a set of full measure. Theorem~\ref{thm_main_compact} proves that this is essentially sharp, because setting $\sigma(x) = x^\gamma$ for any $\gamma  \in (0,1/2)$ leads to the converse: the set of times where the support is compact is a.s.\ a set of full measure.

As noted above, despite the compact support-type behaviour established in Theorem~\ref{thm_main_compact}, support propagation in the models we consider is qualitatively different from the classical compact support property. In particular, as the next result shows, for a solution to \eqref{e_sdesystem} there a.s.\ exists a dense set of exceptional times at which the support is unbounded. Hence, the full measure set of times on which $X_t$ has compact support cannot be improved to all times. For parabolic SPDEs with the compact support property, e.g. \eqref{eq_SHE} with $\gamma \in (0,1)$, such exceptional times do not exist, because the classical compact support property (i.e.~\eqref{CSP}) establishes the boundedness of the union of the supports of the solution at all times $t \in [0,T]$. To state our result on exceptional times, we define the extinction time

\begin{equation} \label{def_extinction}
	\eta := \inf\{t \geq 0: X_t(i) = 0 \,\, \forall i\in \Z^d \}.
\end{equation}
We prove the following under simplified conditions to shorten the proof, but there is little doubt that it also holds in the full generality of Assumption~\ref{assumption1}.
\begin{theorem} \label{thm_unbounded}
	Let $\cL = \Delta$, i.e.\ $q(\cdot)$ corresponds to the jump rates of a nearest-neighbour random walk, $f \equiv 0$, and suppose $\sigma$ satisfies Assumption~\ref{assumption1}(b). Let $X$ be a continuous $\ell^1_+$-valued solution to \eqref{e_sdesystem} with non-zero initial condition $X_0$. Then with probability one, $\eta>0$ and $\{t \geq 0 : \mathrm{supp}(X_t) \text{ is unbounded} \}$ is a dense, uncountable subset of $[0,\eta]$.
\end{theorem}

It is worth pointing out here that for superprocesses with discontinuous motions, which exhibit instantaneous propagation of supports, exceptional times (at which the support \textit{is} compact) were recently studied by Hong and Mytnik \cite{HM2025}, who showed that the set of such times has Hausdorff dimension one for certain super-stable processes. Thus, compactness and unboundedness of the support play opposite roles depending on if one has a superprocess (with discontinuous motion) or a solution to \eqref{e_sdesystem}. For the superprocesses, there is unbounded support for a.a.\ times and compact support at exceptional times; for solutions to \eqref{e_sdesystem}, these properties are switched. In our setting, it is a natural open problem to consider the Hausdorff dimension of the set of exceptional times at which the support is unbounded, and to investigate whether or not it depends on the tails $q$. 



We now summarize our proof strategy for Theorem~\ref{thm_main_compact}. 
The method of~\cite{K1997} does not work in this setting, as it replies both on continuous space and crucially on the continuity of the spatial motion. Instead, we rely on analyzing the evolution of mass on a half-space via a martingale decomposition together with a condition on the quadratic variation.

For simplicity, for this discussion we restrict to $d=1$, and we consider the simplified equation \eqref{eq_sys_intro} with $\cL = \Delta$, which is the same as \eqref{e_sdesystem} with $\cL = \Delta$, $f\equiv 0$ and $\sigma(x) = x^\gamma$ for some $\gamma \in (0,1/2)$. In dimension one, we recall that $\Delta g(i) = \tfrac 12(g(i+1) + g(i-1)) - g(i)$. Suppose that $(X_t)_{t \geq 0}$ is an $\ell^1_+$-valued solution under these assumptions. For $i \in \Z$, for $t \geq 0$ we define
\begin{equation}\label{def_Yti_intro}
	Y_t(i) := \sum_{j = i}^\infty X_t(j). 
\end{equation}
That is, $Y_t(i)$ is the amount of mass of $X_t$ to the right of $i$. In particular, $Y_t(i) \geq 0$, and $Y_t(i) = 0$ if and only if $X_t(j) = 0$ for all $j \geq i$; hence $X_t$ has support bounded to the right whenever $Y_t(i) = 0$, and so in order to prove our result it suffices to show that, for sufficiently large values of $i$, $Y_t(i) = 0$ for most $t$ (in the sense of Lebesgue measure). 

$(Y_t(i))_{t \geq 0}$ evolves according to
\begin{equation} \label{eq_Ytdyn_intro}
	dY_t(i) = \frac 12 \left(X_t(i-1) - X_t(i)\right)\, dt + \sum_{j = i}^\infty X_t(j)^\gamma \, dB_t(j).
\end{equation}
A priori this is only formal, but we dispense with technicalities for this discussion. Clearly $Y_t(i)$ is not Markovian with respect to its own filtration. Moreover, it cannot be easily coupled with a Markov process or diffusion. The key observation concerns its quadratic variation. Considering the martingale term in \eqref{def_Yti_intro}, we observe that since $2\gamma < 1$,
\begin{equation*}
    \sum_{j=i}^\infty X_t(j)^{2\gamma} \geq \bigg(\sum_{j=i}^\infty X_t(j)\bigg)^{2\gamma},
\end{equation*}
which implies that the quadratic variation of $Y_t(i)$ satisfies
\begin{equation} \label{eq_QVbd_informal}
    d \langle Y(i) \rangle_t \geq \frac 1 2 Y_t(i)^{2\gamma} dt.
\end{equation}
Thus, while the dynamics of $(Y_t(i))_{t \geq 0}$ are not autonomous (i.e.\ it is not Markov), its quadratic variation satisfies an autonomous lower bound. Based on the heuristic that larger martingale fluctuations increase the chances of a solution hitting zero, we should therefore expect the zero set of $(Y_t(i))_{t\geq 0}$ to be at least as large as if it had quadratic variation equal to $\tfrac 1 2 \int_0^t Y_s(i)^{2\gamma} ds$. Supposing, for a sufficiently large value of $i$, that the drift in \eqref{def_Yti_intro} can be approximated or bounded above by $\delta>0$, we may then expect that the half-space process $(Y_t(i))_{t \geq 0}$ has similar behaviour to a solution of the following SDE:
\begin{equation} \label{eq_Zsde}
    dZ_t = \delta dt + Z_t^\gamma dB_t.
\end{equation}
An understanding of the behaviour near zero of solutions to \eqref{eq_Zsde} for small $\delta>0$ and $\gamma \in (0,1/2)$ underlies our proof. However, the theory of diffusions and Markov processes is not directly applicable. Instead, we introduce a general class of (not necessarily Markovian) semimartingales, which we call $\sM_{\gamma,\delta}$, whose behaviour generalizes that of \eqref{eq_Zsde}. Ultimately, we are able to study the excursion structure of such processes, using certain uniform estimates which hold for all elements of $\sM_{\gamma,\delta}$ combined with a distributional shift-invariance property we call the ``pseudo-strong Markov property''. In Theorem~\ref{thm_Y_zeroset}, we prove that processes in $\sM_{\gamma,\delta}$ spend Lebesgue-positive time at zero by studying their local time, which we analyze via the aforementioned excursion analysis. The proof of Theorem~\ref{thm_main_compact} is concluded by constructing monotone couplings of $(Y_t(i))_{t \geq 0}$, for large values of $i$, with elements of $\sM_{\gamma,\delta}$ for small $\delta$.

Although we do not actually use the SDE \eqref{eq_Zsde} to prove our main result, our results on class-$\sM_{\gamma,\delta}$ processes have implications for its solutions. Despite \eqref{eq_Zsde} being quite an elementary equation, the following is apparently a novel result.
\begin{theorem} \label{thm_sde}
    Let $\gamma \in (0,1/2)$ and $\delta>0$, and suppose that $(Z_t)_{t  \geq 0}$ is a (non-negative) solution to \eqref{eq_Zsde} with $Z_0 = 0$.     Then $\mathrm{Leb}(\{t \in [0,T] : Z_t = 0\}) > 0$ a.s for any $T>0$.
\end{theorem}
One observes easily that it is impossible to have $Z_t = 0$ for all $t \in [a,b]$ for any $a<b$. Thus, the zero set of $(Z_t)_{t \geq 0}$ has an interesting topological character: it is a (random) fat Cantor set, i.e.\ a nowhere dense set of positive Lebesgue measure. Another continuous process whose zero set is a fat Cantor set is the sticky Brownian motion. (For a treatment of this process via SDEs, see \cite{EP2014}.) However, sticky Brownian motion is virtually engineered to have this property; it is more surprising that it should occur in a setting with no intrinsic stickiness, and we believe that this behaviour has not previously been observed for non-sticky SDEs. We do not explore it here in detail here, but our methods, in particular the arguments in Section~\ref{s_zeros}, can also be used obtain quantitative information on the Lebesgue measure of $\{t \in [0,T] : Z_t>0\}$ when $\delta T \ll 1$. The behaviour of the zero set of solutions to \eqref{eq_Zsde} for $\gamma \in (0,1/2)$ can also be compared with the case $\gamma =1/2$. For $\gamma = 1/2$, the solution is a (time-changed) squared-Bessel process of dimension $\delta / 4$. If $\delta \geq 8$, $\{0\}$ is a polar set. For $\delta \in (0,8)$, $0$ is reflecting, and the Hausdorff dimension of the zero set equals $1 - \tfrac{\delta}{8}$, which is a consequence of the identification of the inverse local time (at $0$) as a stable subordinator. 



We also remark, pursuant to our discussion of uniqueness in the Introduction, that as far as we are aware it is unknown if weak uniqueness holds for \eqref{eq_Zsde} when $\gamma \in (0,1/2)$. As before, we observe that non-negativity removes the known pathologies and so it is plausible that weak uniqueness could hold. In \cite{BBC2007}, Bass, Burdzy and Chen proved pathwise uniqueness of solutions to Girsanov's SDE \eqref{eq_Girsanov} among solutions which spend Lebesgue-null time at zero. While this fails here (evidently!), it remains a possibility that the ``chasing phenomenon'' they use to prove pathwise uniqueness could be shown for solutions to \eqref{eq_Zsde}. 

We prove Theorem~\ref{thm_sde} as a consequence of our more general analysis of class-$\sM_{\gamma,\delta}$ semimartingales; however, if it were shown that weak uniqueness held for \eqref{eq_Zsde}, then solutions would be Markov, and one could prove Theorem~\ref{thm_sde} quite easily using the speed measure and scale function which would therefore exist.

We conclude by discussing the applicability of our method to related problems. Indeed, as discussed in the Introduction, the original motivation for this work was an understanding of the supports of solutions to \eqref{eq_SHE_stable}, and other non-local SPDEs, with $\gamma \in (0,1)$. We studied \eqref{e_sdesystem} instead, in part to avoid technicalities which have no bearing on the main methodological contributions of our work. However, we have developed these new methods in a generality which makes them easily applied in other settings. In particular, the proof of Theorem~\ref{thm_main_compact} consists of a short ``pre-processing'' argument which allows one to reduce the problem to a statement about class-$\sM_{\gamma,\delta}$ semimartingales, and then a general result (Theorem~\ref{thm_Y_zeroset}) on class-$\sM_{\gamma,\delta}$ semimartingales. Since the latter result does not depend on the underlying model, in order to prove a version of Theorem~\ref{thm_main_compact} for solutions to \eqref{eq_SHE_stable}, one simply needs to give the pre-processing argument. This argument requires uniform-in-time bounds of certain integral quantities, which are then used to construct couplings which are useful with high probability. These bounds are the only other ingredient needed to extend our result to other models. 
Thus, in addition to proving Theorem~\ref{thm_main_compact} for a particular class of models, in this work we have developed a general machinery which can be used to prove similar results for different models.

The rest of the paper is organized as follows:
\begin{itemize}
    \item Section~\ref{s_prel} discusses some preliminaries, including the existence of solutions to \eqref{e_sdesystem} and the construction of couplings with a higher-dimensional version of the process from \eqref{def_Yti_intro}.
    \item Section~\ref{s_M} introduces a family of stochastic processes called $\sM_{\gamma,\delta}$ and establishes some of their elementary properties, notably the pseudo-strong Markov property.
    \item Section~\ref{s_intervals} proves some key results concerning the behaviour of class-$\sM_{\gamma,\delta}$ processes near zero.
    \item Section~\ref{s_zeros} contains an analysis of the excursions of class-$\sM_{\gamma,\delta}$ processes from zero and uses it to prove a general result on their zero sets, Theorem~\ref{thm_Y_zeroset}.
    \item Section~\ref{s_proofs} combines the coupling arguments from Section~\ref{s_prel} with Theorem~\ref{thm_Y_zeroset} to prove Theorem~\ref{thm_main_compact}. It is also where we prove Theorem~\ref{thm_unbounded}.
\end{itemize}

\section{Preliminaries} \label{s_prel}

\subsection{Existence of solutions} \label{s_sol}
We recall the random walk jump rates $q_{ij} = q(j-i)$ and the corresponding generator $\cL$ from \eqref{def_gen}. Let $(P_t)_{t \geq 0}$ denote the corresponding semigroup, i.e.\ $P_t = e^{\cL t}$ for $t \geq 0$. We remark that the symmetry condition on $q(\cdot)$ from Assumption~\ref{assumption1.5}(a) implies that $\cL$ and $P_t$ are self-adjoint (for all $t >0$). This fact will be used implicitly in the sequel.

Let $(\Omega, \sF, (\sF_t)_{t\geq 0}, \bP)$ be a filtered probability space on which $\{(B_t(i))_{t \geq 0} : i \in \Z^d\}$ is a family of independent Brownian motions adapted to $(\sF_t)_{t \geq 0}$. An $\R^{\Z^d}_+$-valued process $(X_t)_{t \geq 0}$ is a (non-negative) solution to \eqref{e_sdesystem} if
\begin{enumerate}[label=(\roman*)]
    \item For each $i \in \Z^d$ and $t \geq 0$, $\sum_{j \in \Z^d}q(i-j)|X_t(j) - X_t(i)| < \infty$ a.s.
    \item For each $i \in \Z^d$, $(X_t(i))_{t \geq 0}$ is a continuous, non-negative, $(\sF_t)$-adapted process a.s.\ satisfying 
    \begin{equation} \label{eq_sdei}
	X_t(i) = X_0(i) + \int_0^t \mathcal{L} X_s(i) ds + \int_0^t f(X_s(i))ds +  \int_0^t \sigma(X_s(i)) dB_s(i), \quad t \geq 0.
\end{equation}
\end{enumerate}


Under Assumption~\ref{assumption1}(a), sub-exponential growth of $X_t$ is sufficient to guarantee the finiteness of $\cL X_t$. In addition to the space $\Xrap$ of rapidly decaying functions, see \eqref{def_Xrap}, we introduce the space of tempered functions
\begin{equation*}
    \Xtemp = \Xtemp(\Z^d) := \left\{ f : \Z^d \to \R : \langle f, \Phi_\lambda \rangle < \infty \,\,\, \forall \, \lambda < 0 \right\},
\end{equation*}
where we recall that $\Phi_\lambda(i) = e^{\lambda|i|}$, and denote its non-negative elements by $\Xtemp^+$. The following theorem gives the existence of solutions to \eqref{e_sdesystem} in $\Xrap^+$ and $\Xtemp^+$, the mild form representation, and a first moment bound. 

\begin{theorem} \label{thm_sol} Suppose Assumptions~\ref{assumption1} and~\ref{assumption1.5} hold. 

(a) For any $X_0 \in \Xtemp^+$, there exists a probability space supporting a continuous $\Xtemp^+$-valued solution $X = (X_t)_{t \geq 0}$ with initial state $X_0$.

(b) For a continuous $\Xtemp^+$-valued solution $X$, for any $\phi \in \Xrap$, with probability one, for all $t \geq 0$,
\begin{align} \label{eq_weak}
\langle X_t, \phi \rangle  &= \langle X_0 , \phi\rangle + \int_0^t (\langle X_s, \mathcal{L} \phi \rangle + \langle f(X_s), \phi \rangle ) ds   +  \int_0^t \sum_{j \in \Z^d} \sigma(X_s(j)) \phi(j) dB_s(j).
\end{align}
Furthermore, for all $(t,i) \in \R_+ \times \Z^d$, $X$ a.s.\ satisfies the mild form
	\begin{equation} \label{eq_mild}
		X_t(i) =  P_t X_0(i)  + \int_0^t  P_{t-s} (f(X_s))(i) ds +\int_0^t  \sum_{j \in \Z^d}    P_{t-s}(i-j) \sigma(X_s(j)) dB_s(j),
	\end{equation}
and
	\begin{equation} \label{eq_firstmoment}
		\E[X_t(i)] \leq e^{Lt} P_t X_0(i),
	\end{equation}
where $L\geq 0$ is the Lipschitz constant of $f$.

(c) If $X$ is a continuous $\Xtemp^+$-valued solution and $X_0 \in \Xrap^+$, then $X$ is a continuous $\Xrap^+$-valued solution. In this case, \eqref{eq_weak} holds for all $\phi \in \Xtemp$.
    
\end{theorem}
We remark that Proposition~\ref{prop_exist} is a corollary of parts (a) and (c) of the above. We will not give a full proof of Theorem~\ref{thm_sol}, but we give a brief discussion of the literature and a sketch of the proof. Tempered solutions to infinite-dimensional SDEs, in relatively similar settings, have been constructed in \cite{SS1980} and \cite{DP1998}. In fact, the closest analogue of our situation is in Shiga \cite{S1994}, which covers the exact same assumptions, but for the stochastic heat equation (with spatial operator $\Delta$). To prove Theorem~\ref{thm_sol}, one can adapt the proof in \cite{S1994} to the discrete-space setting with very few modifications. In order to construct solutions in tempered spaces, one requires sufficiently good integrability properties of the random walk heat kernel. In our setting, this is guaranteed by Assumption~\ref{assumption1.5}; see Lemma~2.1 in \cite{DP1998}, where tempered solutions are similarly constructed. 

Given a solution in the sense of \eqref{eq_sdei}, a straightforward limiting argument gives the more general weak form \eqref{eq_weak}. Within $\Xtemp^+$, the equivalence of (analytically) weak solutions and mild solutions (i.e.\ those satisfying \eqref{eq_weak} and \eqref{eq_mild}, respectively) follows from a standard argument. To construct either, the usual approach is to first assume that $\sigma$ is globally Lipschitz. In this case, one can construct a $\Xtemp^+$-valued mild solution via a Picard iteration scheme as in \cite{S1994}, or an analytically weak $\Xtemp^+$-valued solution by approximating $\Z^d$ with finite domains as in \cite{SS1980, DP1998}. The equivalence of weak and mild solutions allows one to establish \eqref{eq_weak} and \eqref{eq_mild}. For non-Lipschitz $\sigma$ satisfying \eqref{eq_sigma_assumption1}, one takes a family $X^n$ of $\Xtemp^+$-valued solutions with Lipschitz noise coefficients $\sigma_n$ such that $\sigma_n \to \sigma$ uniformly, and proves tightness of $(X^n : n \in \N)$; any limit point is a solution to \eqref{e_sdesystem} with noise coefficient $\sigma$. 
The first moment bound \eqref{eq_firstmoment} can be proved using the mild form \eqref{eq_mild} and Itô's lemma. Finally, to prove that an $\Xtemp^+$-valued solution with $X_0 \in \Xrap^+$ is actually $\Xrap^+$-valued, we argue as in \cite[Theorem~2.2(c)]{DP1998} or \cite[Theorem~2.5]{S1994}.

\subsection{The half-space process}
In this section, we obtain a stochastic integral representation for the mass of a solution on a half-space, which is the key quantity analyzed in the proof of Theorem~\ref{thm_main_compact}. We express $j \in \Z^d$ in coordinates as $j = (j_1,\dots,j_d)$. For a solution $X$ to \eqref{e_sdesystem} on~$\Z^d$, for $i \in \Z$ and $t \geq 0$ we define
\begin{equation}
	Y_t(i) := \sum_{j \in \Z^d \, :\, j_1 \geq i} X_t(j).
\end{equation}
This process generalizes to arbitrary dimensions the half-line process \eqref{def_Yti_intro} for which the heuristics were discussed in the Introduction. Let $h \in \R^{\Z^d}$ denote the half-space indicator function 
\begin{equation} \label{eq_heaviside}
	h(j) = \begin{cases}
		1, & j_1 \geq 0, \\ 0, & j_1<0.
	\end{cases}
\end{equation}
For $i \in \Z$ we set $\underline{i} := (i,0,\dots,0) \in \Z^d$, and we remark that $Y_t(i) = \langle X_t , h(\cdot - \underline{i})\rangle$. 

From this point forward, we consider $\ell^1_+$-valued solutions instead of $\Xrap^+$-valued solutions, and we dispense with Assumption~\ref{assumption1.5}. The loss of Assumption~\ref{assumption1.5}(b) has essentially no consequences in the sequel; the loss of symmetry from Assumption~\ref{assumption1.5}(a) means that $\cL$ and $P_t$ are no longer self-adjoint. Going forward, we denote their adjoints by $\cL^*$ and $P_t^*$ and remark that both are bounded operators on $\ell^\infty(\Z^d)$.



In order to obtain a semimartingale representation for $(Y_t(i))_{t\geq 0}$, we essentially use the weak form of \eqref{e_sdesystem}, i.e.~\eqref{eq_weak}, with $\phi = h(\cdot- \underline{i})$. Indeed, for an $\Xrap^+$-valued solution $X$ under Assumptions~\ref{assumption1} and~\ref{assumption1.5}, this is precisely what we obtain from \eqref{eq_weak} in Theorem~\ref{thm_sol} using $\phi = h(\cdot -\underline{i}) \in \Xtemp$. In order to obtain the representation for the more general $\ell^1_+$-valued solutions for which we state Theorem~\ref{thm_main_compact}, we need to prove that the weak form holds for bounded functions in this setting. This is straightforward, as bounded functions are the natural dual pairing for $\ell^1_+$-valued solutions, but we give the proof for completeness' sake. We remark that in the absence of the symmetry which was previously provided by Assumption~\ref{assumption1.5}(a), the adjoint operator $\cL^*$ now appears in the weak form instead of $\cL$, which is the single way in which the formula in the sequel differs from \eqref{eq_weak}.

\begin{lemma} \label{lemma_bddweak}
	 Suppose Assumption~\ref{assumption1} holds. 
    Let $i \in \Z$ and suppose that $X = (X_t)_{t\geq 0}$ is a continuous $\ell^1_+$-valued solution to \eqref{e_sdesystem}. Then for any bounded $\phi : \Z^d \to \R$, $t\mapsto \langle X_t, \phi \rangle$ is a continuous semimartingale with representation
	\begin{equation} \label{eq_bddweak}
		\langle X_t,\phi\rangle  = \langle X_0,\phi\rangle + \int_0^t (\langle X_s ,  \cL^* \phi\rangle + \langle f(X_s) ,  \phi \rangle ) ds + \int_0^t \sum_{j \in \Z^d}\sigma(X_s(j)) \phi(j) dB_s(j), \quad t \geq 0.
	\end{equation}
\end{lemma}

Taking $\phi = h(\cdot - \underline{i})$ in Lemma~\ref{lemma_bddweak}, we obtain the following.
\begin{corollary}  \label{cor_Ystochastic}
	Suppose Assumption~\ref{assumption1} holds. 
    Let $i \in \Z$ and suppose that $X = (X_t)_{t\geq 0}$ is a continuous $\ell^1_+$-valued solution to \eqref{e_sdesystem}. Then for every $i \in \Z$, $t \mapsto Y_t(i)$ is a continuous, non-negative semimartingale with representation
	\begin{equation} \label{eq_Yibp}
		Y_t(i) = Y_0(i) + \int_0^t (\langle X_s ,  \cL^* h(\cdot - \underline{i})\rangle + \langle f(X_s) ,  h(\cdot - \underline{i})\rangle )ds + \int_0^t \sum_{j \in \Z^d \, :\, j_1 \geq i}\sigma(X_s(j))dB_s(j), \,\, t \geq 0.
	\end{equation}
\end{corollary}

\begin{proof}[Proof of Lemma~\ref{lemma_bddweak}]
Let $X = (X_t)_{t \geq 0}$ be a continuous $\ell^1_+$-valued solution to \eqref{e_sdesystem} and let $\phi : \Z^d \to \R$ be bounded. We remark that it suffices to prove that \eqref{eq_Yibp} holds a.s.\ on $[0,T]$ for any $T>0$. We fix $T>0$. For $n \in \N$, let $\Lambda_n := [-n,n]^d \cap \Z^d$ and $\phi_n := \indc_{\Lambda n} \phi$. Note that $\phi_n \to \phi$ point-wise as $n \to \infty$. Summing \eqref{eq_sdei} over $i \in \Lambda_n$ 
	\begin{equation} \label{eq_Yibp_prelim}
		\langle X_t , \phi_n\rangle  = \langle X_0, \phi_n \rangle + D^n_t + M^n_t, \quad t \geq 0,
	\end{equation}
	where \begin{equation*}
		D^n_t := \int_0^t (\langle \cL X_s ,  \phi_n \rangle  + \langle f(X_s), \phi_n\rangle) ds \quad \text{ and } \quad M^n_t := \int_0^t \sum_{j \in \Lambda_n} \sigma(X_s(j)) \phi(j) dB_s(j).
	\end{equation*}
    We note that the first term in the integrand of $D^n_t$ is equal to $\langle X_s, \cL^* \phi_n \rangle$.
    Next, we observe that
    \begin{equation} \label{eq_phi_uniform}
        m := \sup_{n \in \N} \|\phi_n\|_{\infty} + \|\cL \phi_n\|_{\infty} < \infty.
    \end{equation}  
    Since $t \mapsto X_t$ is continuous in $\ell^1_+$, this implies by dominated convergence that $\lim_{n \to \infty} \langle X_t,\phi_n\rangle = \langle X_t,\phi \rangle$ for all $t \in [0,T]$ with probability one. 
    
    We remark that since $X$ is continuous in $\ell^1$, $V_T := \sup_{t \in [0,T]} \|X_t\|_1 < \infty$ a.s.\ for every $T>0$. Let $L$ denote the Lipschitz constant of $f$. Since $f(0) = 0$, it follows that $f : \ell^1 \to \ell^1$ (defined by applying $f$ component-wise) is a bounded operator with norm at most $L$, and in particular $\sup_{t \in [0,T]} \|f(X_t)\|_1 \leq L V_T$. In particular, by \eqref{eq_phi_uniform} we have
    \begin{align} \label{eq_ibp_uniform}
        \sup_{n \in \N} \sup_{t \in [0,T]} \left( |\langle X_t , \phi_n\rangle |+ | D^n_t| \right) \leq V_T + m(1+L)TV_T < \infty \,\text{ a.s.}
    \end{align}
    From \eqref{eq_Yibp_prelim}, this implies that 
    \[\sup_{n \in \N} \sup_{t \in [0,T]} | M^n_t| < \infty.\]
    In particular, the maximal processes of the local martingales $t\mapsto M^n_t$ are uniformly bounded in $n$. A short argument using the Burkholder-Davis-Gundy inequality then proves convergence of the local martingales $M^n$. Indeed, $\sup_{t\in[0,T]} |M^n_t - M_t| \to 0$ in probability, where
    \[ M_t := \int_0^t \sum_{j \in \Z^d} \sigma(X_s(j)) \phi(j) dB_s(j).\]
    Finally, in order to take the limit of $D^n_t$, using $V_T < \infty$ a.s.\ and \eqref{eq_phi_uniform} we observe that
    \[\sup_{n \in \N} \sup_{t \in [0,T]} |\langle X_t , \cL \phi_n \rangle|  + |\langle f(X_t), \phi_n\rangle| < \infty \, \text{ a.s.} \]
    Thus, by dominated convergence it suffices to show that $\langle f(X_t), \phi_n\rangle \to \langle f(X_t),\phi \rangle$ and $\langle X_t , \cL^* \phi_n\rangle \to \langle X_t, \cL^* \phi \rangle$
    as $n \to \infty$ for every $t \in [0,T]$. Both claims follow by dominated convergence using \eqref{eq_phi_uniform}, in the former case using $f(X_t) \in \ell^1$ as well. The proof is complete.
\end{proof}

We end the section with a final remark on $\ell^1_+$-valued solutions. 

\begin{remark} \label{remark_bddmild}
    We previously proved the mild form \eqref{eq_mild} and first moment bound \eqref{eq_firstmoment} for $\Xtemp^+$-valued solutions under Assumptions~\ref{assumption1} and~\ref{assumption1.5}. These also hold if $X$ is a continuous $\ell^1_+$-valued solution under Assumption~\ref{assumption1}. One first generalizes Lemma~\ref{lemma_bddweak} to time-dependent functions $\phi \in C^1([0,T],\ell^\infty(\Z^d))$ using a routine limiting argument, and \eqref{eq_mild} follows by choosing $\phi(t,j) = P_{T-t}^* \delta_i(j)$, where $\delta_i$ denotes the Kronecker delta function at $i$. The first moment bound \eqref{eq_firstmoment} then follows from a short argument using Itô's lemma.
\end{remark}

\subsection{A one-dimensional coupling result}
We have given a representation for $(Y_t(i))_{t \geq 0}$ in Corollary~\ref{cor_Ystochastic}. The purpose of this section is to construct a coupling which will allow us to bound $Y_t(i)$ above, until a certain stopping time, by a process which has constant drift, while maintaining an important condition on the quadratic variation. The process with which $(Y_t(i))_{t \geq 0}$ is coupled is an element of $\sM_{\gamma,\delta}$, the class of semimartingales introduced in the next section. In this section, we continue to work with $\ell^1_+$-valued solutions without assuming Assumption~\ref{assumption1.5}.

For $i \in \Z$ and $t \geq 0$, we define
\begin{equation} \label{def_dti}
d_t(i) = \langle X_t, \cL^* h (\cdot - \underline{i})\rangle + \langle f(X_t), h(\cdot - \underline{i})\rangle
\end{equation}
and
\[ D_t(i) := \int_0^t d_s(i) ds,\]
and note that $D_t(i)$ is the drift term from \eqref{eq_Yibp}.
For $\delta > 0$ and $i \in \Z$, define
\begin{equation} \label{def_T}
	T^i_\delta = \inf\{ t \geq 0 : d_t(i) > \delta\}.
\end{equation}

\begin{lemma} \label{lemma_Delta} Suppose Assumption~\ref{assumption1} holds and that $X = (X_t)_{t\geq0}$ is a continuous $\ell^1_+$-valued solution to \eqref{e_sdesystem}. For $i \in \Z$ and $\delta >0$, there exists a filtered probability space $(\Omega,\sF,(\sF_t)_{t\geq 0},\bP)$ supporting a copy of $X$ and an $(\sF_t)$-adapted pair $(W,\Delta) = ((W_t,\Delta_t))_{t \geq 0}$, where $W$ is an independent Brownian motion and $\Delta$ is a non-negative solution to
		\begin{equation} \label{eq_Delta}
		d\Delta_t = (\delta - d_t(i))_+ dt + \Delta_t^\gamma dW_t
	\end{equation}
with $\Delta_0 = 0$.
\end{lemma}
\begin{proof}
	Without loss of generality we take $i = 0$ and we write $d_t = d_t(i)$ throughout the proof. Let $\delta>0$, and let $(\Omega,\sF, (\sF_t)_{t \geq 0},\bP)$  be a probability space on which $X$ is a continuous $\ell^1_+$-valued solution to \eqref{e_sdesystem}, and let $(\Omega',\sF',\bP')$ be another probability space on which $W = (W_t)_{t \geq 0}$ is a standard Brownian motion. For $n \in \N$, let $\sigma_n : \R_+ \to \R_+$ be a family of globally Lipschitz functions such that such that $\sigma_n(u) \to u^\gamma$ uniformly in $u \geq 0$, which satisfy $\sigma_n(u) \leq u^\gamma + u$ for all $u \geq 0$ and $\sigma_n(0) = 0$ for all $n \in \N$. Then for almost every $\omega \in \Omega$, $b_t(\omega) = (\delta - d_t(\omega))_+$ is continuous, so because $\sigma_n$ is Lipschitz, there exists a unique (non-negative) strong solution $\Delta^{n,\omega}$ to the SDE
	\[\Delta^{n,\omega}_t = \int_0^t b_s(\omega) ds + \int_0^t \sigma_n(\Delta_s)  dW_s.\]
	adapted to $(\sF^W_t)_{t \geq 0}$, the filtration of $W$. However, since $b_t \in \sF_t$, then working on the product space $\Omega \times \Omega'$, if $\Delta^n$ is defined by $\Delta^n(\omega,\omega') = \Delta^{n,\omega}(\omega')$, $\Delta^n$ is a solution to \[\Delta^{n}_t = \int_0^t b_s ds + \int_0^t \sigma_n(\Delta_s)  dW_s\]
	and is adapted to the product filtration, i.e.\ $\Delta^n_t \in \sF_t \vee \sF^W_t$.
	
	It is routine to then verify the tightness of the sequence $((X,\Delta^n))_{n \in \N}$, and hence there exists a limit point $(X,\Delta)$ of $((X,\Delta^n))_{n \in \N}$. Because the first coordinate is the same solution $X$ for every $n \in \N$, clearly any limit point has first coordinate $X$.  Since $\sigma_n \to \sigma$ uniformly, another routine argument shows that $\Delta$ must solve the martingale problem associated to \eqref{eq_Delta}, and hence is a solution to \eqref{eq_Delta} for some Brownian motion. This completes the proof.
\end{proof}

\begin{remark} \label{remark_sigmaassumpt} We note that under Assumption~\ref{assumption2}, there exist $a,b>0$ such that $\sigma(x) \geq a x^{\gamma}$ for all $x \in (0,b]$. Apart from changing the values of various constants which will appear later on, the values of $a$ and $b$ do not play any role in the proof of Theorem~\ref{thm_main_compact}, so when working under Assumption~\ref{assumption2} we may take $a=b=1$ without loss of generality. That is, hereafter when working under Assumption~\ref{assumption2} we will assume that  
\begin{equation} \label{eq_sigmaassumpstrong}
\sigma(x) \geq x^\gamma \quad \text{for all } x \in (0,1].
\end{equation}
\end{remark}

The purpose of the process $\Delta$ constructed in Lemma~\ref{lemma_coupling} is made clear by the next result. In addition to the stopping time $T^i_\delta$, we define
\begin{equation} \label{def_taui1}
	\tau_1^i := \inf\{ t \geq 0 : Y_t(i) > 1\}.
\end{equation}

\begin{lemma} \label{lemma_coupling} Suppose Assumptions~\ref{assumption1} and ~\ref{assumption2} hold, the latter in the form \eqref{eq_sigmaassumpstrong}. Suppose $X$ is a continuous $\ell^1_+$-valued solution to \eqref{e_sdesystem}. 
Let $i \in \Z$ and $\delta>0$. There exists a filtered probability space $(\Omega, \sF, (\sF_t)_{t \geq 0}, \bP)$ on which there lives a copy of $X$ and a continuous semimartingale $\bar{Y}(i) = (\bar{Y}_t(i))_{t \geq 0}$ such that
\begin{itemize}
	\item $\bar{Y}_t(i) \geq Y_t(i) \geq 0$ for all $t \in [0,T^i_\delta \wedge \tau^i_1]$ a.s.
	\item $\bar{Y}(i)$ has Doob-Meyer decomposition
		\begin{equation*}
			\bar{Y}_t(i) = Y_0(i) + \delta t + M_t, \quad t \geq 0,
		\end{equation*}
	where $M = (M_t)_{t \geq 0}$ is a continuous $(\sF_t)$-local martingale a.s.\ satisfying
	\begin{equation} \label{eq_MQV_coupling}
		\langle M\rangle_t - \langle M \rangle_s \geq \frac 1 2\int_s^t \bar{Y}_u^{2\gamma} du \quad \text{for all }\,\, t \geq s \geq 0.
	\end{equation}
\end{itemize}
\end{lemma}
\begin{proof}
	Without loss of generality we choose $i = 0$ (and write $Y_t = Y_t(i)$, $d_t = d_t(i)$, $T_\delta = T^i_\delta$ and $\tau_1 = \tau^i_1$). Fix $\delta > 0$ and let $\Delta$ be the process constructed in Lemma~\ref{lemma_Delta}. For $t \geq 0$, we define $\hat{Y}_t$ by
	\[\hat{Y}_t := Y_t + \Delta_t.\]
	Since $\Delta_t \geq 0$, we a.s.\ have $\hat{Y}_t \geq Y_t$ for all $t \geq 0$. By Corollary~\ref{cor_Ystochastic} and \eqref{eq_Delta},
	\begin{align} \label{eq_barY_rep}
		\hat{Y}_t = Y_0 + \int_0^t (d_s + (\delta - d_s)_+)ds + \hat{M}_t,
	\end{align}
	where $\hat{M}$ is a continuous local martingale whose quadratic variation $\langle \hat{M}\rangle$ satisfies
	\begin{equation} \label{eq:couplemartQV}
		\langle \hat{M} \rangle_t - \langle \hat{M} \rangle_s = \frac 1 2 \int_s^t \bigg( \Delta_u^{2\gamma} +  \sum_{j \in \Z^d : j_1 \geq 0} \sigma(X_u(j))^2 \bigg)du, \quad t \geq s \geq 0.
	\end{equation}
    If $u \leq \tau_1$, then $Y_u(i) \leq 1$. In particular, since $X_u(j) \leq Y_u(i)$ for all $j \in \Z^d$ with $j_1 \geq 0$, $X_s(j) \leq 1$ for all such $j$ and all $u \in [0, \tau_1]$. Hence, by \eqref{eq_sigmaassumpstrong}, $\sigma(X_u(j)) \geq X_u(j)^{\gamma}$ for all $j$ with $j_1 \geq 0$, for all $u \in [0,\tau_1]$, and we conclude from \eqref{eq:couplemartQV} that
    	\begin{equation*} 
		\langle \hat{M} \rangle_t - \langle \hat{M} \rangle_s \geq \frac 1 2 \int_s^t \bigg( \Delta_u^{2\gamma} +  \sum_{j \in \Z^d : j_1 \geq 0} X_u(j)^{2\gamma} \bigg)du \quad \text{ for all }s,t \in [0,\tau_1] \text{ with }\, s\leq t.
	\end{equation*}
	Since $2\gamma < 1$, for $a,b \geq 0$, $a^{2\gamma} + b^{2\gamma} \geq (a+b)^{2\gamma}$. Applying this countably
    many times (we omit the limiting argument), it follows that 
	\[\Delta_u^{2\gamma} + \sum_{j \in \Z^d : j_1 \geq 0} X_u(j)^{2\gamma} \geq \bigg(\Delta_u +  \sum_{j \in \Z^d : j_1 \geq 0} X_u(j) \bigg)^{2\gamma} = (\Delta_u + Y_u)^{2\gamma} =\hat{Y}_u^{2\gamma},\]
	and we conclude that with probability one,
		\begin{equation} \label{eq:hatMqv}
		\langle \hat{M} \rangle_t - \langle \hat{M} \rangle_s \geq \frac 1 2  \int_s^t\hat{Y}_u^{2\gamma} du, \quad \quad \text{ for all } s,t \in [0,\tau_1] \text{ with } s \leq t.
	\end{equation}
	Noting that $d_s + (\delta - d_s)_+ = \delta$ for all $s \in [0,T_\delta]$, it follows from \eqref{eq_barY_rep} that
	\begin{align*} 
		\hat{Y}_t = Y_0 +\delta t+ \hat{M}_t, \quad t \in [0,T_\delta].
	\end{align*}
    Combined with \eqref{eq:hatMqv}, the above implies that for $t \in [0,T_\delta \wedge \tau_1]$, $\hat{Y}_t$ satisfies the desired conditions. To complete the construction, we define
	\[\bar{Y}_t := \begin{cases}
		\hat{Y}_t, &t \in [0,T_\delta \wedge \tau_1], \\ \tilde{Y}_t, &t > T_\delta \wedge \tau_1,
	\end{cases}\]
	where $\tilde{Y}$ is a solution to 
	\begin{equation}
		\tilde{Y}_t = \hat{Y}_{T_\delta \wedge \tau_1} + \delta (t - (T_\delta \wedge \tau_1)) +  \int_{T_\delta \wedge \tau_1}^t \tilde{Y}_ s^{\gamma}d\tilde{W}_s, \quad t \geq T_\delta \wedge \tau_1.
	\end{equation}
	driven by an independent Brownian motion $\tilde{W}$. The existence of a probability space supporting $X$, $\Delta$ and $\tilde{Y}$ can be shown along similar lines to the proof of Lemma~\ref{lemma_Delta}, and we omit the argument. It is readily verified that the Doob-Meyer decomposition of $\bar{Y}$ satisfies the desired properties; since $\bar{Y}_t = \hat{Y}_t$ for $t \in [0,T_\delta \wedge \tau_1]$, this completes the proof.
\end{proof}

\begin{remark} \label{remark_countablecoupling}
	The proofs of Lemmas~\ref{lemma_Delta} and \ref{lemma_coupling} can be easily generalized so that the probability space supports countably many couplings of the same type. Thus, we may work on a probability space supporting a copy of $X$ such that for all $i \in \Z$ and $n \in \N$, there exists a process $(\bar{Y}_t(i,n))_{t \geq 0}$ such that
	\[ \bar{Y}_t(i,n) = Y_0(i) + \frac{1}{n}t + M_t(i,n),  \quad t \geq 0,\]
	where $M_t(i,n)$ satisfies \eqref{eq_MQV_coupling} (with $\bar{Y}_t(i,n)$ on the right-hand side), and with probability one,
	\[ \bar{Y}_t(i,n) \geq Y_t(i) \geq 0 , \quad t \in [0, T^i_{1/n} \wedge \tau^i_1].\]
\end{remark}

We conclude the section by proving that for any $\delta > 0$, $T^i_\delta \wedge \tau^i_1$ is large with high probability provided $i$ is sufficiently large. This will allow us to use the couplings constructed above with high probability. 

\begin{lemma} \label{lemma_Tbound} Suppose Assumption~\ref{assumption1} holds and that $X$ is a continuous $\ell^1_+$-valued solution to \eqref{e_sdesystem}. Then for any $\delta>0$ and $T>0$,
	\[ \lim_{i \to \infty} \bP(T^i_\delta \wedge \tau^i_1 \leq T) = 0.\]		
\end{lemma}
\begin{proof}
    We recall the definitions of $d_t(i)$, $T^i_\delta$ and $\tau^i_1$ from \eqref{def_dti}, \eqref{def_T} and \eqref{def_taui1}. Let $\psi := h + |\cL^* h|$. Then it is immediate that for any $i \in \Z$, for all $t \geq 0$,
    \begin{equation*}
        Y_t(i) \leq \langle X_t, \psi(\cdot - \underline{i})\rangle.
    \end{equation*}
    Moreover, if $L>0$ denotes the Lipschitz constant of $f$, since $f(0) = 0$, we have
    \begin{equation*} 
        |d_t(i)| \leq (1+L)\langle X_t,  \psi(\cdot - \underline{i}) \rangle. 
    \end{equation*}
    Again using the definition~\eqref{def_dti} of $d_t(i)$, we conclude from these two bounds that the desired result will follow if we can prove that for any $T>0$,
    \begin{equation} \label{eq_dti_sts}
    \sup_{t \in [0,T]} \langle X_t, \psi(\cdot - \underline{i}) \rangle \xrightarrow[i \to \infty]{p} 0.
    \end{equation}
    Let $T>0$. First, we note that $\psi$ and $\cL^* \psi$ are both bounded. From the boundedness of the former and Lemma~\ref{lemma_bddweak}, we have 
    \begin{align} \label{eq_dti_psi_rep}
		\langle X_t,  \psi(\cdot - \underline{i})\rangle  &=\langle X_0,  \psi(\cdot - \underline{i})\rangle + \int_{0}^t \langle X_s, \cL^*  \psi(\cdot - \underline{i}) \rangle ds  + \int_{0}^t \langle f(X_s),   \psi(\cdot - \underline{i}) \rangle ds  \notag
        \\ &\hspace{1 cm}+ \int_{0}^t \sum_{j \in \Z^d}  \psi(j - \underline{i}) X_s(j)^\gamma dB_s(j), \quad t \geq 0.
	\end{align}
    Since $h(j) = 0$ for all $j \in \Z^d$ with $j_1 < 0$, it is not hard to see that 
    \begin{equation} \label{eq_Lhvanish}
    \lim_{i \to \infty} \sup_{j \in \Z^d: j_1 \leq -i} | \psi (j)| = 0, \qquad \lim_{i \to \infty} \sup_{j \in \Z^d: j_1 \leq -i} |\cL^* \psi (j)| = 0.
    \end{equation} 
    Boundedness of $\psi$, \eqref{eq_Lhvanish} and $X_0 \in \ell^1_+$ imply that
	\begin{equation} \label{eq_dti_d0}
		\lim_{i \to \infty} \langle X_0, \psi(\cdot - \underline{i}) \rangle = 0.
	\end{equation}
	Next, we remark that
	\[ \sup_{t \in [0,T]} \left|\int_{0}^t \langle X_s, \cL^*  \psi(\cdot - \underline{i})\rangle ds\right| \leq \int_{0}^{T} \langle X_s, |\cL^* \psi (\cdot - \underline{i})|\rangle |) ds,\]
    and hence by \eqref{eq_firstmoment} and Remark~\ref{remark_bddmild},
    \begin{equation} \label{eq_dti_drift_uniform0}
    \E\left[\sup_{t \in [0,T]} \left| \int_{0}^t \langle X_s, \cL^* \psi(\cdot - \underline{i})\rangle ds \right| \right] \leq \int_0^T \langle P_s X_0, |\cL^* \psi(\cdot- \underline{i}) | \rangle ds.
	\end{equation}
    Using $X_0 \in \ell^1_+$ and elementary properties of the random walk semigroup, it is straightforward to show that
    \[\lim_{i \to \infty} \sup_{s \in [0,T]} \langle P_s X_0, \indc_{j_1 \geq i} \rangle = 0.\]
    Combined with \eqref{eq_Lhvanish} and \eqref{eq_dti_drift_uniform0}, this implies that
	\begin{equation} \label{eq_dti_drift_uniform}
		\lim_{i \to \infty} \E\left[ \sup_{t \in [0,T]} \left| \int_{0}^t \langle X_s, \cL^* \psi(\cdot - \underline{i})\rangle ds \right| \right] = 0.
	\end{equation} 
    Since $|\langle f(X_s), \psi \rangle | \leq L \langle X_s, \psi \rangle$, we can control the third term on the right-hand side of \eqref{eq_dti_psi_rep} using the same argument as above, and we conclude that 
    \begin{equation} \label{eq_dti_drift_uniform2}
		\lim_{i \to \infty} \E\left[ \sup_{t \in [0,T]} \left| \int_{0}^t \langle f(X_s), \psi(\cdot - \underline{i})\rangle ds \right| \right] = 0.
	\end{equation} 
	We now turn our attention to the martingale term in \eqref{eq_dti_psi_rep}. We write 
	\[ M_t(i) := \int_{0}^t \sum_{j \in \Z^d } \psi (j-\underline{i}) X_s(j)^\gamma dB_s(j)\]
    and observe that its quadratic variation satisfies
	\begin{equation} \label{eq_dti_qv}
		\langle M(i) \rangle_{T} = \int_{0}^T \sum_{j \in \Z^d}\psi(j - \underline{i}) X_s(j)^{2\gamma} ds. 
	\end{equation}
    To see that the above vanishes a.s. as $i \to \infty$, we first observe that 
    \[ \int_0^T \sum_{j \in \Z^d} X_s(j)^{2\gamma} ds < \infty \text{ a.s.},\]
    which follows \eqref{eq_bddweak} using $\phi \equiv 1$. Since $\psi$ is bounded and $\psi (\cdot - \underline{i})$ vanishes point-wise as $i \to \infty$ from \eqref{eq_Lhvanish}, it follows from \eqref{eq_dti_qv} that 
    \[ \langle M(i) \rangle_T \xrightarrow[i \to \infty]{} 0 \,\text{ a.s.} \]
    A priori the quadratic variation may not be integrable, but a routine argument using stopping and the Burkholder-Davis-Gundy inequality implies that 
	\begin{equation} \label{eq_dti_mart}
		\sup_{t \in [0,T]} |M_t(i)| \stackrel{p}{\to} 0 \,\text{ as } i \to \infty.
	\end{equation}
	It now follows from \eqref{eq_dti_d0}, \eqref{eq_dti_drift_uniform} and \eqref{eq_dti_drift_uniform2} with Markov's inequality, and \eqref{eq_dti_mart} that 
    \begin{equation*} 
    \lim_{i \to \infty} \bP\left(\sup_{t \in [0,T]} \langle X_t, \psi(\cdot - \underline{i})\rangle > \epsilon \right) = 0.
    \end{equation*}
    for any $\epsilon>0$, which is equivalent to \eqref{eq_dti_sts}. This completes the proof.
\end{proof}

\section{A class of semimartingales} \label{s_M}

We now introduce a class of semimartingales whose analysis will constitute a large part of this paper. These processes generalize, in a sense, the behaviour of solutions to the SDE 
\[dY_t = \delta dt + Y_t^\gamma dB_t.\]
 The processes we consider are not in general Markovian, but we show that they possess a pseudo-strong Markov property which will allow us to analyze them almost as though they were Markovian. The process $(\bar{Y}_t)_{t \geq 0}$ constructed in Lemma~\ref{lemma_coupling}, which dominates $(Y_t(i))_{t \geq 0}$ until a certain stopping time, belongs to the family we now introduce, which we call $\sM_{\gamma,\delta}$. Thus, studying the support properties of solutions to \eqref{e_sdesystem} can be reduced to studying the zero sets of processes in $\sM_{\gamma,\delta}$.



We will work on the canonical space of continuous paths, $\Omega = C(\R_+,\R)$. We denote the Borel $\sigma$-algebra on $\Omega$ by $\cB(\Omega)$. For a typical element we write $\omega = (\omega(t) : t \geq 0)$, and for $t \geq 0$ we denote by $Y_t$ the coordinate mapping, i.e.\ $Y_t(\omega) = \omega(t)$. 
The class of processes we are interested in are semimartingales with Doob-Meyer decomposition
\begin{equation} \label{eq_DMY}
	Y_t = Y_0 + \delta t + M_t,
\end{equation}
where $\delta>0$ and $M = (M_t )_{t \geq 0}$ is a certain type of continuous local martingale. We now fix $\delta >0$. It will be convenient to introduce a second, ``tilted''  coordinate mapping on $\Omega$ corresponding to the associated martingale process. For $\omega \in \Omega$ and $t \geq 0$ we define 
\begin{equation}\label{eq_Mdef}
M_t(\omega) := Y_t(\omega) - Y_0(\omega) - \delta t.
\end{equation}
Besides constant drift $\delta$, the chief characteristic of the processes we consider is that $t \mapsto M_t$ is a continuous local martingale with a certain condition on its quadratic variation. We define
\begin{equation} \label{def_It}
	I_t(\omega) := \frac 1 2\int_0^t Y_s(\omega)^{2\gamma} ds.
\end{equation}
As alluded to in the Introduction, e.g. in \eqref{eq_QVbd_informal}, and further suggested by Lemma~\ref{lemma_coupling}, we will be interested in processes with decomposition \eqref{eq_DMY} for which $t\mapsto \langle M \rangle_t - I_t$ is a.s. non-decreasing.

Where possible, we will omit dependence on $\omega$ of $Y_t$, $M_t$ and $I_t$, although for certain arguments, especially within this section, inclusion of the argument $\omega$ will be useful. Here, and for the rest of the paper, we write $Y = (Y_t)_{t \geq 0}$ and $M = (M_t)_{t \geq 0}$ to denote the entire process, and use the same convention for other time-indexed processes. 


For $t \geq 0$, we write $\sF^0_t = \sigma(Y_s: s \in [0,t])$ and $\sF^0_{t+} = \cap_{\eps>0} \sF^0_{t+\eps}$, so that $(\sF^0_t)_{t \geq 0}$ and $(\sF^0_{t+})_{t \geq 0}$ are respectively the standard coordinate filtration and its right continuous modification. Given a probability measure $\bP$ on $(\Omega,\cB(\Omega))$, we denote by $(\sF_t)_{t \geq 0}$ the $\bP$-completed filtration satisfying the usual conditions. We remark that $M_t \in \sF^0_t$ for all $t\geq 0$ by \eqref{eq_Mdef}. We will use the shorthand $(\sF^0_t)$ to denote the filtration $(\sF^0_t)_{t\geq0}$, and likewise for the other filtrations.

\begin{definition} \label{def_M} A probability measure $\bP$ on $(\Omega,\cB(\Omega))$ is of class $\sM_{\gamma,\delta}$, written $\bP \in \sM_{\gamma,\delta}$, if the following properties hold:
\begin{itemize}
\item $\bP(Y_t \geq 0 \,\, \forall t \geq 0) = 1$.
\item The process $M = (M_t)_{t \geq 0}$ defined by \eqref{eq_Mdef} is a continuous local martingale with respect to $(\sF_t)_{t \geq 0}$ with quadratic variation $\langle M\rangle$ satisfying
\begin{equation} \label{assump_QV}
		\bP(\langle M \rangle_t - \langle M\rangle_s \geq I_t - I_s \,\,\, \forall  \,t \geq s \geq 0)  = 1.
\end{equation}
\end{itemize}
If $\bP \in \sM_{\gamma,\delta}$ and $\bP(Y_0 = x) = 1$ for some $x \geq 0$, we write $\bP \in \sM_{\gamma,\delta}(x)$. 
\end{definition}

Several remarks are in order. 
Note that, given $\bP \in \sM_{\gamma,\delta}$, $M$ is a local martingale with respect any of $(\sF^0_t)$, $(\sF^0_{t+})$ or $(\sF_t)$, and that to prove that a process belongs to $\sM_{\gamma,\delta}$, when checking the local martingale property it suffices to check that $M$ is a local martingale with respect to any of the three. Outside of the proof of the pseudo-strong Markov property (Proposition~\ref{prop_markov}), we will always assume we are working with $(\sF_t)$.
Next, we observe that for any $\bP \in \sM_{\gamma,\delta}$, the canonical process $Y$ has Doob-Meyer decomposition given by \eqref{eq_DMY}, as desired. We also note that $\langle Y \rangle_t = \langle M \rangle_t$ for all $t \geq 0$.

We recall that $\sM_{\gamma,\delta}(x)$ denotes the elements of $\sM_{\gamma,\delta}$ with initial value $x$. We will frequently denote elements of $\sM_{\gamma,\delta}(x)$ by $\bP_x$, which is natural and convenient, but we remind the reader that despite this suggestive notation, there is not a unique law $\bP_x$ of a class-$\sM_{\gamma,\delta}$ process started from $x$; there are many, and $\bP_x$ denotes any such law.

In order to reduce the number of parameters, we have not included an additional constant in the quadratic variation, i.e.\ the more general condition $\langle M\rangle_t - \langle M\rangle_s \geq c(I_t - I_s)$ for some $c >0$. As will be clear from the analysis in Sections~\ref{s_intervals} and \ref{s_zeros}, varying this constant will not change the behaviour we study in any significant way. Thus, the class $\sM_{\gamma,\delta}$ as defined above is calibrated to the special form \eqref{eq_sigmaassumpstrong} for Assumption~\ref{assumption2}, but handling the general case is the same.


%

We now introduce some formalism around shifts and stopping times. For $s \geq 0$ let $\theta_s$ denote the shift operator on $\Omega$, i.e.\ $\theta_s(\omega) = (\omega(s+t) : t \geq 0)$, or equivalently $Y_t(\theta_s(\omega)) = Y_{s+t}(\omega)$ for all $s,t \geq 0$. Now, for $\bP \in \sM_{\gamma,\delta}$, let $\tau$ be an $(\sF_t)$-stopping time. For all $\omega$ such that $\tau(\omega)$ is finite, we define the random shift $\theta_\tau$ by $\theta_\tau(\omega) = \theta_{\tau(\omega)}(\omega)$. In order to make $\theta_\tau$ well defined for all~$\omega$, we redefine $\tau$ by the arbitrary value $\tau(\omega) = 0$ on the event $\{ \omega : \tau(\omega) = \infty\}$. 
We remark that if $\tau$ is an a.s.\ finite stopping time, then as $\{ \omega : \tau(\omega) = \infty\}$ is a null set and so contained in $\sF_0$, this modification of $\tau$ is still an $(\sF_t)$-stopping time.



Processes in $\sM_{\gamma,\delta}$ are not, in general, Markov. However, as we now show, they have another property which for our purposes is just as useful, which we refer to as the pseudo-strong Markov property. It is an essential tool in our analysis. This property has already appeared in the literature and is useful when studying SDEs for which uniqueness in law is a priori unknown, or does not hold, for example see \cite{BBC2007}. Because the processes we consider are not solutions to an SDE or martingale problem, but instead have a condition on their quadratic variation, the variant of the pseudo-strong Markov property which we use here differs from those in the literature (that we are aware of). It is a powerful tool, as it allows us given an ``autonomous" or Markov-like analysis of processes whose dynamics are not autonomous but for which a certain quantity can be bounded in terms of the process itself, in this case the quadratic variation. Certain notations and arguments in the sequel are based on those appearing in \cite{Bass}.

Let $\bP \in \sM_{\gamma,\delta}$, and let $\tau$ be an a.s.\ finite $(\sF_t)$-stopping time $\tau$, and assume as usual that we redefine $\tau$ to be zero on the potential null set where it is infinite. We consider the shifted measure $\bP \circ \theta_\tau$ on $(\Omega,\mathcal{B}(\Omega))$ defined by
\begin{equation}
	\bP \circ \theta_\tau (A) = \bP(A \circ \theta_\tau).
\end{equation} 
Then, let $\Q_\tau^\omega(\cdot)$ be a regular conditional probability for $\bP \circ \theta_\tau (\cdot \, | \, \sF_\tau)$. This exists, for example, by the results in Sections 69-73 of \cite{DellacherieMeyerA}. $\Q_\tau^\omega(\cdot)$ is characterized by the following:
\begin{itemize}
 	\item For $\bP$-a.e.\ $\omega$, $B(\Omega)\ni A \mapsto \Q^\omega_\tau(A)$ is a probability measure on $(\Omega, \sF)$.
	\item For every $A \in B(\Omega)$, $\omega \mapsto  \Q^\omega_\tau(A)$ is a version of the conditional probability $\bP \circ \theta_\tau (A \,| \, \sF_\tau)$.
\end{itemize}
We will sometimes suppress dependence on $\omega$ of $\Q^\omega_\tau$, as we will do with other random variables, in which case we simply write $\Q_\tau$. In particular, for some random variable $Z$ on $(\Omega, \mathcal{B}(\Omega))$, the conditional expectation $\Q_\tau(Z)$, which is a version of $\E[Z \circ \theta_\tau \, | \, \sF_\tau]$, denotes the $\sF_\tau$-measurable random variable $\omega \mapsto \Q^\omega_\tau(Z)$.

We also define, for an event $A \in \sF_\tau$ with $\bP(A) >0$, the probability measure
\begin{equation}
	\bP^A(\cdot) = \frac{1}{\bP(A)} \bP(\cdot \, \cap  A).
\end{equation}
The measure $\bP^A \circ \theta_\tau$ then denotes $\bP^A$ shifted by $\tau$.

The pseudo-strong Markov property (hereafter PSMP) is established in the following result.
\begin{proposition} \label{prop_markov}
	Let $\bP \in \sM_{\gamma,\delta}$ and let $\tau$ be an a.s.  finite $(\sF_t)$-stopping time. \\
	(a) For $\bP$-a.e. $\omega$, $\Q^\omega_\tau \in \sM_{\gamma,\delta}(Y_\tau(\omega))$. \\
	(b) For $A \in \sF_\tau$ with $\bP(A) > 0$, $\bP^A \circ \theta_\tau \in \sM_{\gamma,\delta}$. Moreover, if $\bP(Y_\tau = x \, | \, A) = 1$, then $\bP^A \circ \theta_\tau \in \sM_{\gamma,\delta}(x)$.
\end{proposition}
	Part (a) is stronger and is the ``correct'' notion of a strong Markov property. However, part~(b), which asserts the same conclusion for naïve conditional probabilities, will be a useful simplified version. A typical application of the PSMP is as follows: suppose that $\tau$ is a finite $(\sF_t)$-stopping time, $Z_1$ and $Z_2$ are bounded random variables on $(\Omega,\mathcal{B}(\Omega))$, and $Z_1 \in \sF_\tau$. Then
    \begin{align} \label{eq:markov_sample}
        \E[Z_1 \cdot (Z_2 \circ \theta_\tau)] = \E[Z_1 \cdot \E[Z_2 \circ \theta_\tau \, | \, \sF_\tau]] &= \E[Z_1 \cdot \Q_\tau(Z_2)].
    \end{align} 
    By the PMSP, we a.s. have $\Q_\tau \in \sM_{\gamma,\delta}(Y_\tau)$. For the sake of illustration, suppose $Y_\tau = x$ a.s. for some $x \geq 0$; if we can compute uniform bounds on $\E_x[Z_2]$ over all $\bP_x \in \sM_{\gamma,\delta}(x)$, then we can compute conditional expectations of the form \eqref{eq:markov_sample} more or less as one would do for a Markov process.

	\begin{proof}[Proof of Proposition~\ref{prop_markov}]
	First, we remark that
	\[ \bP \circ \theta_\tau(Y_t \geq 0 \, \forall t \geq 0) = \bP(Y_{\tau + t} \geq 0 \, \forall t \geq 0) \geq \bP(Y_t \geq 0 \, \forall t \geq 0) = 1.\]
	In particular, $\bP \circ \theta_\tau(Y_t \geq 0 \, \forall t\geq0\, | \, \sF_\tau)(\omega) = 1$ for $\bP$-a.e.\ $\omega$, and so by the definition of $\Q^\omega_\tau$ we have
	\begin{equation*}
		\Q^\omega_\tau(Y_t \geq 0 \, \forall t \geq 0) = 1 \text{ for $\bP$-a.e.\ $\omega$.}
	\end{equation*}
	Next, we establish that $M$ is a local martingale under $\Q^\omega_\tau$ for $\bP$-a.e.\ $\omega$. Here we follow the proof of \cite[Proposition~VI.2.1]{Bass}, but we modify it to include localization. First, since $t \mapsto M_{\tau + t} - M_\tau$ is a $(\sF_{\tau+t})$-local martingale, there exists a family of reducing $(\sF_{\tau + t})$-stopping times $(\hat{\sigma}_n : n\in \N)$ so that $t\mapsto M_{\tau + t \wedge \hat{\sigma}_n} - M_\tau$ is a martingale. In fact, we can and will choose to use the explicit stopping times defined by
	\[ \hat{\sigma}_n := \inf \{t \geq 0 : |M_{\tau + t} - M_\tau| = n\},\]
	which satisfy $\lim_{n \to \infty} \hat{\sigma}_n(\omega) = \infty$ for all $\omega$ due to path continuity. We remark that
	\begin{align*}
		M_t \circ \theta_\tau(\omega) = Y_{\tau+t}(\omega) - Y_\tau(\omega) - \delta t &=  (Y_{\tau+t}(\omega) - Y_0(\omega) - \delta(\tau+t))- (Y_\tau(\omega)  - Y_0(\omega) - \delta \tau)
		\\ &= M_{\tau + t}(\omega) - M_\tau(\omega).
	\end{align*}
	By this identity, we have shown that $t \mapsto M_t \circ \theta_\tau$ is an $(\sF_{\tau+t})$-local martingale with the same localizing sequence, which we may express as
	\begin{equation*}
		\hat{\sigma}_n(\omega) = \inf \{ t \geq 0 : |M_t \circ \theta_\tau(\omega)| = n\} = \sigma_n \circ \theta_\tau(\omega),
	\end{equation*}
	where we define $\sigma_n(\omega) := \inf\{t \geq 0 : |M_t(\omega)| = n\}$. In particular, $t\mapsto M_{t \wedge \sigma_n} \circ \theta_\tau$ is a $\sF_{\tau+t}$-martingale for all $n$.

	
	Let $0 \leq s < t$. For $E \in \sF_\tau$ and $F \in \sF_s^0$, by the martingale property,
	\begin{align*}
		\E [ (M_{t \wedge \sigma_n} \circ \theta_\tau )\indc_{E \cap (F \circ \theta_\tau)}] &= \E [\indc_{E \cap (F\circ \theta_\tau)} \E [M_{\tau+(t \wedge \sigma_n)} \, | \, \sF_{\tau + s}]]
		\\ &= \E [\indc_{E \cap (F\circ \theta_\tau)} M_{\tau+(s \wedge \sigma_n)}] 
		\\&=  \E [ (M_{s \wedge \sigma_n} \circ \theta_\tau )\indc_{E \cap (F \circ \theta_\tau)}].
	\end{align*}
	This is equivalent to
	\begin{equation*}
		\E [ ((M_{t \wedge \sigma_n} \indc_{F}) \circ \theta_\tau )\indc_{E}] =  \E [ ((M_{s \wedge \sigma_n} \indc_{F} )\circ \theta_\tau )\indc_{E}].
	\end{equation*}
	Since this holds for all $E \in \sF_\tau$, it implies that for $\bP$-a.e.\ $\omega$,
	\begin{equation*}
	\E [ (M_{t \wedge \sigma_n} \indc_{F})\circ \theta_\tau \, | \, \sF_\tau](\omega) =  \E [ (M_{s \wedge \sigma_n} \indc_{F} ) \circ \theta_\tau \, | \, \sF_\tau](\omega)
	\end{equation*}
	for any version of this conditional expectation. By the definition of $\Q^\omega_\tau$, $\omega \mapsto \E_{\Q^\omega_\tau}[M_{t \wedge \sigma_n }\indc_F]$ is a version of this conditional expectation, and hence for $F \in \sF_s^0$, for $\bP$-a.e.\ $\omega$,
	\begin{equation} \label{eq_rcd1}
		\E_{\Q^\omega_\tau}[M_{t \wedge \sigma_n} 
        \indc_F] = \E_{\Q^\omega_\tau}[M_{s \wedge \sigma_n}\indc_F]. 
	\end{equation}
	Since $ \sF_s^0$ defined over the space of continuous functions is generated by the coordinate mappings at a countable dense set of times, and each of these coordinate $\sigma$-algebras is itself countably generated, there exists a countable monotone class of events $F$ generating $\sF^0_s$ such that \eqref{eq_rcd1} holds for all~$F$ for $\bP$-a.e.\ $\omega$. Thus, by the monotone class theorem, we conclude that for $\bP$-a.e.\ $\omega$,
	\begin{equation*}
		\E_{\Q^\omega_\tau}[M_{t \wedge \sigma_n } \, | \, \sF^0_s ] = M_{s \wedge \sigma_n } \, \text{ $\Q^\omega_\tau$-a.s.},
	\end{equation*}
	where $\E_{\Q^\omega_\tau}[X \, | \, \sF^0_s]$ denotes the conditional expectation (given $ \sF^0_s$) of $X$ under the measure $\Q^\omega_\tau$ . Let $\mathcal{I} \subset \R_+$ be a countable dense subset. The above implies that there exists $\Omega_0 \subset \Omega$ such that $\bP(\Omega_0) = 1$ and 
	\begin{align*}
		&\text{For all $\omega \in \Omega_0$ and $s,t \in \mathcal{I}$ with $s <t$,} \quad \E_{\Q^\omega_\tau}[M_{t \wedge \sigma_n } \, | \, \sF^0_s ] = M_{s \wedge \sigma_n } \, \text{ $\Q^\omega_\tau$-a.s.}
	\end{align*}
	Thus, for $\bP$-a.e.\ $\omega$, the (local) martingale property holds for $M$ under $\Q^\omega_\tau$ for all $s$ and $t$ in $\mathcal{I}$. We now show that it extends to all $s$ and $t$ using path regularity. Fix $\omega\in \Omega_0$. Let $t \in \R_+$ and suppose that $t > s \in \mathcal{I}$. Since $t\mapsto M_{t \wedge \sigma_n}$ is bounded for every $n \in \N$ (recall our assumption on $\sigma_n$), we may take a sequence $(t_k)_{k \in \N} \subset \mathcal{I}$ such that $t_k \downarrow t$ and conclude from continuity of $M$ and the (conditional) dominated convergence theorem that 
	\begin{equation*}
		\lim_{k\to \infty} \E_{\Q^\omega_\tau}[ M_{t_k \wedge \sigma_n} \, | \, \sF^0_s ] = E_{\Q^\omega_\tau}[ M_{t \wedge \sigma_n} \, | \, \sF^0_s ] \, \text{ $\Q^\omega_\tau$-a.s.}
	\end{equation*}
	Since $t_k \in \mathcal{I}$, the left hand side equals $M_{s\wedge\sigma_n}$ a.s.\ for each $k$, and hence the $\sF^0_t$-local martingale property for times $t \in \R_+$ and $s \in \mathcal{I}$. Now let $s,t \in \R_+$ with $s<t$ and let $(s_k)_{k \in \N} \subset \mathcal{I}$ be a sequence decreasing to $s$. By the continuity of paths, we have
	\begin{equation*}
		\E_{\Q^\omega_\tau}[ M_{t \wedge \sigma_n} \, | \, \sF^0_{s+} ] = \lim_{k \to \infty} \E_{\Q^\omega_\tau}[ M_{t \wedge \sigma_n} \, | \, \sF^0_{s_k} ]  =  \lim_{k \to \infty} M_{s_k \wedge \sigma_n} \ = M_{s \wedge \sigma_n} \,  \text{ $\Q^\omega_\tau$-a.s.}
	\end{equation*}
	Thus, we have proved that for $\bP$-a.e.\ $\omega$, $M$ is a $\sF^0_{t+}$-local martingale under $\Q^\omega_\tau$.
	
	Next, we verify the condition \eqref{assump_QV} on the quadratic variation of $M$ under $\Q^\omega_\tau$. First, recall that the quadratic variation of a local martingale can be characterized by the condition that $t \mapsto M_t^2 - \langle M\rangle_t$ is a local martingale. We will prove \eqref{assump_QV} by verifying that $M_t^2 - I_t$ is a local sub-martingale. Indeed, if this is the case, the Doob-Meyer representation implies that there exists a local martingale $t\mapsto N_t$ and a non-decreasing process $t \mapsto A_t$ such that
    \[ M_t^2 - I_t = N_t + A_t.\]
    Since $t \mapsto M_t^2 - \langle M\rangle_t$ is a local martingale, the uniqueness of the decomposition implies that $A_t = \langle M\rangle_t - I_t$ for $t \geq 0$. Since $t\mapsto A_t$ is a.s.\ non-decreasing, this is equivalent to \eqref{assump_QV}. Thus, \eqref{assump_QV} is equivalent to the condition that $t \mapsto M_t^2 - I_t$ is a local submartingale. One can prove that $t \mapsto M_t^2 - I_t$ is a submartingale under $\Q^\omega_\tau$ for $\bP$-a.e.\ $\omega$ using virtually the same argument just used to prove that $M$ is a local martingale, and we omit the details.
	
	It remains to verify that for $\bP$-a.e.\ $\omega$, we have $Y_0(\omega') = Y_\tau(\omega)$ for $\Q^\omega_\tau$-a.e.\ $\omega'$. This uses the same argument as the proof of \cite[Proposition VI.2.1]{Bass}, which we reproduce for completeness. We define the event $A(\omega) = \{\omega' : Y_0(\omega') = Y_\tau(\omega)\}$. Then for any $E \in \sF_\tau$, by definition of $\Q^\omega_\tau$, 
	\begin{align*}
		\E [ \indc_E  \,\Q^\omega_\tau(A(\omega))] = \E[ \indc_E \, \bP\circ \theta_\tau (A \, | \sF_\tau) ] &= \E [ \indc_E \,\bP(Y_\tau = Y_0 \circ \theta_\tau \, | \, \sF_\tau) ]
		\\ &=\E [ \indc_E \,\bP(Y_\tau = Y_\tau \, | \, \sF_\tau) ]
		\\ &= \bP(E).
	\end{align*}
	This implies that $\bP(Q^\omega_\tau(A(\omega)) = 1)=1$. We have now shown that $\bP(\Q^\omega_\tau \in \sM_{\gamma,\delta}(Y_\tau)) = 1$, which proves part (a).
	
	Part (b) is a consequence of part (a). For any $B \in \mathcal{B}(\Omega)$, one has
	\[\bP^A \circ \theta_\tau(B) = \bP(A)^{-1} \int_A \Q^\omega_\tau(B) \bP(d\omega),\]
	and the desired properties can be easily inferred from the result on $\Q^\omega_\tau$. \end{proof}



The next lemma establishes that any stopping time until which $Y_t$ remains bounded is finite, and in fact has exponentially decaying tails. This will allow us to apply the PSMP at all such stopping times. 

\begin{lemma} \label{lemma_bddstop_tail} Let $b > 0$ and $\delta >0$. If $\bP \in \sM_{\gamma,\delta}$ and $\tau$ is any $(\sF_t)$-stopping time such that $\bP(Y_t \leq b \text{ for all }t \in [0, \tau]) = 1$, then 
	\[\bP(\tau > t) \leq 2 e^{-ct}\,\, \text{ for all } \,t \geq 0,\]
	where $c = c(b,\delta) = \delta / (18b)$.
\end{lemma}
\begin{proof}
	Recall that $M_t = Y_t - \delta t - Y_0$. Hence with probability one, $ M_t \leq b - \delta t$ for all $t \leq \tau$. On the other hand, since $Y_t \geq 0$ for all $t \geq 0$ a.s., it follows that with probability one, $M_t \geq -b - \delta t$ for all $t \geq 0$. In particular, the stopped local martingale $M^\tau = (M_{\tau \wedge t})_{t \geq 0}$ satisfies $M^\tau_t \in [-b - \delta (t \wedge \tau), b - \delta (t \wedge \tau)]$ for all $t \geq 0$. One straightforward consequence of this is that $M^\tau$ is in fact a martingale, which follows from a dominated convergence argument on the local martingale. A second consequence is that, for all $0 \leq s < t$, we have
	\[ - 2b - \delta(t-s) \leq M^\tau_t - M^\tau_s \leq 2b - \delta(t-s),\]
	and in particular 
	\[ | M^\tau_t - M^\tau_s | \leq 2b + \delta(t-s).\]
	For $k \in \N_0$, let $t_k = b\delta^{-1} k$. We define the discrete-time process $m = (m_k)_{k \in \N_0}$ by $m_k := M^\tau_{t_k}$. Then $m_0 = 0$, and moreover $m$ is a discrete-time martingale. 
	The previous inequality for the increments of $M^\tau$ implies that $|m_{k+1} - m_k| \leq 3b$ for all $k$. 
	
	Let $n \in \N$. We remark that if $\tau > t_n$, then $m_n = M^\tau_{t_n} = M_{t_n} \leq b -\delta t_n = -b(n-1)$. Because $m$ is a martingale with increments bounded by $3b$, we obtain from Azuma's inequality that
	\begin{align*}
		\bP(\tau > t_n ) &\leq \bP(m_n \leq -b(n-1) )
		\\ &= \exp \left(- \frac{b^2(n-1)^2}{18b^2 n} \right)
		\\ & \leq \exp \left( -(n - 2)/18\right).
	\end{align*}
	To conclude, we observe that if $t \in [t_n, t_{n+1})$, then 
	\begin{align*} \bP(\tau > t) \leq \bP( \tau > t_n) \leq e^{1/9} e^{ - n/18} &\leq e^{1/9 + 1/18} e^{- (n+1)/18}
		\\ &= e^{1/6} e^{- \delta t_{n+1}/(18b)}
		\\ &\leq 2 e^{-ct}, \end{align*}
	where $c = \delta/(18b)$. This completes the proof.
\end{proof}

We now specialize the PSMP to state a useful result concerning hitting times, which illustrates how it works as a substitute for the strong Markov property when working with $\sM_{\gamma,\delta}$-class semimartingales. Here and for the rest of the paper, for $x \geq 0$ we define
\begin{equation} \label{def_tau}
\tau_x(\omega) = \inf \{t \geq 0 : \omega(t) = x\} = \inf \{t \geq 0 : Y_t = x\}.
\end{equation}
Of course, $\tau_x$ is an $(\sF_t)$-stopping time (as well as an $(\sF^0_t)$-stopping time). 




\begin{lemma}\label{lemma_stop_markov} Suppose $0 \leq a \leq x \leq w \leq y \leq b$ with $x < y$ and let $\bP_w \in \sM_{\gamma,\delta}(w)$. Then provided $\bP_w(\tau_x < \tau_y) >0$,
\begin{equation}
\bP_w(\tau_x < \tau_y, \tau_a < \tau_b) = \bP_w(\tau_x < \tau_y) \bP_x'(\tau_a < \tau_b),
\end{equation}
for some $\bP_x' \in \sM_{\gamma,\delta}(x)$. Moreover, the analogous inequality holds for all variations of the ordering of the stopping times, i.e.\ exchanging $\tau_x < \tau_y$ for $\tau_y < \tau_x$, and similarly for $\tau_a$ and $\tau_b$.
\end{lemma}
\begin{proof}
%
%
Since $\bP_w \in \sM_{\gamma,\delta}(w)$, there is a $\bP_w$-null set $N_1$ such that $Y_0(\omega) = w$ for $\omega \in N_1^c$. 
From continuity of the sample paths and the fact that $w \in [x,y] \subseteq [a,b]$, it follows that for all $\omega \in N_1^c$, $\tau_a \geq \tau_x$ and $\tau_b \geq \tau_y$. In particular, this implies that for all $\omega \in N_1^c \cap \{\tau_x < \tau_y\}$, $\tau_a(\omega) \geq \tau_x(\omega) $ and $\tau_b(\omega) \geq \tau_x(\omega)$. It then follows from the definition of the hitting times that
\begin{align*}
& \text{For all $\omega \in N_1^c \cap \{\tau_x < \tau_y\}$,}
\\ &\hspace{1 cm} \tau_a \circ \theta_{\tau_x}(\omega) = \tau_a(\omega) - \tau_x(\omega) \quad  \text{ and } \quad \tau_b \circ \theta_{\tau_x}(\omega) = \tau_b(\omega) - \tau_x(\omega).
\end{align*}
For all such $\omega$, it holds that 
\[ \tau_a(\omega) < \tau_b(\omega) \iff \tau_a(\omega) - \tau_x(\omega) < \tau_b(\omega) - \tau_x(\omega) \iff \tau_a \circ \theta_{\tau_x}(\omega)  < \tau_b \circ \theta_{\tau_x}(\omega).\]
In particular, 
\begin{align*}
\bP_w(\tau_x < \tau_y, \tau_a < \tau_b) &= \bP_w(\tau_x < \tau_y, \tau_a \circ \theta_{\tau_x} < \tau_b \circ \theta_{\tau_x})
\\ &= \bP_w(\tau_x < \tau_y) \cdot \bP_w^{\{\tau_x < \tau_y\}} (\tau_a \circ \theta_{\tau_x} < \tau_b \circ \theta_{\tau_x})
\\ &=  \bP_w(\tau_x < \tau_y) \cdot \bP^{\{\tau_x < \tau_y\}} \circ \theta_{\tau_x} (\tau_a  < \tau_b ).
\end{align*}
By the PSMP (i.e.\ Proposition~\ref{prop_markov}(b)), $\bP_w^{\{\tau_x < \tau_y\}} \circ \theta_{\tau_x} \in \sM_{\gamma,\delta}(x)$. This proves the result.
\end{proof}

%

To conclude the section, we remark that the process constructed in Lemma~\ref{lemma_Delta} belongs to $\sM_{\gamma,\delta}$. 
\begin{lemma} Under the assumptions of Lemma~\ref{lemma_Delta}, we have $\cL(\bar{Y}) \in \sM_{\gamma,\delta}(Y_i(0))$, where $\cL(\bar{Y})$ denotes the law of the canonical projection of $\bar{Y}$ onto $C(\R_+,\R)$.
\end{lemma}
\begin{proof}
	This follows immediately from Lemma~\ref{lemma_Delta} and Definition~\ref{def_M}.
\end{proof}

\section{The behaviour of class-$\sM_{\gamma,\delta}$ semimartingales near $0$} \label{s_intervals}
In this section we analyze the behaviour of class-$\sM_{\gamma,\delta}$ semimartingales near zero and prove several key results. These results are the basis for our analysis of the local time at zero, which is central to the proof of our main result. Throughout the section we fix
\[ \gamma \in (0,1/2).\] 
This is of critical importance to the results we prove. Viewing $\gamma$ as fixed, we will not track the dependence of constants on this parameter. We will also restrict our attention to $\delta \in (0,1]$. This is much less important, and indeed versions of our results hold for arbitrary positive $\delta$; we make the restriction for convenience, as it allows us to state various results uniformly in $\delta$ instead of tracking dependencies of various constants on the parameter $\delta$. Moreover, ultimately we apply the results here for small values of $\delta$.

Throughout the section, we are interested in the behaviour of the canonical process $Y$ under $\bP \in \sM_{\gamma,\delta}$. We recall our notation for hitting times defined in \eqref{def_tau}, that is
\begin{equation*}
\tau_a = \inf \{t\geq 0 : Y_t = a\}.
\end{equation*}
If $Y$ were a continuous martingale, we would have that for $a < x < b$, if $Y_0 =x$, then the probability that $\tau_b < \tau_a$ is equal to $\frac{x-a}{b-a}$. However, $Y$ has a non-negative drift and so is not a martingale. The first main result, which we now state, is that for intervals close enough to zero, the effect of the drift is negligible and the ``martingale probabilities" are approximately recovered.


\begin{proposition} \label{prop_hittingprob_main} There exist a constant $c_1>0$ such that for sufficiently small $\zeta$, for all $\delta \in (0,1]$ and $\bP_\zeta \in \sM_{\gamma,\delta}(\zeta)$,
\begin{equation} \label{eq_hittingprob_main}
\frac 1 2 \leq \bP_\zeta( \tau_{2 \zeta} < \tau_0) \leq  \frac 1 2 + 8 \zeta^{c_1}.
\end{equation}
\end{proposition}

The other main result characterizes the tails of the exit time of intervals which have lower endpoint zero. For such an interval of length $\zeta$, the expectation of the exit time has order $\zeta^{2-2\gamma}$ (see Lemma~\ref{lemma_exit_mean}). The following stronger results gives exponential decay of the tail along constant multiplies of the mean. For $A \subseteq \R_+$, we introduce the notation
\begin{equation} \label{eq:M_geninit}
    \sM_{\gamma,\delta}(A) := \left\{ \bP \in \sM_{\gamma,\delta} : \bP(Y_0 \in A) = 1 \right\}.
\end{equation}


\begin{proposition} \label{prop_exit_exp} There are constants $C_1 \geq 1$ and $c_2> 0$ such that for sufficiently small $\zeta$, for all $\delta \in (0,1]$ and $\bP \in \sM_{\gamma,\delta}((0,2\zeta))$,
	\begin{equation*}
		\bP(\tau_0 \wedge \tau_{2\zeta} > m \zeta^{2-2\gamma}) \leq C_1 e^{-c_2 m}
	\end{equation*}
	for all $m \in \N$.
\end{proposition}


The remainder of this section, which is divided into several subsections, is devoted to proving the two results above. If the processes under consideration were Markov, the argument would be considerably simplified using scale functions and speed measures. Lacking this machinery, we must do a number of calculations by hand.



\subsection{Exit times and exit distributions for intervals} Let $\zeta \in (0,1]$, and for $n \in \Z$, let $a_n = \zeta 2^{-n}$. We define the interval $I_n = (a_{n+1}, a_{n-1})$. Thus, if $Y_0  \in I_n$, $\tau_{a_{n-1}} \wedge \tau_{a_{n+1}}$ is the exit time from $I_n$. In order to prove Proposition~\ref{prop_hittingprob_main}, we first obtain estimates on $\bP_{a_n}(\tau_{a_{n-1}} < \tau_{a_{n+1}})$ for all $\bP_{a_n} \in \sM_{\gamma,\delta}(a_n)$ for all $n \in \N_0$. If $Y$ were a martingale this probability would equal $\frac 1 3$.  The next several results quantify the error, and show that the upper and lower exit probabilities from $I_n$ when started from $a_n$ are close to the martingale probabilities to within an error that is exponentially small in $n$. 

Recall the decomposition $Y_t = Y_0 + \delta t+ M_t$. By Dubins-Schwarz, there exists a standard Brownian motion $W$ such that $M_t = W_{\langle M\rangle_t} = W_{\langle Y\rangle_t}$ for all $t \geq 0 $. Suppose that $Y_0 \geq a > 0$. By \eqref{assump_QV}, we have
\begin{align}
\langle Y\rangle_{t \wedge \tau_{a}} \geq \frac 1 2 \int_0^{t \wedge \tau_{a}} Y_s^{2\gamma} ds &\geq \frac{a^{2\gamma}}{2} (t \wedge  \tau_a).
\end{align}
In particular, if we define the constant $\kappa_n := a_{n+1}^{2\gamma}/2$, for every $n \in \N$ we have
\begin{align} \label{eq_QVstop}
&\text{For every $\bP\in \sM_{\gamma,\delta}$ with $\bP(Y_0 \geq a_{n+1}) = 1$, $\bP$-a.s.,}
\\ & \hspace{1 cm} \langle Y\rangle_{t \wedge \tau_{a_{n+1}}} \geq \kappa_n (t \wedge \tau_{a_{n+1}}) \quad \text{ for all $t \geq 0$}. \notag
\end{align}
We will generally use $W$ to denote the Brownian motion obtained as a time-change of $M$, whereas we will use $B$ to denote a generic Brownian motion, for which we adopt the notation $\bP_x^B$ and $\E^B_x$ to denote the law and expectation when $B_0 = x$.

Suppose that $Y_0  \in (a_{n+1}, a_{n-1})$. Recalling that $M_t = Y_t - Y_0 - \delta t$ under $\bP$, in this case we may rewrite $\tau_{a_{n-1}}$ and $\tau_{a_{n+1}}$ as 
\begin{equation} \label{eq_taustop_Mrep}
\tau_{a_{n-1}}= \inf \{ t> 0 : M_t = a_{n-1} - Y_0 -\delta t\}, \quad \tau_{a_{n+1}} = \inf \{ t> 0 : M_t = a_{n+1} - Y_0 -\delta t\}.\end{equation}


We define $M_t^* := \sup_{s \in [0,t]} M_s$ and $|M_t|^* := \sup_{s \in [0,t]} |M_s|$ and will use the same notation for other processes.
Recall that $a_n = \zeta 2^{-n}$ and $I_n = (a_{n+1}, a_{n-1})$ for all $n \in \Z$, and recall from \eqref{eq:M_geninit} that $\bP \in \sM_{\gamma,\delta}(I_n)$ if $\bP \in \sM_{\gamma,\delta}$ and $\bP(Y_0 \in I_n) = 1$.


\begin{lemma} \label{lemma_Y_exit}
There exists $q \in (0,1)$ such that for all $\delta \in (0,1]$ and all $\zeta>0$ such that $2\delta\zeta^{1-2\gamma} \leq 1$, for every $n \in \N_0$ and $\bP \in \sM_{\gamma,\delta}(I_n)$,
\[ \bP (\tau_{a_{n-1}} \wedge \tau_{a_{n+1}} > a_n^{2-2\gamma}) \leq q.\]
\end{lemma}
\begin{proof}
Let $n \in \N$, and define $T = a_n^{2-2\gamma}$. By \eqref{eq_taustop_Mrep}, we observe that remark that $\tau_{a_{n-1}} \wedge \tau_{a_{n+1}} > T$ if and only if $M_t \in [a_{n+1} - Y_0 - \delta t, a_{n-1} - Y_0 - \delta t]$ for all $t \in [0,T]$. Since $Y_0 \in I_n$, we have $a_{n+1} - Y_0 \geq -3\zeta 2^{-n-1}$ and $a_{n-1} - Y_0 \leq 3  \zeta 2^{-n-1}$. In particular, we obtain that 
\begin{equation*}
\{ \tau_{a_{n-1}}  \wedge \tau_{a_{n+1}}  > T \} \subseteq  \{ M_t \in [-3 a_{n+1} - \delta t, 3 a_{n+1} - \delta t] \, \text{ for all } t \in [0,T]\}.
\end{equation*}
Note that since $\delta \in (0,1]$ and $2\delta\zeta^{1-2\gamma} \leq 1$ that $\delta T \leq a_{n+1}$, and hence the above implies that
\begin{equation*}
\{ \tau_{a_{n-1}} \wedge \tau_{a_{n+1}}  > T \} \subseteq  \{ M_t \in [-a_{n-1}, a_{n-1}] \, \text{ for all } t \in [0,T]\},
\end{equation*}
where we have used the fact that $4 a_{n+1} = a_{n-1}$. The event on the right hand side above is equivalent to $\{|M_T|^* \leq a_{n-1}\}$. Hence, using the Dubins-Schwarz representation for $M$ and \eqref{eq_QVstop},
%
\begin{align} \label{e_exitlemma1}
\bP (  \tau_{a_{n-1}}  \wedge \tau_{a_{n+1}}  > T ) &= \bP_x ( \tau_{a_{n-1}}  \wedge \tau_{a_{n+1}}  > T,  |M_T|^* \leq a_{n-1} ) \notag
\\ &\leq \bP_x ( T < \tau_{a_{n+1}}  ,  |W_{\kappa_n T}|^* \leq a_{n-1} ) \notag
\\ &\leq  \bP^B_0 ( |B_{\kappa_n T}|^* \leq a_{n-1} ).
\end{align}
The probability above can easily be bounded using Brownian scaling as follows:
\begin{align*}
\bP^B_0 ( |B_{\kappa_n T}|^* \leq a_{n-1} )&= \bP^B_0 ((\kappa_n T)^{1/2} |B_1|^* \leq a_{n-1})
\\ &= \bP^B_0 ( |B_1|^* \leq a_{n-1}(\kappa_n T)^{-1/2}).
\end{align*}
Recalling that $T = a_n^{2-2\gamma}$ and $\kappa_n =  a_{n+1}^{2\gamma}$, we find that $a_{n-1}(\kappa_n T)^{-1/2} =  2^{1 + \gamma}$. Substituting this into the above, \eqref{e_exitlemma1} now implies that
\begin{equation*}
\bP ( \tau_{a_{n-1}} \wedge \tau_{a_{n+1}}   > T ) \leq \bP^B_0 ( |B_1|^* \leq  2^{1 + \gamma})=:q .
\end{equation*}
This proves the result. \end{proof}

The previous lemma leads directly to the following.
\begin{lemma} \label{lemma_exit_exp} There are constants $c_3>0$ and $C_2 \geq 1$ such that for all $\delta \in (0,1]$ and all $\zeta >0$ satisfying $2\delta \zeta^{1-2\gamma} \leq 1$, for all $n \in \N_0$ and $\bP \in \sM_{\gamma,\delta}(I_n)$, 
\begin{equation*} \bP ( \tau_{a_{n-1}} \wedge \tau_{a_{n+1}}  > m a_n^{2-2\gamma}) \leq e^{-c_3 m} \end{equation*}
for all $m \in \N$, and 
\begin{equation*}
\E[\tau_{a_{n-1}} \wedge \tau_{a_{n+1}}] \leq C_2 a_n^{2-2\gamma}.
\end{equation*}
\end{lemma}
\begin{proof}
Let $\zeta>0$ and $\delta>0$ be as in the statement, and let $n \in \N_0$, $\bP \in \sM_{\gamma,\delta}(I_n)$ and $m \in \N$. We recall that $\Q_t$ denotes the regular conditional distribution of $Y \circ \theta_t$ given $\sF_t$. Conditioning on $\sF_{(m-1) a_n^{2-2\gamma}}$ and applying the PSMP (Proposition~\ref{prop_markov}), we obtain
\begin{align*}
&\bP(\tau_{a_{n-1}} \wedge \tau_{a_{n+1}} > m a_n^{2-2\gamma}) 
\\&\hspace{.5 cm}= \E\left[ \indc(\tau_{a_{n-1}} \wedge \tau_{a_{n+1}} > (m-1) a_n^{2-2\gamma}) \cdot \E\left[\tau_{a_{n-1}} \wedge \tau_{a_{n+1}} > m a_n^{2-2\gamma} \, | \, \sF_{(m-1) a_n^{2-2\gamma}}\right] \right]
\\&\hspace{.5 cm}= \E\left[ \indc(\tau_{a_{n-1}} \wedge \tau_{a_{n+1}} > (m-1) a_n^{2-2\gamma}) \cdot \Q_{(m-1)a_n^{2-2\gamma}}\left(\tau_{a_{n-1}} \wedge \tau_{a_{n+1}} > a_n^{2-2\gamma}\right) \right]
\\ &\hspace{.5 cm}\leq \E[ \indc(\tau_{a_{n-1}} \wedge \tau_{a_{n+1}} \geq (m-1) a_n^{2-2\gamma})] \cdot q.
\end{align*}
In the last line, we apply Lemma~\ref{lemma_Y_exit} to $\Q_{(m-1)a_n^{2-2\gamma}} \in \sM_{\gamma,\delta}(I_n)$, which holds because $\Q^\omega_{(m-1)a_n^{2-2\gamma}} \in \sM_{\gamma,\delta}(Y_{(m-1)a_n^{2-2\gamma}})$ for $\bP$-a.e.\ $\omega$ by the PSMP, and $Y_{(m-1)a_n^{2-2\gamma}} \in I_n$ on $\{\tau_{a_{n-1}} \wedge \tau_{a_{n+1}} \geq (m-1) a_n^{2-2\gamma}\}$. (Lemma~\ref{lemma_Y_exit} is also applicable because $\delta \in (0,1]$ and $2\delta \zeta^{1-2\gamma} \leq 1$.) Applying the same procedure iteratively proves that $\bP(\tau_{a_{n-1}} \wedge \tau_{a_{n+1}} > m a_n^{2-2\gamma}) \leq q^m$, and hence the result holds if $c_3$ is defined by $e^{-c_3} = q$. The bound on the expectation follows immediately.
\end{proof}

%


We now proceed to obtain upper and lower bounds on $\bP_{a_n}(\tau_{a_{n-1}} < \tau_{a_{n+1}})$, starting with the following elementary lower bound.
\begin{lemma}  \label{lemma_exit_lwrbd}
For all $\delta>0$ and $\zeta > 0$, for all $n \in \N_0$ and $\bP_{a_n} \in \sM_{\gamma,\delta}(a_n)$, $\bP_{a_n}(\tau_{a_{n-1}} < \tau_{a_{n+1}}) \geq \frac 1 3$.
\end{lemma}
\begin{proof}
Working under $\bP_{a_n}\in \sM_{\gamma,\delta}(a_n)$, we define $N_t = M_t + a_n$. In other words, $N_t = Y_t - \delta t$. We define the stopping times
\[\sigma_1= \inf \{ t \geq 0 : N_t = a_{n-1} \}, \quad \sigma_2= \inf \{ t \geq 0 : N_t = a_{n+1} \}.\] 
Since $N_0 = Y_0$ and $N_t < Y_t$ for all $t > 0$, it is immediate that $\tau_{a_{n-1}} \leq \sigma_1$ and $\tau_{a_{n+1}} \geq \sigma_2$. In particular
\begin{equation*}
\bP_{a_n}(\tau_{a_{n-1}} < \tau_{a_{n+1}} ) \geq \bP_{a_n}(\sigma_1 < \sigma_2).
\end{equation*}
Thus, it suffices to prove that $\bP_{a_n}(\sigma_1 < \sigma_2)= \frac 1 3$. Let $\sigma = \sigma_1 \wedge \sigma_2$. Since $\sigma \leq \tau_{a_{n+1}}$, \eqref{eq_QVstop} implies that $\langle Y\rangle_{t \wedge \sigma} \geq  \kappa_n (t \wedge \sigma)$ for all $t \geq 0$ almost surely. 
We recall the Brownian motion $W$ from the Dubins-Schwarz representation of $M_t$ and observe that $N_t = a_n + W_{\langle Y\rangle_t}$. Suppose that $\sigma = + \infty$, so that $N_t \in (a_{n+1}, a_{n-1})$ for all $t \geq 0$. Then our previous observation implies that $\langle Y \rangle_t \to \infty$ as $t  \to \infty$, and hence $W_s \in (a_{n+1},a_{n-1})$ for all $s \geq 0$. This has probability zero, and hence $\sigma$ is a.s.\ finite. 


The claim just shown is equivalent to 
\[ \bP_{a_n}( \sigma_1 < \sigma_2) + \bP_{a_n}( \sigma_1 > \sigma_2) = 1.\]
The result now follows by applying optional stopping to the bounded martingale $(N_{t \wedge \sigma})_{t \geq 0}$.
\end{proof}

Next we obtain an upper bound for $\bP_{a_n}(\tau_{a_{n-1}} < \tau_{a_{n+1}})$ which matches the lower bound up to an exponentially small error. First, we prove a preliminary version, which we then refine to obtain the desired bound. 

\begin{lemma} \label{lemma_exit_upperbd} For all $\delta \in (0,1]$ and $\zeta>0$ satisfying $2\delta\zeta^{1-2\gamma} \leq 1$, and for all $n_0 \in \N$ and $\bP_{a_n} \in \sM_{\gamma,\delta}(a_n)$,
\[ \bP_{a_n}(\tau_{a_{n-1}} < \tau_{a_{n+1}}) \leq \frac 1 3 + e^{-c_3 m} +  m \delta \zeta^{1-2\gamma} 2^{-(1-2\gamma)n}.\]
for all $m \in \N$ satisfying $m \leq \frac 1 2 (\delta \zeta^{1-2\gamma})^{-1} 2^{(1-2\gamma)n}$.
\end{lemma}

\begin{proof}
Let $m \in \N$ be as in the statement and let $T = m a_n^{2-2\gamma}$. Then a short calculation shows that 
\begin{equation} \label{eq:deltaTbd}
    \delta T \leq a_{n+1}.
\end{equation}
By Lemma~\ref{lemma_exit_exp}, for $\bP_{a_n} \in \sM_{\gamma,\delta}(a_n)$,
\[ \bP_{a_n} ( \tau_{a_{n-1}} \wedge \tau_{a_{n+1}} > T) \leq e^{-c_3 m},\]
and hence
\begin{equation} \label{eq_approxmartlemmapf1}
 \bP_{a_n}(\tau_{a_{n-1}} < \tau_{a_{n+1}}) \leq e^{-c_3 m} + \bP_{a_n}(\tau_{a_{n-1}} < \tau_{a_{n+1}}, \tau_{a_{n-1}} \wedge \tau_{a_{n+1}} \leq T).
 \end{equation}
We again work with $N_t = a_n + M_t$, which is a continuous local martingale started from $a_n$, and define the stopping times
\[\sigma_1= \inf \{ t> 0 : N_t = a_{n-1} - \delta T \}, \quad \sigma_2= \inf \{ t \geq 0 : N_t = a_{n+1} - \delta T  \}.\] 
Note that by \eqref{eq:deltaTbd}, we have $a_{n-1} - \delta T \geq a_{n-1} - a_{n+1} > a_n$, and hence $\sigma_1$ is the hitting time of $N$ of a point above $N_0 = a_n$. Observing that $\tau_{a_{n-1}}$ and $\tau_{a_{n+1}}$ are respectively the first times at which $N_t$ hits the curves $t \mapsto a_{n-1} - \delta t$ and $t \mapsto a_{n+1} - \delta t$, we conclude that on the event $\{  \tau_{a_{n-1}} \wedge \tau_{a_{n+1}} \leq T\}$, we have $ \tau_{a_{n-1}} \geq \sigma_1$ and  $\tau_{a_{n+1}} \leq \sigma_2$. This implies that
\begin{align} \label{eq_approxmartlemmapf2}
\bP_{a_n}(\tau_{a_{n-1}} < \tau_{a_{n+1}}, \tau_{a_{n-1}} \wedge \tau_{a_{n+1}} \leq T) &\leq \bP_{a_n}( \sigma_1 < \sigma_2, \tau_{a_{n-1}} \wedge \tau_{a_{n+1}}\leq T) \notag
\\ &\leq \bP_{a_n}( \sigma_1 < \sigma_2 ).
\end{align}
Similarly to the proof of the Lemma~\ref{lemma_exit_lwrbd}, it is straightforward to show that $\sigma_1 \wedge \sigma_2 < \infty$ a.s.\ using a lower bound on $\langle Y\rangle_t$. In particular, since $Y_t = N_t + \delta t$, by definition of $\sigma_1$ we have $Y_t \geq a_{n+1} - \delta T>0$ for $t \leq \sigma_1$, where the strict positivity of the lower bound is a consequence of our assumption on $m$. Writing $N_t = a_n + W_{\langle Y\rangle_t}$, one can then conclude in identical fashion to the proof of the Lemma~\ref{lemma_exit_lwrbd} that $\bP_{a_n}(\sigma_1  \wedge \sigma_2 < \infty) =1$. The martingale convergence theorem applied to the bounded martingale $(N_{t \wedge \sigma_1 \wedge \sigma_2})_{t\geq 0}$ then implies that 
\begin{align*}
(a_{n-1} - \delta T)  \bP_{a_n}(\sigma_1 < \sigma_2 ) + (a_{n+1} - \delta T) (1 - \bP_{a_n}(\sigma_1 < \sigma_2)) = a_n.
\end{align*}
Rearranging and substituting the values of $a_n$ and $T$, we obtain
\begin{align*}
\bP_{a_n}(\sigma_1 < \sigma_2) = \frac{\zeta 2^{-n-1}+ \delta m  \zeta^{2-2\gamma} 2^{-(2-2\gamma)n}}{3 \zeta 2^{-n-1}} = \frac 1 3 + 2 m \delta \zeta^{1-2\gamma} 2^{-(1-2\gamma)n} /3.
\end{align*}
The above combined with \eqref{eq_approxmartlemmapf1} and \eqref{eq_approxmartlemmapf2} implies the desired result.\end{proof}

\begin{lemma} \label{lemma_exit_upperbd2} There are constants $c_4, c_5 > 0$ such that for sufficiently small $\zeta$, for all $\delta \in (0,1]$, $n \in \N_0$ and $\bP_{a_n} \in \sM_{\gamma,\delta}(a_n)$,
\begin{equation}
\bP_{a_n} (\tau_{a_{n-1}} < \tau_{a_{n+1}} ) \leq \frac 1 3 + \zeta^{c_4} \cdot e^{-c_5 n}.
\end{equation}
\end{lemma}
\begin{proof}%
Let $K = \lceil \frac{1-2\gamma}{c_1} \log (1/\zeta) \rceil \in \N$. For $n \in \N_0$ and $\bP_{a_n} \in \sM_{\gamma,\delta}(a_n)$, we will apply Lemma~\ref{lemma_exit_upperbd} with $m_n = K + n$. To do so, we require
\[\frac 1 2 (\delta \zeta^{1-2\gamma})^{-1} 2^{(1-2\gamma)n} \geq m_n.\]
To see that this holds for small enough $\zeta$, we observe that for every $\delta \in (0,1]$, the left hand side above is bounded below by  
\begin{align*}
 \frac{1}{4} \zeta^{-(1-2\gamma)} +  \frac{1}{4} \zeta^{-(1-2\gamma)}  2^{(1-2\gamma)n}.
\end{align*}
It is then immediate that there exists $\zeta_0>0$ such that for all $\zeta \leq \zeta_0$, the first term exceeds $K$, and the second exceeds $n$ (for all $n\in \N_0$), and hence the desired inequality for $m_n$ holds and we may apply Lemma~\ref{lemma_exit_upperbd} with $m = m_n$ for any $n \in \N$. Doing so, we obtain
\begin{align*}\bP(\tau_{a_{n-1}} < \tau_{a_{n+1}}) &\leq \frac 1 3 + e^{-c_3 (K + n)} +  (K  + n )\delta \zeta^{1-2\gamma} 2^{-(1-2\gamma)n}
\\ & \leq \frac 1 3 + \zeta^{1-2\gamma}e^{-c_3 n}  + \zeta^{1-2\gamma}  \left(\frac{1-2\gamma}{c_3} \log (1/\zeta) + n + 1\right) 2^{-(1-2\gamma)n}.
\end{align*}
For any positive $c_4  < 1-2\gamma$ and $c_5 <(\log 2)(1-2\gamma) \wedge c_3$, for sufficiently small $\zeta$ the above is bounded above by $ \frac 1 3 + \zeta^{c_4} e^{-c_5 n}$ for all $n \in \N_0$. Thus, making $\zeta_0$ smaller if necessary, the result is proved.
\end{proof}

\subsection{Some discrete operators and approximate harmonic functions}

We now define a family of discrete second order operators. To do so, we first introduce some notation. We will consider functions $g(n)$ with domains $\{0,1,\dots , N\}$ and $\N_0 := \N \cup \{0\}$. The variable $n$ is always understood to be an integer and we will use e.g. $0 < n < N$ as a shorthand for the integers $n = 1,\dots,N-1$. 

Let $p = (p_n)_{n \in \N_0}$ be a sequence such that $p_n \in [0,1]$ for all $n \in \N_0$. For a real-valued function $g$ with domain $\{0,1,\dots , N\}$ (resp. $\N_0$), we define the discrete operator $\sL_p$ by
\begin{align*} \sL_p g(n) := p_n g(n-1) + (1-p_n) g(n+1) - g(n)
\end{align*}
for $0 < n < N$ (resp. $n>0$). We remark that, if $\Delta$ is the discrete Laplacian, i.e.\ $\Delta g(n) = \frac 1 2 (g(n-1) + g(n+1)) - g(n)$, then
\begin{equation} \label{e_discreteop2}
\sL_p g(n) = \Delta g(n) + \left(p_n - \frac 1 2 \right)( g(n-1) - g(n+1)).
\end{equation}


We will consider the following boundary value problem: for coefficients $(p_n)_{n\in \N}$, $b \in [0,1]$ and $N \in \N$,
\begin{equation}\label{eq_diff_boundary}
		\left\lbrace
		\begin{aligned}
			 &-\sL_p g(n) = 0, \quad 0 < n < N
            \\  &g(0) = 1, \qquad
             g(N) = b. 
		\end{aligned}
		\right.
	\end{equation}
If the equation $-\sL_p g(n) = 0$ is replaced by the inequalities $-\sL_p g(n) \leq 0$ and $-\sL_p g(n) \geq 0$, $g$ is respectively called a subsolution or supersolution to \eqref{eq_diff_boundary}.

Heuristically, we are interested in \eqref{eq_diff_boundary} when $p_n =  \bP_{a_n}(\tau_{a_{n-1}} < \tau_{a_{n-1}})$, where $\bP_{a_n}$ is ``the law of $Y$ started from $a_n$". In this case, the natural candidate for a solution to \eqref{eq_diff_boundary} is $n \mapsto \bP_{a_n} (\tau_{a_0} < \tau_{a_N})$. However, as there is no unique law $\bP_{a_n}$, this choice does not make sense, which complicates the problem of finding an appropriate solution to \eqref{eq_diff_boundary}. We resolve this in the following way: in the previous section, we have obtained uniform bounds on exit probabilities of the form $\bP_{a_n}(\tau_{a_{n-1}} < \tau_{a_{n-1}})$ which hold uniformly for all $\bP_{a_n} \in \sM_{\gamma,\delta}(a_n)$. We will consider the equation \eqref{eq_diff_boundary} with coefficients $p_n$ given by these upper and lower bounds. Combined with the pseudo-strong Markov property, this will allow us to treat the problem almost as though we were working with a Markov process. Instead of exact $\sL_p$-harmonic functions, we work with certain approximate harmonic functions which we now introduce.

We remind the reader that $a_n = \zeta 2^{-n}$, and hence $a_0 = \zeta$. For $N \in \N$ and $ 0 \leq n \leq N$, we define
\begin{align} \label{def:HN}
&H^N_{\zeta,\delta}(n) := \sup_{\bP \in \sM_{\gamma,\delta}(a_n)} \bP(\tau_{a_0} < \tau_{a_N}).
\end{align}
We note that 
$H^N_{\zeta,\delta}(0) = 1$ and $H^N_{\zeta,\delta}(N) = 0$ by definition. We also define, for $n \in \N_0$,
\begin{align} \label{def:H}
&H_{\zeta,\delta}(n) := \sup_{\bP \in \mathscr{M}_{\gamma,\delta}(a_n)} \bP(\tau_{a_0} < \tau_0),
\end{align}
and observe that $H_{\zeta,\delta}(0) = 1$. As noted, we cannot define $\sL_{p}$ with coefficients $p_n =  \bP_{a_n}(\tau_{a_{n-1}} < \tau_{a_{n-1}})$ directly, because they are not uniquely defined. We will instead work with coefficients corresponding to the upper and lower bounds for $p_n$ obtained in Lemmas~\ref{lemma_exit_lwrbd} and \ref{lemma_exit_upperbd2}. For $n \in \N_0$, we define
\begin{equation}
\op_n = \frac 1 3 + \zeta^{c_4} e^{-c_5 n}.
\end{equation}
We will consider the operator $\sL_{\op}$ associated to these coefficients, as well as $\sL_{1/3}$, where the latter corresponds to the case $p_n = 1/3$ for all $ n \in \N_0$. 

Ultimately it is $H_{\zeta,\delta}$ which is of interest to us, because to obtain the upper bound stated in Proposition~\ref{prop_hittingprob_main}, it suffices to prove that $H_{\zeta,\delta}(1)$ is bounded above by the right-hand side of \eqref{eq_hittingprob_main}. First, we establish the following.

\begin{lemma} \label{lemma_HN_limit} For any $\delta,\zeta >0$, for every $n \in \N$, $\lim_{N \to \infty} H^N_{\zeta,\delta}(n) = H_{\zeta,\delta}(n)$.
\end{lemma}
\begin{proof} Fix $n \in \N$, and for arbitrary $\bP \in \sM_{\gamma,\delta}(a_n)$ consider the mapping $N \mapsto \bP(\tau_{a_0} < \tau_{a_N})$. Since $Y_0 = a_n > a_N$ for all $N > n$, continuity of the sample paths of $Y$ implies that $\tau_{a_N}$ is an increasing family of random variables. The monotonicity of the stopping times implies that there exists a (possibly infinite) random variable $\tau^* \in \R_+ \cup \{+\infty\}$ such that $\tau_{a_N} \uparrow \tau^*$. Since $\tau_{a_N} < \tau_0$ for all $N > n$, we have $\tau^* \leq \tau_0$; moreover, $Y_{\tau_{a_N}} = a_N$ for all $N$, and so continuity of the sample paths of $Y$ implies that $Y_{\tau^*} = 0$ when $\tau^* < \infty$. We therefore conclude that $\tau^* = \tau_0$ a.s., and hence by continuity of measure we have
\begin{equation} \label{eq:probNlimit}
\lim_{N \to \infty} \bP(\tau_{a_0} < \tau_{a_N}) = \bP(\tau_{a_0} < \tau_0).
\end{equation}
The result now follows from \eqref{def:HN} and \eqref{def:H}.
\end{proof}


Our next goal is to establish that $H^N_{\zeta,\delta}$ is a subsolution to $\sL_{\op} g = 0$ on $\{0,\dots,N\}$. We start with a monotonicity property.

\begin{lemma} \label{lemma_harmonic_mono} For every $N \in \N$, $\zeta >0$ and $\delta>0$, $n \mapsto H^N_{\zeta,\delta}(n)$ is non-increasing on $\{0,\dots,N\}$.
\end{lemma}
\begin{proof} We observe that the result is trivial for $N = 1$ and $N=2$, so let $N \geq 3$. Let $ n \leq N-2$ and $\bP_{a_{n+1}} \in \sM_{\gamma,\delta}(a_{n+1})$. Then
\begin{align*}
\bP_{a_{n+1}}(\tau_{a_0} < \tau_{a_N}) &= \bP_{a_{n+1}} (\tau_{a_n} < \tau_{a_N}, \tau_{a_0} < \tau_{a_N}) + \bP_{a_{n+1}}(\tau_{a_N} < \tau_{a_n}, \tau_{a_0} < \tau_{a_N}) 
\\ &=   \bP_{a_{n+1}} (\tau_{a_n} < \tau_{a_N}, \tau_{a_0} < \tau_{a_N})
\end{align*}
where the second probability vanishes because, recalling that $Y_0 = a_{n +1} < a_n < a_0$, we must have $\bP_{a_{n+1}}(\tau_{a_0} < \tau_{a_N} < \tau_{a_n}) = 0$. Applying Lemma~\ref{lemma_stop_markov} to the above, we obtain that for some $\bP_{a_n} \in \sM_{\gamma,\delta}(a_n)$,
\[ \bP_{a_{n+1}}(\tau_{a_0} < \tau_{a_N}) = \bP_{a_{n+1}} (\tau_{a_n} < \tau_{a_N})\cdot \bP_{a_n}(\tau_{a_0} < \tau_{a_N}), \]
and hence 
\[  \bP_{a_{n+1}}(\tau_{a_0} < \tau_{a_N}) \leq \bP_{a_{n+1}} (\tau_{a_n} < \tau_{a_N}) H^N_{\zeta,\delta}(n) \leq H^N_{\zeta,\delta}(n).\]
Taking the supremum over $\bP_{a_{n+1}} \in\sM_{\gamma,\delta}(a_{n+1})$ gives $H^N_{\zeta,\delta}(n+1) \leq H^N_{\zeta,\delta}(n)$, proving the result.
\end{proof}

We now verify that $H^N_\zeta$ is a subsolution to the equation associated to the operator $\sL_{\op}$. 
\begin{lemma}\label{lemma_HN_subsolution} For sufficiently small $\zeta>0$, the following holds for any $\delta \in (0,1]$: for any $N \in \N$, \[- \sL_{\op} H^N_{\zeta,\delta}(n) \leq 0 \, \text{ for all } \,0 < n < N.\] In particular, $H^N_{\zeta,\delta}$ is a subsolution to \eqref{eq_diff_boundary} with boundary condition $H^N_{\zeta,\delta}(0) = 1$ and $H^N_{\zeta,\delta}(N) = 0$.
\end{lemma} 
\begin{proof} Let $N \geq 2$ and suppose that $\zeta_0>0$ is small enough so that Lemma~\ref{lemma_exit_exp} holds for all $\zeta \in (0,\zeta_0]$. It is immediate that $H^N_{\zeta,\delta}$ satisfies the desired boundary condition, as we have already remarked upon this. Let $0 < n < N$ and $\bP_{a_n} \in \sM_{\gamma,\delta}(a_n)$. To begin, we partition on the point via which $Y$ exits $I_n$, i.e.
\begin{align*}
\bP_{a_n}(\tau_{a_0} < \tau_{a_N}) = \bP_{a_n}(\tau_{a_{n+1}} < \tau_{a_{n-1}}, \tau_{a_0} < \tau_{a_N}) +\bP_{a_n}(\tau_{a_{n-1}} < \tau_{a_{n+1}}, \tau_{a_0} < \tau_{a_N}). \end{align*}
We next apply Lemma~\ref{lemma_stop_markov} to each term in the above. This yields that for some $\bP_{a_{n-1}} \in \sM_{\gamma,\delta}(a_{n-1})$ and $\bP_{a_{n+1}} \in \sM_{\gamma,\delta}(a_{n+1})$,
\begin{align*}
\bP_{a_n}(\tau_{a_0} < \tau_{a_N}) = \bP_{a_n}(\tau_{a_{n-1}} < \tau_{a_{n+1}})\cdot \bP_{a_{n-1}}(\tau_{a_0} < \tau_{a_N}) + \bP_{a_{n}}(\tau_{a_{n+1}} < \tau_{a_{n-1}}) \cdot \bP_{a_{n+1}}(\tau_{a_0} < \tau_{a_N}), 
\end{align*}
and hence, by \eqref{def:HN},
\begin{align*}
\bP_{a_n}(\tau_{a_0} < \tau_{a_N}) \leq p \cdot H^N_{\zeta,\delta}(n-1) + (1-p)\cdot H^N_{\zeta,\delta}(n+1),
\end{align*}
where we write $p = \bP_{a_n}(\tau_{a_{n-1}} < \tau_{a_{n+1}})$. Since $\delta \leq 1$ and $\zeta \leq \zeta_0$, we have $p \leq \op_n$ by Lemma~\ref{lemma_exit_upperbd2}. By Lemma~\ref{lemma_harmonic_mono}, $H^N_\zeta(n+1) \leq H^N_\zeta(n-1)$. Together these imply that
\[ \bP_{a_n}(\tau_{a_0} < \tau_{a_N}) \leq \op_n \cdot H^N_{\zeta,\delta}(n-1) + (1-\op_n)\cdot H^N_{\zeta,\delta}(n+1).\]
Since the above holds for any $\bP \in \sM_{\gamma,\delta}(a_n)$, we may take the supremum over $\sM_{\gamma,\delta}(a_n)$ to obtain the same bound with the left-hand side replaced by $H^N_{\zeta,\delta}(n)$. Rearranging the resulting inequality gives $-\sL_{\op} H^N_{\zeta,\delta}(n) \leq 0$, which is the desired result.
\end{proof}

Next, we construct an appropriate supersolution to \eqref{eq_diff_boundary} with operator $\sL_{\op}$. For $K>0$, we define the function $q_K : \N_0 \to \R$ by
\[ q_K(n) = c_K \left( \frac{K + n}{K + n+1}\right)2^{-n} = c_K \left( 2^{-n} - \frac {1}{K+ n+1} 2^{-n}\right),\]
where $c_K = (K+1)/K$. Recall that $\op_n = \tfrac 1 3 + \zeta^{c_4} e^{-c_5 n}$. Given $\zeta > 0$, we define $K_\zeta$ by
\[ K_\zeta := \zeta^{-c_4/2}/4.\] 
\begin{lemma} \label{lemma_supersolution}
For sufficiently small $\zeta >0$, $- \sL_{\bar{p}} q_{K_\zeta}(n) \geq 0$ for all $n \in \N$. 
\end{lemma} 
\begin{proof}
The proof follows by a direct calculation. For the time being let $K>0$ be arbitrary. Then for $n \in \N$,
\begin{align*}
\sL_{\bar{p}} q_K(n) &= \left( \frac 1 3 +  \zeta^{c_4} e^{-c_5 n} \right) q_K(n-1)+  \left( \frac 2 3 -  \zeta^{c_4} e^{-c_5 n} \right)  q_K(n+1) - q_K(n)
\\ &= c_K \bigg[ \frac 1 3 \left(\frac{K+n-1}{K+ n} \right) 2^{-n+1} + \frac 2 3 \left( \frac{K+n+1}{K+n+2} \right)2^{-n-1}  - \left( \frac{K+n}{K+n+1} \right)2^{-n} \bigg]
\\ &\hspace{3 cm} + c_K \zeta^{c_4} e^{-c_5 n} \bigg[ \left(\frac{K+n-1}{K+ n} \right) 2^{-n+1}  - \left( \frac{K+n}{K+n+1} \right)2^{-n-1} \bigg] 
\\ &= \frac{c_K 2^{-n}}{3} \bigg[ 2 \left(\frac{K+n-1}{K+ n} \right) +  \left( \frac{K+n+1}{K+n+2} \right)  - 3\left( \frac{K+n}{K+n+1} \right) \bigg]
\\ &\hspace{3 cm} + c_K 2^{-n}  \zeta^{c_4} e^{-c_5 n} \bigg[ 2\left( \frac{K+n-1}{K+ n} \right)   - \frac 1 2 \left( \frac{K+n}{K+n+1} \right) \bigg].
\end{align*}
We remark that the last square-bracketed term in the above is at most $2$. Elementary manipulations of the first term of the right-hand side above then yield that
\begin{align*}
    \sL_{\bar{p}} q(n) & \leq c_K 2^{-n} \left[ \frac{-(K+n + 4)}{3(K+n)(K+n+1)(K+n+2) } + 2  \zeta^{c_4} e^{-c_5 n} \right]
    \\ &\leq  c_K 2^{-n} \left[ -\frac 1 3(K+n + 4)^{-2}+ 2  \zeta^{c_4} e^{-c_5 n} \right].
\end{align*}
Now we choose $K = \zeta^{-c_4/2}/4$. For sufficiently small $\zeta$ we have $\zeta^{-c_4/2}/4 + n + 4 \leq \zeta^{-c_4/2}/3 +n$ for all $n \in \N_0$, and hence for such $\zeta$ we have
\begin{align*}
    \sL_{\bar{p}} q(n) &\leq c_K 2^{-n} \left[ -\frac 1 3(\zeta^{-c_4/2}/3 +n)^{-2} + 2 \zeta^{c_4} e^{-c_5 n} \right], \quad n \in \N_0.
\end{align*}
To complete the proof it suffices to prove that $f(x) \geq g(x)$ for all $x \geq 0$, where
\[ f(x) = \frac 1 3(\zeta^{-c_4/2}/3 +x)^{-2}, \quad g(x) = 2 \zeta^{c_4} e^{-c_5 x}.\]
This will follow if we show that $f(0) \geq g(0)$ and $f/g$ is increasing. The first property is immediate. For the second, we note that
\begin{align*}
\left(\frac{f}{g}\right)'(x) = \frac{e^{c_5x}}{6\zeta^{c_4}}\left(\frac{c_5 - 2(\zeta^{-c_4/2}/3 + x)^{-1}}{(\zeta^{-c_4/2}/3 + x)^2} \right) \geq  \frac{e^{c_5x}}{6\zeta^{c_4}}\left(\frac{c_5 - 6\zeta^{c_4/2}}{(\zeta^{-c_4/2}/3 + x)^2} \right),
\end{align*}
where the last inequality holds for all $x \geq 0$. The right-hand side above is positive provided $\zeta^{c_4/2} < c_5 / 6$, and hence $f/g$ is increasing as desired for $\zeta$ sufficiently small. This completes the proof.
\end{proof}

We are now able to complete the proof of (the upper bound from) Proposition~\ref{prop_hittingprob_main}.

\begin{proof}[Proof of Proposition~\ref{prop_hittingprob_main}] Let $\zeta_0 >0$ be small enough so that we can apply Lemmas~\ref{lemma_HN_subsolution} and \ref{lemma_supersolution} for all $\delta \in (0,1]$, $\zeta \in (0,\zeta_0]$ and $n \in \N$. Let $N \in \N$. By Lemma~\ref{lemma_HN_subsolution}, $H^N_{\zeta,\delta}$ is a subsolution to \eqref{eq_diff_boundary} with boundary conditions $H^N_{\zeta,\delta}(0) = 1$ and $H^N_{\zeta,\delta}(N) = 0$. By Lemma~\ref{lemma_supersolution}, $q_{K_\zeta}$ is a supersolution to \eqref{eq_diff_boundary} with boundary conditions $q_{K_\zeta}(0) = 1$ and $q_{K_\zeta}(N) > 0$. Hence, by the comparison principle,
\[ H^N_{\zeta,\delta}(n) \leq q_{K_\zeta}(n), \quad 0 \leq n \leq N.\] 
For fixed $n$, the above holds for all $N > n$. Thus, taking $N \to \infty$, it follows from Lemma~\ref{lemma_HN_limit} that $H_{\zeta,\delta} (n) \leq q_{K_\zeta}(n)$ for all $n \in \N$.

To complete the proof, we compute
\begin{align*}
 H_{\zeta,\delta}(1) \leq q_{K_\zeta}(1) &=\frac 1 2 \left( \frac{\zeta^{-c_4/2}/4 + 1}{\zeta^{-c_4/2}/4}\right). \left(  \frac{\zeta^{-c_4/2}/4 + 1}{\zeta^{-c_4/2}/4+2} \right)
\\ & = \frac 1 2 + \frac 1 2 \left( \frac{1}{\zeta^{-c_4}/16+ \zeta^{-c_4/2}/2} \right)
\\ & \leq \frac 1 2 + 8 \zeta^{c_4}.
\end{align*}
Recall that $a_0 = \zeta$. By \eqref{def:H}, this implies that for any $\zeta \in (0,\zeta_0]$, for all $\bP_{\zeta/2} \in \sM_{\gamma,\delta}(\zeta/2)$, 
\[ \bP_{\zeta/2} (\tau_\zeta < \tau_0) \leq \frac 1 2 + 8 \zeta^{c_4}.\]
Taking $c_1 = c_4$, this completes the proof.
\end{proof}

We end this section with a brief discussion of the lower bound in \eqref{eq_hittingprob_main}. While the lower bound is less important for the proof of the main result, it is still used, and moreover it is useful to compare it with the upper bound. The proof of the lower bound uses the same ideas as the proof of the upper bound but is simpler. Thanks to Lemma~\ref{lemma_exit_lwrbd}, the ``transition probabilities" satisfy $\bP_{a_n}(\tau_{a_{n+1}} < \tau_{a_{n-1}}) \geq \tfrac 1 3$ for all $n \in \N$ and $\bP_{a_n} \in \sM_{\gamma,\delta}(a_n)$. Similarly to above, we can define another family of approximate harmonic functions $h^N_{\zeta,\delta}$ and $h_{\zeta,\delta}$ by replacing the supremum with an infimum. A parallel argument then shows that $h_{\zeta,\delta}^N$ is a supersolution (with boundary data $h^N_{\zeta,\delta}(0) = 1$ and $h^N_{\zeta,\delta}(N) = 0$) of \eqref{eq_diff_boundary} with the operator $\sL_{1/3}$, which is obtained by taking $p_n = \tfrac 1 3$ for all $n \in \N$. Because $\sL_{1/3}$ has constant coefficients, its harmonic functions can be computed directly, and a straightforward comparison argument shows that $h_{\zeta,\delta}(n) \geq  2^{-n}$ for all $n$.

\subsection{Exponential tails for the exit time of an interval}


We now turn our attention to the distribution of $\tau_0 \wedge \tau_{2\zeta}$ under $\bP_\zeta \in \sM_{\gamma,\delta}(\zeta)$; in fact, the bounds that we show hold for all $\bP\in \sM_{\gamma,\delta}(\zeta)$. In particular, we will prove Proposition~\ref{prop_exit_exp}, which establishes that mean of $\tau_0 \wedge \tau_{2\zeta}$ is of order $\zeta^{2-2\gamma}$, and its tail decays exponentially along multiples of its mean.


We first prove the following bound on the expected value of $\tau_0 \wedge \tau_{2\zeta}$. We recall the definition of $\sM_{\gamma,\delta}((0,2\zeta))$ as in \eqref{eq:M_geninit}.

\begin{lemma} \label{lemma_exit_mean}
There exists a constant $C_3 \geq 1$ such that for sufficiently small $\zeta$, for all $\delta \in (0,1]$ and $\bP \in \sM_{\gamma,\delta}((0,2\zeta))$,
\[ \E[\tau_0 \wedge \tau_{2\zeta}] \leq C_3 \zeta^{2-2\gamma}.\]
\end{lemma}

Proposition~\ref{prop_exit_exp} follows easily from this lemma, as we now show.

\begin{proof}[Proof of Proposition~\ref{prop_exit_exp}]
Let $\zeta>0$ be small enough so that Lemma~\ref{lemma_exit_mean} holds, and let $\delta \in (0,1]$. It follows from Lemma~\ref{lemma_exit_mean} and Markov's inequality that
\begin{equation} \label{eq_meanexit_impl}
\bP' \in \sM_{\gamma,\delta}((0,2\zeta)) \quad \implies \quad \bP'(\tau_0 \wedge \tau_{2\zeta} > 2C_3 \zeta^{2-2\gamma}) \leq \frac 1 2.
\end{equation}
Fix $\bP \in \sM_{\gamma,\delta}((0,2\zeta))$, let $m \in \N$, and define $A = \{\tau_0 \wedge \tau_{2\zeta} > (m-1) \cdot 2C_3 \zeta^{2-2\gamma} \}$. For $\omega \in A$, we have 
\[ (\tau_0 \wedge \tau_{2\zeta}) \circ \theta_{(m-1) \cdot 2C_3 \zeta^{2-2\gamma} }(\omega) = (\tau_0 \wedge \tau_{2\zeta})(\omega) - (m-1)\cdot 2C_3 \zeta^{2-2\gamma}.\]
We then compute
\begin{align*}
&\bP(\tau_0 \wedge \tau_{2\zeta} > m \cdot 2C_3 \zeta^{2-2\gamma})
\\ &\hspace{1 cm} =\E\left[\indc(\tau_0 \wedge \tau_{2\zeta} > (m-1) \cdot 2C_3 \zeta^{2-2\gamma}) \cdot \Q_{(m-1) \cdot 2C_3 \zeta^{2-2\gamma}}(\tau_0 \wedge \tau_{2\zeta} > 2C_3 \zeta^{2-2\gamma}) \right]
\\ &\hspace{1 cm} \leq \E\left[\indc(\tau_0 \wedge \tau_{2\zeta} > (m-1) \cdot 2C_3 \zeta^{2-2\gamma}) \right] \cdot \frac 1 2.
\end{align*}
The final inequality holds by \eqref{eq_meanexit_impl}, because for $\omega \in A$, by the PSMP (Proposition~\ref{prop_markov}), $\Q^\omega_{(m-1) \cdot 2C_3 \zeta^{2-2\gamma}} \linebreak \in \sM_{\gamma,\delta}((0,2\zeta))$. Iterating the argument above, we obtain that 
\begin{equation*}
\bP(\tau_0 \wedge \tau_{2\zeta} > m \cdot 2C_3 \zeta^{2-2\gamma})  \leq 2^{-m},
\end{equation*}
which implies the claimed result. \end{proof}

To complete the proof of Proposition~\ref{prop_exit_exp}, we need only prove Lemma~\ref{lemma_exit_mean}. We make one further reduction, arguing that Lemma~\ref{lemma_exit_mean} is itself a consequence of the following weaker tail estimate for $\tau_0 \wedge \tau_{2\zeta}$. 

\begin{proposition} \label{prop_stop_tail} There exist constants $C_4 \geq 1$ and $c_6> 0$ such that for sufficiently small $\zeta$, for all $\delta \in (0,1]$, $\bP \in \sM_{\gamma,\delta}((0,2\zeta))$ and $m \in \N$, 
\begin{equation}
\bP( \tau_0 \wedge \tau_{2\zeta} > m \zeta^{2-2\gamma}) \leq C_4 e^{-c_6 \sqrt{m}}.
\end{equation}
\end{proposition}

This trivially implies Lemma~\ref{lemma_exit_mean}, since for $\bP \in \sM_{\gamma,\delta}((0,2\zeta))$,
\begin{align*}
\E[\tau_0 \wedge \tau_{2\zeta}] = \zeta^{2-2\gamma} \int_0^\infty \bP(\tau_0 \wedge \tau_{2\zeta} > \lambda \zeta^{2-2\gamma})d \lambda &\leq \zeta^{2-2\gamma} \sum_{m=0}^\infty \bP( \tau_0 \wedge \tau_{2\zeta} > m \zeta^{2-2\gamma})
\\ &\leq C_4 \zeta^{2-2\gamma} \sum_{m=0}^\infty e^{-c_6 \sqrt m}
\\ & =: C_3 \zeta^{2-2\gamma}.
\end{align*}

The remainder of the section is devoted to the proof of Proposition~\ref{prop_stop_tail}. The proof is based on the joint analysis of the ``random walk" obtained by stopping $Y$ along the discrete set of points $\{a_n : n \in \N_0\}$, and the amount of time $Y$ spends between its visits to these points. This section makes heavy use of the PSMP, which is stated in Proposition~\ref{prop_markov}. We begin by introducing some notation.

For $\zeta > 0$, we define the set $\cA := \{\zeta 2^{-n} : n \in \Z\} = \{a_n : n \in \Z\}$. Suppose $\bP \in \sM_{\gamma,\delta}$. Let $\sigma_0 = \inf \{ t \geq 0 : Y_t \in \cA\}$ and for $k \in \N$ recursively define
\[ \sigma_k := \inf \left\{t > \sigma_{k-1} : Y_t \in \cA \backslash \{ Y_{\sigma_{k-1}} \} \right\}.\]
Whenever $\bP \in \sM_{\gamma,\delta}(a)$ for some $a \in \cA$, we have $\sigma_0 = 0$ a.s., and we will frequently use this fact without explicit reference. However, we will eventually require more general initial conditions for $Y$, in which case $\sigma_0$ may be non-zero with positive probability. It is not hard to argue using the PSMP and \eqref{eq_QVstop} that $\sigma_k < \infty$ a.s.\ for each $k$. The sequence of stopping times $(\sigma_k)_{k \in \N}$ are the sequence of times at which $Y$ visits $\cA$, ignoring consecutive visits to the same point.

The walk $S = (S_k)_{k \in \N_0}$  is defined as the height, on a logarithmic scale, of $Y$ at the times $(\sigma_k)_{k \in \N_0}$. That is,
\[ S_k = m \,\,\, \iff \,\,\, Y_{\sigma_k} = a_m.\]
This is equivalent to $S_k = - \log_2(Y_{\sigma_k} / \zeta)$. Because $Y$ is continuous, it follows that $S$ is restricted to $\pm 1$ increments. We remark that $S = S(\omega)$, and that shifting $\omega$ by $\sigma_j$ induces a shift by $j$ on $S$, i.e.\ for $j,k \in \N_0$, $S_k(\omega \circ \sigma_j) = S_{k+j}(\omega)$. This is immediate from the definition above, and it also extends to random times associated to the process $S$ as well. Because this is a discrete time process and there are very few technicalities, to ease notation we will not formally introduce the associated filtrations. We call an $\N_0$-valued random variable an $S$-stopping time if it is a stopping time associated to the filtration generated by $S$. It is elementary to argue that if $\Gamma$ is an $S$-stopping time, then $\sigma_\Gamma$ is an $(\sF_t)$-stopping time. If $\Gamma$ is an $S$-stopping time and $\Gamma(\omega) < \infty$, then
\begin{equation} \label{eq_S_shift}
S_k(\omega \circ \theta_{\sigma_\Gamma}) = S_{k+\Gamma(\omega)}(\omega).
\end{equation}

Next, we express $\tau_0 \wedge \tau_{2\zeta}$ in terms of the inter-visit times of $Y$ to $\cA$. Note that $\tau_{2\zeta} = \tau_{a_{-1}}$. 
For $N \in \N$, let $J_N := \inf \{n : S_n \in \{-1, N\}\}$, and let $J := \inf \{n : S_n = -1\}$. Then for $N \in \N$, if $Y_0 \in ( a_N,2\zeta)$, we have
\begin{equation*}
\tau_{a_N} \wedge \tau_{2\zeta} = \sigma_0 +  \sum_{i=0}^{ J_N - 1} (\sigma_{i + 1} - \sigma_i).
\end{equation*}
Recall that $\tau_0 \wedge \tau_{2\zeta} = \lim_{N \to \infty} \tau_{a_N} \wedge \tau_{2\zeta}$. On the other hand, $J_N \uparrow J$ a.s.\ as $N \to \infty$. Taking $N \to \infty$ in the above, we obtain that for any $a \in (0,\zeta)$, with with probability one on $\{Y_0 \in [a,2\zeta)\}$ we have
\begin{equation} \label{eq_exitrep1}
\tau_0 \wedge \tau_{2\zeta} = \sigma_0 + \sum_{i=0}^{ J - 1} (\sigma_{i + 1} - \sigma_i).
\end{equation}
In particular, if $\bP(Y_0 \in (0,2\zeta)) = 1$, then the representation above holds almost surely. We remark that the sum may have infinitely many terms, since $J$ is infinite with positive probability. 

Next, we define the local time of $S$ at level $k$ and time $m$ by
\begin{equation}
\ell^k_m = \sum_{j=0}^m \indc(S_j = k),
\end{equation}
For all $k \in \N_0$ we set $\Gamma^k_0 = - 1$, and for $k \in \N_0$ and $i \in \N$ we define random times 
\begin{align*}
&\Gamma^k_i =  \inf \{ n > \Gamma^k_{i-1} : S_n = k \}.
\end{align*}
That is, $\Gamma^k_i$ is the $i$th time that $S$ visits $k$, which corresponds to the $i$th (non-consecutive) visit of $Y$ to $a_k$. Next, we rewrite the sum in \eqref{eq_exitrep1} by grouping terms which correspond to visits to the same site. This yields
\begin{equation} \label{eq_tau0_decomp}
\tau_0 \wedge \tau_{2\zeta } = \sigma_0 +\sum_{k=0}^\infty \sum_{i=1}^{\infty} \indc(\Gamma^k_i < J ) (\sigma_{\Gamma^k_i + 1} - \sigma_{\Gamma^k_i}).
\end{equation}
We note as well that $\Gamma^k_i < J$ if and only if $\ell^k_{J - 1} \geq i$, so \eqref{eq_tau0_decomp} is equivalent to
\begin{align} \label{eq_hitting_Xk_decomp}
\tau_0 \wedge \tau_{2\zeta} &= \sigma_0 +\sum_{k=0}^\infty \sum_{i=1}^{\ell^k_{J- 1}} \sigma_{\Gamma^k_i + 1} - \sigma_{ \Gamma^k_i} \notag
\\ &=:\sigma_0 + \sum_{k=0}^\infty X_k
\end{align}
where $X_k$ is defined to equal zero when $\ell^k_{J - 1} = 0$. The next two lemmas concern $X_k$. First, we observe how $X_k$ behaves under shift operations. It follows from its definition in \eqref{eq_hitting_Xk_decomp} that for any $k \in \N_0$,
\begin{align} \label{eq_Xk_shift}
&\text{If $l \in \N_0$ and $\omega \in \Omega$ satisfy $l \leq \Gamma^k_1(\omega)$, then} \,\,X_k(\omega) = X_k (\omega \circ \theta_{\sigma_l}). 
\end{align}

\begin{lemma} \label{lemma_X_shift} For all $\zeta>0$ and $\delta>0$, for any $\bP \in \sM_{\gamma,\delta}$, $k \in \N_0$ and $r \geq 0$,
\begin{equation*}
\bP(X_k > r) = \bP(\Gamma^k_1< J) \cdot \bP_{a_k}(X_k > r)
\end{equation*}
for some $\bP_{a_k} \in \sM_{\gamma,\delta}(a_k)$.
\end{lemma}
\begin{proof} If $\bP \in \sM_{\gamma,\delta}(a_k)$, then one has $\Gamma^1_k = 0$ a.s.\ and the result holds with $\bP_{a_k} = \bP$. Now take a general $\bP \in \sM_{\gamma,\delta}$. Then by \eqref{eq_Xk_shift} and the fact that $X_k = 0$ if $\Gamma^k_1  \geq J$, we have
\[ X_k(\omega) =  \indc(\Gamma^k_1(\omega) < J) X_k(\omega \circ \theta_{\sigma_{\Gamma^k_1}}),\]
and hence
\begin{align*}
\bP(X_k > r)& = \bP(\Gamma^1_k < J) \cdot \bP(X_k(\omega \circ \theta_{\sigma_{\Gamma^k_1}}) > r\, | \, \Gamma^1_k < J)
\\ &= \bP(\Gamma^1_k < J) \cdot  \bP^{\{\Gamma^1_k < J\}} \circ \theta_{\sigma_{\Gamma^1_k}} (X_k > r).
\end{align*}
By the PSMP (i.e.\ Proposition~\ref{prop_markov}(b)), $\bP^{\{\Gamma^1_k < J\}} \circ \theta_{\sigma_{\Gamma^1_k}} \in \sM_{\gamma,\delta}(a_k)$. The result follows. \end{proof}

\begin{lemma} \label{lemma_Xk_recursion} For sufficiently small $\zeta>0$, for all $\delta \in (0,1]$, $\bP \in \sM_{\gamma,\delta}(x)$, $k \in \N_0$ and $R > r \geq 0$,
\begin{equation}
\bP(X_k > R) \leq \bP_{a_k}'( \tau_{a_{k-1}} \wedge \tau_{a_{k-1}} > r) + \frac 3 4 \cdot \bP_{a_k}''(X_k > R-r)
\end{equation}
for some $\bP_{a_k}', \bP_{a_k}'' \in \sM_{\gamma,\delta}(a_k)$.
\end{lemma}
\begin{proof}
Let $\zeta>0$ be sufficiently small so that Lemma~\ref{lemma_exit_upperbd2} can be applied for all $\delta \in (0,1]$ and $n \in \N_0$. First, suppose that $\bP \in \sM_{\gamma,\delta}(a_k)$.  In this case, $\Gamma^k_1 = \sigma_0 = 0$, and so
\[ \sigma_{\Gamma^k_1 + 1} - \sigma_{\Gamma^k_1} = \sigma_1 - \sigma_0 = \tau_{a_{k-1}} \wedge \tau_{a_{k+1}}.\]
From \eqref{eq_hitting_Xk_decomp}, it follows that
\begin{align}
X_k(\omega) = \tau_{a_{k-1}} \wedge \tau_{a_{k+1}}(\omega) + \sum_{i=2}^{\ell^k_{J - 1}(\omega)} (\sigma_{\Gamma^k_i + 1} - \sigma_{\Gamma^k_i})(\omega). \notag
\end{align}
Next we observe that for $i \geq 1$, $\Gamma^k_i(\omega \circ \theta_{\sigma_1}(\omega)) = \Gamma^k_{i+1}(\omega) - 1$, since the $i$th visit of $S \circ \theta_{\sigma_1}$ to $k$ is the ($i + 1$)th visit of $S$ to $k$, because $S_0 = k$ under $\bP$. In particular, 
\begin{align*}
X_k(\omega) &= \tau_{a_{k-1}} \wedge \tau_{a_{k+1}}(\omega) + \sum_{i=1}^{\ell^k_{J - 1}(\omega \circ \theta_{\sigma_1}(\omega))} (\sigma_{\Gamma^k_i + 1} - \sigma_{\Gamma^k_i})(\omega \circ \theta_{\sigma_1}(\omega)) \notag
\\ &=  \tau_{a_{k-1}} \wedge \tau_{a_{k+1}}(\omega) + X_k(\omega \circ \theta_{\sigma_1}(\omega)).
\end{align*}
In the above, we have used our convention that $X_k$ and the sum defining it are zero if $\ell^k_{J-1} = 0$. It follows that 
\begin{equation} \label{eq_Xkrec1}
\bP(X_k > R) = \bP(\tau_{a_{k-1}} \wedge \tau_{a_{k+1}} > r) + \bP \circ \theta_{\sigma_1} (X_k > R-r).
\end{equation}
To handle the second term, we partition on $\{\sigma_1 = \tau_{a_{k+1}}\}$, which we note is the same as $\{\tau_{a_{k+1}} < \tau_{a_{k-1}}\}$. By the PSMP,
\begin{align*}
\bP \circ \theta_{\sigma_1} (X_k > R-r) &= \bP^{\{\sigma_1 = \tau_{a_{k+1}}\}} \circ \theta_{\sigma_1} (X_k > R-r)  \cdot \bP(\tau_{a_{k+1}} < \tau_{a_{k-1}})
\\ &\quad + \bP^{\{\sigma_1 = \tau_{a_{k-1}}\}} \circ \theta_{\sigma_1} (X_k > R-r) \cdot \bP(\tau_{a_{k-1}} < \tau_{a_{k+1}})
\\ &= \bP^1(X_k > R-r) \cdot \bP(\tau_{a_{k+1}} < \tau_{a_{k-1}})
\\ &\quad +  \bP^2 (X_k > R-r) \cdot \bP(\tau_{a_{k-1}} < \tau_{a_{k+1}}),
\end{align*}
where $\bP^1_{a_{k+1}} \in \sM_{\gamma,\delta}(a_{k+1})$ and $\bP^2_{a_{k-1}} \in \sM_{\gamma,\delta}(a_{k-1})$. Thus, by Lemmas~\ref{lemma_exit_lwrbd} and \ref{lemma_exit_upperbd2}, we have
\[\bP \circ \theta_{\sigma_1} \,(X_k > R-r)  \leq \frac 2 3 \cdot \bP^1_{a_{k+1}}(X_k > R-r) + \frac 3 8 \cdot\bP^2_{a_{k-1}}(X_k > R-r).\]
Next, we decompose the $\bP_{a_{k+1}}^1$ and $\bP^2_{a_{k-1}}$ probabilities using Lemma~\ref{lemma_X_shift}. By this result, there exist $\bP^3_{a_k}, \bP^4_{a_k} \in \sM_{\gamma,\delta}(a_k)$ such that the above yields
\[\bP \circ \theta_{\sigma_1} (X_k > R-r)  =  \frac 2 3 \cdot \bP_{a_{k+1}}^1(\Gamma^k_1 < J) \cdot \bP^3 (X_k > R-r)+ \frac 3 8 \cdot \bP^2_{a_{k-1}}(\Gamma^k_1 < J) \cdot \bP^4(X_k > R-r).\] 
To conclude, we observe that, since $\bP^1_{a_{k+1}} \in \sM_{\gamma,\delta}(a_{k+1})$, 
\[ \bP^1(\Gamma^k_1 < J) \leq \bP^1(\Gamma^k_1 < \infty)  = \bP^1(\tau_{a_k} < \tau_0) \leq \frac{9}{16},\]
where final inequality follows by Proposition~\ref{prop_hittingprob_main}, provided $\zeta$ is sufficiently small. Bounding $ \bP^2_{a_{k-1}}(\Gamma^k_1 < J) $ above by $1$, we conclude that 
\[\bP \circ \theta_{\sigma_1} (X_k > R-r)  =  \frac 2 3 \cdot \frac{9}{16} \cdot \bP^3_{a_k} (X_k > R-r)+ \frac 3 8 \cdot \bP^4_{a_k}(X_k > R-r).\] 
Without loss of generality, suppose that $\bP^3_{a_k} (X_k > R-r) \geq \bP^4_{a_k} (X_k > R-r)$. Then we obtain that
\[\bP \circ \theta_{\sigma_1} (X_k > R-r)  = \frac 3 4 \cdot \bP^3_{a_k} (X_k > R-r). \] 
Substituting the above into \eqref{eq_Xkrec1}, we obtain that 
\begin{equation*} 
\bP(X_k > R) = \bP(\tau_{a_{k-1}} \wedge \tau_{a_{k+1}} > r) + \frac 3 4 \cdot \bP^3 (X_k > R-r).
\end{equation*}
This proves the result for $\bP \in \sM_{\gamma,\delta}(a_k)$. For general $\bP \in \sM_{\gamma,\delta}$, by the PSMP there exists $\bP'_{a_k} \in \sM_{\gamma,\delta}(a_k)$ such that
\begin{align*}
    \bP(X_k > R) = \bP(\Gamma^k_1 < J) \cdot \bP^{\{\tau_{a_k} < \tau_{a_0} \wedge \tau_0 \}} \circ \theta_{\tau_{a_k}} (X_k > R) &= \bP(\Gamma^k_1 < J) \cdot \bP'_{a_k} (X_k > R) 
    \\ &\leq \bP'_{a_k}(X_k > R).
\end{align*}
The claimed inequality for $\bP(X_k > R)$ now follows from the proven inequality for $\bP'_{a_k} \in \sM_{\gamma,\delta}(a_k)$. This completes the proof.
\end{proof}

We now prove Proposition~\ref{prop_stop_tail}.
\begin{proof}[Proof of Proposition~\ref{prop_stop_tail}]
Fix $m \in \N$. Later on we will assume that $m$ is sufficiently large depending on certain universal parameters. For $k \in \N_0$ we define
\begin{align*}
R_k := (m+k)^2 \zeta^{2-2\gamma} 2^{-(2-2\gamma)k}.
\end{align*}
It is clear from this definition that there exists $\nu > 0$, independent of $m$, such that 
\begin{align*}
R_0 + \sum_{k=0}^\infty R_k \leq  \nu (m^2+1) \zeta^{2-2\gamma}.
\end{align*}
The above, \eqref{eq_hitting_Xk_decomp} and a union bound then imply that
\begin{align} \label{eq_propstoptail1}
\bP(\tau_0 \wedge \tau_{2\zeta} > \nu(m^2 + 1) \zeta^{2-2\gamma}) &\leq \bP(\sigma_0 > R_0) + \bP(X_k > R_k \text{ for some } k \in \N)  \notag
\\ &\leq \bP(\sigma_1 > R_0) + \sum_{k =1}^\infty \bP(X_k > R_k).
\end{align}
First, we note that $\sigma_0$ is the time of the first visit of $Y$ to $\cA$. First suppose without loss of generality that $\bP(Y_0 \in \cA) = 0$, which gives an upper bound because $\sigma_0 = 0$ if $Y_0 \in \cA$. We then partition over the events $\{Y_0 \in I_n\}$ for $n \in \N$ and observe from a (trivial) application of the PSMP that $\bP^{\{Y_0 \in I_n\}} \in \sM_{\gamma,\delta}(I_n)$. Hence
\begin{align}\label{eq:sigma0bd}
    \bP(\sigma_0 > R_0) &= \sum_{n = 1}^\infty \bP(Y_0 \in I_n)\cdot \bP^{\{Y_0 \in I_n\}}(\tau_{a_{n+1}} \wedge \tau_{a_{n-1}} > R_0) \notag
    \\ &=\sum_{n = 1}^\infty \bP(Y_0 \in I_n)\cdot \bP^{\{Y_0 \in I_n\}}(\tau_{a_{n+1}} \wedge \tau_{a_{n-1}} > m^2 a_1^{2-2k} ) \notag
    \\ &=\sum_{n = 1}^\infty \bP(Y_0 \in I_n)\cdot \bP^{\{Y_0 \in I_n\}}(\tau_{a_{n+1}} \wedge \tau_{a_{n-1}} > m^2 a_k^{2-2k} ) \notag
    \\ &\leq e^{-c_3 m^2}.
\end{align}

Our next goal is to obtain a bound for $\bP(X_k > R_k)$ for each $k \in \N$. By Lemma~\ref{lemma_Xk_recursion} applied with $R= R_k > (m+k)a_k^{2-2\gamma} = r$, we have
\[ \bP(X_k > R_k ) \leq \bP^0(\tau_{a_{k+1}}\wedge \tau_{a_{k-1}} > (m+k) a_k^{2-2\gamma}) + \frac 3 4 \cdot \bP'( X_k > R_k -  (m+k) a_k^{2-2\gamma}),\]
where $\bP^0, \bP' \in \sM_{\gamma,\delta}(a_k)$. Provided $R_k - 2(m+k) a_k^{2-2\gamma} \geq 0$, we may apply the lemma again to the second probability on the right-hand side above (now with $R = R_k -  (m+k) a_k^{2-2\gamma}$ and $r = (m+k)a_k^{2-2\gamma}$). We then continue apply this procedure for a total of $(m+k)$ iterations, so that remainder on the final iteration is $R_k - (m+k)^2 a_k^{2-2\gamma} = 0$. This yields that, for some collection of probabilities $\bP^i \in \sM_{\gamma,\delta}(a_k)$, $i =0,\dots,m+k$, we have
\begin{align*}
 \bP(X_k > R_k ) &\leq \sum_{i=0}^{m+k-1} \left( \frac 3 4\right)^i \bP^i(\tau_{a_{k+1}}\wedge \tau_{a_{k-1}} > (m+k)a_k^{2-2\gamma}) + \left( \frac 3 4 \right)^{m+k} \bP^{m+k }( X_k > 0)
 \\ &\leq \sum_{i=0}^{m+k-1} \left( \frac 3 4\right)^i \bP^i(\tau_{a_{k+1}}\wedge \tau_{a_{k-1}} > (m+k) a_k^{2-2\gamma}) + \left( \frac 3 4 \right)^{m+k}.
\end{align*}
By Lemma~\ref{lemma_exit_exp}, we have $\bP^i(\tau_{a_{k+1}}\wedge \tau_{a_{k-1}} > (m+k) a_k^{2-2\gamma}) \leq e^{-c_3 (m+k)}$ for each $i = 0,\dots,m+k-1$. It follows that, if $c = \min(c_3, -\log(3/4))$, then
\begin{equation*}
\bP(X_k > R_k) \leq 5 e^{-c(m+k)}. 
\end{equation*}
Substituting the above into \eqref{eq_propstoptail1} and using \eqref{eq:sigma0bd}, we deduce that for any $m \in \N$,
\begin{align} \label{eq_propstoptail2}
\bP(\tau_0 \wedge \tau_\zeta > \nu(m^2 + 1) \zeta^{2-2\gamma}) \leq e^{-c_3 m^2} +  \sum_{k=0}^\infty 5e^{-c(m+k)} \leq C e^{-cm}
\end{align}
for some $C>0$. Possibly enlarging $C$, it also holds for $m = 0$.

We now conclude the stated bound from the above. Let $m \in \N$ such that $m \geq \nu$. Then 
\begin{align*}
\bP(\tau_0 \wedge \tau_{2\zeta} > m \zeta^{2-2\gamma}) &= \bP( \tau_0 \wedge \tau_{2\zeta} > \nu (\sqrt{(m/\nu) - 1}^2 + 1)\zeta^{2-2\gamma} )
\\ &\leq \bP( \tau_0 \wedge \tau_{2\zeta} > \nu (\lfloor\sqrt{(m/\nu) - 1}\rfloor^2 + 1)\zeta^{2-2\gamma} )
\\ &\leq C e^{-c \lfloor\sqrt{(m/\nu) - 1}\rfloor}
\\ &\leq C' e^{-\frac{c}{\sqrt \nu} \sqrt m },
\end{align*} 
for some $C' > C$, where the second last inequality uses \eqref{eq_propstoptail2} and the fact that $\lfloor\sqrt{(m/\nu) - 1}\rfloor$ is a non-negative integer. By possibly further enlarging the constant, we may obtain that the same inequality holds if $m < \nu$. This implies the desired result.
\end{proof}

\section{The zero sets of class-$\sM_{\gamma,\delta}$ semimartingales} \label{s_zeros}
We are now able to assemble the results from Section~\ref{s_intervals} and prove our main result about the zero sets of $\sM_{\gamma,\delta}$-class semimartingales. Theorem~\ref{thm_main_compact} will follow from this result along with some coupling arguments.

Given $\bP \in \sM_{\gamma,\delta}$, we define $\cZ_t := \{s \in [0,t] : Y_t = 0\}$ to be the zero set of $Y$ up to time $t>0$, and we denote its Lebesgue measure by $|\cZ_t|$.

\begin{theorem} \label{thm_Y_zeroset}
(a) For any $t>0$ and $\eps >0$, for sufficiently small $\delta>0$,
\[ \bP(|\cZ_t| < (1-2 \eps) t) < \eps \]
for any $\bP \in \sM_{\gamma,\delta}(0)$.

(b) For any $t >0$, $\delta >0$ and $\bP \in \sM_{\gamma,\delta}(0)$, $|\cZ_t| > 0$ $\bP$-a.s.
\end{theorem}

Part (a) is the key to proving Theorem~\ref{thm_main_compact}, and the weaker claim of part (b) is a by-product of our methods. We remark that part (b) implies Theorem~\ref{thm_sde}, which follows immediately when one observes that a solution $Z$ to \eqref{eq_Zsde} with $Z_0 = 0$ belongs to $\sM_{\gamma,\delta}(0)$. The proof of Theorem~\ref{thm_Y_zeroset} is based on a formula for $|\cZ_t|$ in terms of the local time of $Y$ at level $0$, which we denote by $L_t$. Then, we use the downcrossing formulation of the local time and results from the previous section to obtain lower bounds on downcrossing counts until time $t$, which allow us to obtain lower bounds on the local time, and hence on $|\cZ_t|$, proving Theorem~\ref{thm_Y_zeroset}. 

To keep track of downcrossings, we now introduce some notation. Let $\zeta > 0$ and $\gamma^0_\zeta = 0$. For $n \in \N$ we then recursively define stopping times
\begin{equation} \label{def:taungamman}
\tau^n_\zeta := \inf \{t > \gamma^{n-1}_\zeta : Y_t = \zeta \}, \quad \gamma^n_\zeta = \inf \{t > \tau^n_{\zeta} : Y_t = 0 \},
\end{equation}
with the convention that $\inf \emptyset = + \infty$, and hence if $\gamma^m_\zeta = \infty$ for some $m \in \N$, then $\tau^n_\zeta = \gamma^n_\zeta = \infty$ for all $n > m$. A downcrossing of the interval $(0,\zeta)$ is said to occur each time that $Y$ visits zero after a new visit to $\zeta$. That is, each $\gamma^n_\zeta$ marks a downcrossing of $(0,\zeta)$. We denote by $N_\zeta(t)$ the number of downcrossings of $(0,\zeta)$ which have occurred by time $t$, i.e.\ $N_\zeta(t) := \max \{ n \in \N_0 : \gamma^{n}_\zeta \leq t\}$.

As $Y$ is a semimartingale, its local time at $0$, denoted by $(L_t)_{t \geq 0}$, exists (see e.g.\ \cite[Theorem VI.1.2]{RevuzYor}). 
Moreover, for $t \geq 0$, $L_t$ can be obtained as the 
re-normalized limit of $N_\zeta(t)$ as $\zeta \downarrow 0$. In particular, it is a consequence of \cite[Theorem VI.1.10]{RevuzYor} that 
\begin{equation} \label{eq_downcross_localtime} \zeta N_\zeta(t) \xrightarrow[]{p} \frac 1 2 L_t \, \text{ as $\zeta \downarrow 0$.} \end{equation}
To connect $L_t$ to $|\cZ_t|$, we use \cite[Theorem VI.1.7]{RevuzYor}. This result states that for any semimartingale $Y_t = Y_0 + M_t + A_t$, $a \in \R$ and $t \geq 0$, with probability one,
\begin{equation*}
L^a_t - L^{a-}_t = 2 \int_0^t \indc_{\{Y_s = a\}} dA_s,
\end{equation*}
where $L^a_t$ is the local time of $Y$ at level $a$ and time $t$. Applying this to $Y$, the canonical process of $\bP \in \sM_{\gamma,\delta}(0)$, with $a = 0$, we remark that $L^0_t - L^{0-}_t = L^0_t = L_t$, since $Y$ is a.s. non-negative. Since $A_t \equiv \delta t$, we obtain that $\bP$-a.s.,
\begin{equation} \label{eq_localtime_zeroset}
L_t = 2\int_0^t \indc_{\{Y_s = 0\}} \delta ds = 2 \delta |\cZ_t|.
\end{equation}
Combining \eqref{eq_downcross_localtime} and \eqref{eq_localtime_zeroset}, we conclude that under $\bP \in \sM_{\gamma,\delta}(0)$,
\begin{equation} \label{eq_downcross_zeroset} 
 \frac{\zeta N_\zeta(t)}{\delta} \xrightarrow[]{p} |\cZ_t| \, \text{ as $\zeta \downarrow 0$.} 
\end{equation}
Our strategy to prove Theorem~\ref{thm_Y_zeroset} is therefore to prove that with high probability, $N_\zeta(t)$ is sufficiently large. Our main result on the number of downcrossings is the following.
\begin{proposition} \label{prop_downcrosscount} (a) For $t>0$ and $\eps>0$, if $\delta$ is sufficiently small, then for sufficiently small $\zeta>0$ and any $\bP \in \sM_{\gamma,\delta}(0)$,
\begin{equation} \label{eq_prop_downcount}
\bP(N_\zeta(t) < \lceil (1-2 \eps) \delta t \zeta^{-1}\rceil) < \eps.
\end{equation}
(b) For any $t>0$, $\delta>0$ and $\bP \in \sM_{\gamma,\delta}(0)$,
\[\lim_{\eps \downarrow 0}\lim_{\zeta\downarrow 0} \bP(N_\zeta(t) < \eps \zeta^{-1}) = 0. \]
\end{proposition}

We can now obtain Theorem~\ref{thm_Y_zeroset} as a consequence of this result and the previous discussion.
\begin{proof}[Proof of Theorem~\ref{thm_Y_zeroset}]
Suppose that $t, \eps$ and $\delta$ are such that Proposition~\ref{prop_downcrosscount} holds. By \eqref{eq_downcross_zeroset}, there exists a subsequence $(\zeta_n)_{n \in \N}$ converging to zero such that, $\bP$-a.s., $\delta^{-1} \zeta_n N_{\zeta_n}(t) \to |\cZ_t|$ as $n \to \infty$. We then compute
\begin{align*}
\bP(|\cZ_t| < (1-2\eps)t) &= \bP( \lim_{n \to \infty} \delta^{-1} \zeta_n N_{\zeta_n}(t) < (1-2\eps)t)
 \\ &\leq  \liminf_{n \to \infty} \bP(N_{\zeta_n}(t) < \lceil (1-2\eps) \delta t \zeta_n^{-1} \rceil)
\\ & < \eps,
\end{align*}
where in the second-to-last inequality we use Fatou's Lemma, and the final inequality uses \eqref{eq_prop_downcount}. This completes the proof of part (a). The proof of part (b) is the same, but uses Proposition~\ref{prop_downcrosscount}(b) instead of \eqref{eq_prop_downcount}. \end{proof}

The rest of this section is dedicated to the proof of Proposition~\ref{prop_downcrosscount}. To ease notation, we introduce $n_0 = n_0(t,\delta,\eps,\zeta) = \lceil (1-2\eps)\delta t \zeta^{-1} \rceil$. We remark that for $\bP \in \sM_{\gamma,\delta}(0)$
and for $\gamma^{n_0}_\zeta$ defined in~\eqref{def:taungamman},
\begin{equation} \label{eq:downcount_stop}
\bP(N_{\zeta}(t) < n_0) = \bP(\gamma^{n_0}_\zeta > t),
\end{equation} so it suffices to bound the latter probability. To this end, for $n \in \N$ we decompose $\gamma^{n}_\zeta$ in terms of the upcrossings and downcrossings of $(0,\zeta)$ by writing 
 \begin{align} \label{eq_UpDownDecomp}
\gamma^{n}_\zeta &= \sum_{i=1}^{n} \big(\tau^i_{\zeta} - \gamma^{i-1}_{\zeta}\big) \indc(\gamma^{i-1}_\zeta < \infty) + \sum_{i=1}^{n} \big( \gamma^{i}_{\zeta} - \tau^i_\zeta \big) \indc(\tau^{i}_\zeta < \infty)\notag
\\ &=: \sum_{i = 1}^{n} U_i + \sum_{i = 1}^{n} D_i. 
\end{align}
To see that the above is true, first note that it holds when $\gamma^{n}_\zeta < \infty$, since then $\tau^i_\zeta, \gamma^i_\zeta < \infty$ for all $i \in [n]$ and hence all of the indicator functions are equal to $1$ and the sum telescopes into $\gamma^n_\zeta - \tau^0_\zeta = \gamma^n_\zeta$. If $\gamma^{n}_\zeta = \infty$, then it must be the case that $Y$ never returned to $0$ after visiting $\zeta$ for the $i$th time for some $i \in [n]$, and hence $D_i = \infty$ for some $i \in [n]$, and the formula holds. 

The two terms on the right-hand side of \eqref{eq_UpDownDecomp} correspond respectively to the amount of time spent going from $0$ to $\zeta$ (up) and the amount of time spent going from $\zeta$ to $0$ (down). To denote these quantities, we write
\begin{align} \label{defTupdown}
T_{\text{up}}(n) :=  \sum_{i = 1}^{n} U_i, \quad T_{\text{down}}(n) := \sum_{i = 1}^{n} D_i.
\end{align}
These are analyzed separately in the next two sections.

\subsection{Time spent going up} \label{s_up} We recall that $n_0 = n_0(t,\delta,\eps,\zeta) = \lceil (1-2\eps)\delta t \zeta^{-1} \rceil$. The following is our main result on the duration of upcrossings.

\begin{lemma} \label{lemma_uptime} There is a constant $C_5 \geq 1$ such that the following holds: for any $\delta >0$, $t > 0$, $\eps \in (0,1]$ and $\zeta \in (0, \delta t\eps / 2)$, for any $\bP \in \sM_{\gamma,\delta}(0)$, 
\[\bP(T_{\mathrm{up}}(n_0) > (1 - \eps) t) < \frac{C_5}{\eps^2} \cdot \frac{\zeta}{\delta t}.\]
\end{lemma}

We recall the random variable $\tau_\zeta = \inf \{t \geq 0 : Y_t = \zeta\}$. We remark that the random variables $\big(\tau^i_{\zeta} - \gamma^{i-1}_{\zeta}\big)$, for $i \in \N$,  are essentially ``iid" copies of $\tau_\zeta$ under $\bP \in \sM_{\gamma,\delta}(0)$, with the usual caveat that $\bP$ is not in fact unique. Nonetheless, we can obtain sufficient control over the first and second moments of $\tau_\zeta$ which hold uniformly for all $\bP \in \sM_{\gamma,\delta}(0)$. 


\begin{lemma} \label{lemma_upcrossing} For any $\delta>0$, $\zeta >0$, any $\bP \in \sM_{\gamma,\delta}(0)$, the following hold: \\
(a) $\bP(\tau_\zeta > t) \leq 2e^{-\frac{\delta \zeta^{-1}}{18} t}$ for all $t \geq 0$. \\
(b) $\E[\tau_\zeta] = \frac \zeta \delta$. \\ 
(c) $\E[ \tau_\zeta^2] \leq 1296 \cdot \frac{ \zeta^2}{ \delta^2}$.
\end{lemma}
\begin{proof} Let $\bP \in \sM_{\gamma,\delta}(0)$. Part (a) follows immediately from Lemma~\ref{lemma_bddstop_tail} and the fact that $\bP(Y_{t \wedge \tau_\zeta} \leq \zeta \text{ for all } t \geq 0) =1$. Under $\bP$, we have $Y_t = \delta t + M_t$ for all $t \geq 0$, and hence
\[\zeta = Y_{\tau_\zeta} = \delta \tau_\zeta + M_{\tau_\zeta}.\]
Next, we claim that $t\mapsto M_{t \wedge \tau_\zeta}$ is a uniformly integrable martingale. To see this, note that $|M_{t \wedge \tau_\zeta}| \leq \zeta + \delta t$ for all $t \geq 0$, which along with the tail bound from part (a) proves that $M_{\cdot \wedge \tau_\zeta}$ is bounded in $L^2$, and hence uniformly integrable. Thus, we can apply the optional stopping theorem at $\tau_\zeta$, implying $\E[M_{\tau_\zeta}] = 0$, from which the claim from part (b) follows.

To prove part (c), we use the tail bound from part (a) to compute directly
\begin{align*}
\E[\tau_\zeta^2] = \int_0^\infty \bP( \tau_\zeta > \sqrt t) dt &\leq 2 \int_0^\infty e^{-\frac{\delta \zeta^{-1}}{18} \sqrt t} dt = 1296  \cdot \frac{\zeta^2}{\delta^2},
\end{align*}
with the last equality following from elementary arguments.
\end{proof}

We now use these estimates to prove Lemma~\ref{lemma_uptime}.

\begin{proof}[Proof of Lemma~\ref{lemma_uptime}]
Let $\bP \in \sM_{\gamma,\delta}(0)$. We define a random integer $N_0$ by
\[ N_0 := \sum_{i = 1}^{n_0} \indc(\gamma_\zeta^{i-1} < \infty) = (N_\zeta(\infty)+1) \wedge n_0.\]
We note that 
\[ N_0 \leq n_0 < (1 - (3/2)\eps) \delta t \zeta^{-1},\]
where the final inequality holds for $\zeta\leq \delta t \eps/2$. Hence, writing $T_{\text{up}} := T_{\text{up}}(n_0)$, we have
\begin{align} \label{eq_upcrossfinal00}
\bP( T_{\text{up}}> (1-\eps) t) &= \bP \left(T_{\text{up}} - N_0 \frac \zeta \delta > (1-\eps) t - N_0 \frac \zeta \delta \right) \notag
\\ &\leq \bP \left(T_{\text{up}} - N_0 \frac \zeta \delta > \frac{\eps t}{2 }  \right) \notag
\\ &\leq \frac{4 \E[ (T_{\text{up}} - N_0 \frac \zeta \delta)^2 ]}{\eps^2 t^2}
\end{align}
by Markov's inequality. We now bound the expectation above. Separating the diagonal and off-diagonal terms, we have
\begin{align}\label{eq_upcrossfinal0}
&\E \bigg[ \left(T_{\text{up}} - N_0 \frac \zeta \delta \right)^2 \bigg] 
\\ &\hspace{1 cm}= \sum_{i=1}^{n_0} \E \bigg[ \indc(\gamma^{i-1}_\zeta < \infty) \left(\tau^{i}_\zeta - \gamma^{i-1}_\zeta - \frac \zeta \delta \right)^2 \bigg] \notag
\\ &\hspace{1 cm} \quad + 2 \sum_{i=1}^{n_0} \sum_{j= 1}^{i-1}\E \bigg[ \indc(\gamma^{j-1}_\zeta < \infty) \left(\tau^{j}_\zeta - \gamma^{j-1}_\zeta - \frac \zeta \delta \right) \indc(\gamma^{i-1}_\zeta < \infty) \left(\tau^{i}_\zeta - \gamma^{i-1}_\zeta - \frac \zeta \delta \right)\bigg]. \notag
\end{align}
We remark that for any $i =1,\dots,n_0$, we have $(\tau^{i}_\zeta - \gamma^{i-1}_\zeta)(\omega) = \tau_\zeta  \circ \theta_{\gamma^{i-1}}(\omega)$ provided $\gamma^{i-1}_\zeta(\omega) < \infty$. Furthermore, $Y_{\gamma^{i-1}_\zeta} = 0$ on $\{\gamma^{i-1}_\zeta < \infty\}$. Hence, for each $i$, for every $K >0$ we can apply the PSMP at the finite stopping time $\gamma^{i-1}_\zeta \wedge K$, which implies that for some $\bP_0' \in \sM_{\gamma,\delta}(0)$ (with expectation $\E_0'$) such that
\begin{align*}
 \E \bigg[ \indc(\gamma^{i-1}_\zeta \leq K) \left(\tau^{i}_\zeta - \gamma^{i-1}_\zeta - \frac \zeta \delta \right)^2 \bigg]  &= \bP(\gamma^{i-1}_\zeta \leq K) \cdot \E^{\{\gamma^{i-1}_\zeta \leq K \}} \circ \theta_{\gamma^{i-1}_\zeta} \bigg [ \left(\tau_\zeta - \frac \zeta \delta \right)^2 \bigg] 
\\ &\leq \E'_0 [\tau_\zeta^2] + \frac{\zeta^2}{\delta^2} 
\\ &\leq 1297 \cdot \frac{\zeta^2}{\delta^2},
 \end{align*}
where the final inequality follows from Lemma~\ref{lemma_upcrossing}(c).
Since this holds for all $K >0$, we may conclude by monotone convergence that 
\begin{equation} \label{eq_upcrossfinal1}
    \E \bigg[ \indc(\gamma^{i-1}_\zeta < \infty ) \left(\tau^{i}_\zeta - \gamma^{i-1}_\zeta - \frac \zeta \delta \right)^2 \bigg] \leq 1297 \cdot \frac{\zeta^2}{\delta^2}.
\end{equation}

 Now let us handle the off-diagonal terms appearing in \eqref{eq_upcrossfinal0}. 
 We remark that $(\tau^{i}_\zeta - \gamma^{i-1}_\zeta)(\omega) = \tau_\zeta  \circ \theta_{\gamma^{i-1}}(\omega)$ on $\{\gamma^{i-1}_\zeta < \infty\}$. 
 By the (regular conditional form of the) PSMP, for $1 \leq j < i \leq n_0$, for any $K>0$ we have
 \begin{align*}
 &\E \bigg[ \indc(\gamma^{j-1}_\zeta < \infty) \left(\tau^{j}_\zeta - \gamma^{j-1}_\zeta - \frac \zeta \delta \right) \indc(\gamma^{i-1}_\zeta < K) \left(\tau^{i}_\zeta - \gamma^{i-1}_\zeta - \frac \zeta \delta \right)\bigg] \notag
 \\ &\hspace{1 cm} =\E \bigg[\left(\tau^{j}_\zeta - \gamma^{j-1}_\zeta - \frac \zeta \delta \right) \indc(\gamma^{i-1}_\zeta \leq K) \cdot \mathbb{Q}^\omega_{\gamma^{i-1}_\zeta \wedge K}\bigg[\tau_\zeta  - \frac \zeta \delta \bigg]\bigg] \notag
 \\ &\hspace{1 cm}=0,
 \end{align*}
 where the final equality is by Lemma~\ref{lemma_upcrossing}(b), using $\mathbb{Q}^\omega_{\gamma^{i-1}_\zeta \wedge K} \in \sM_{\gamma,\delta}(0)$ a.s.\ on $\{\gamma^{i-1}_\zeta \leq K\}$. Letting $K \to \infty$, we obtain
 \begin{equation*}
      \E \bigg[ \indc(\gamma^{j-1}_\zeta < \infty) \left(\tau^{j}_\zeta - \gamma^{j-1}_\zeta - \frac \zeta \delta \right) \indc(\gamma^{i-1}_\zeta < \infty) \left(\tau^{i}_\zeta - \gamma^{i-1}_\zeta - \frac \zeta \delta \right)\bigg] = 0
 \end{equation*}
 for $1 \leq j < i \leq n_0$.
Substituting this and \eqref{eq_upcrossfinal1} into \eqref{eq_upcrossfinal0}, we obtain
\begin{align*}
\E \bigg[ \left(T_{\text{up}} - N_0 \frac \zeta \delta \right)^2 \bigg] \leq C \cdot n_0 \cdot \frac{\zeta^2}{\delta^2} \leq 2C \cdot t \zeta/\delta.
\end{align*}
where the final inequality holds provided $\zeta \leq \delta t$ (which holds, since $\zeta \leq \delta t \eps/2$), and $C \geq 1$ does not depend on any parameters. Hence, by \eqref{eq_upcrossfinal00}, we have
\begin{align*}
\bP( T_{\text{up}}> (1-\eps) t) \leq \frac{8C}{\eps^2} \cdot \frac{\zeta}{\delta t}.
\end{align*}
This completes the proof. \end{proof}


\subsection{Time spent going down} \label{s_down}

We now quantify the time spent on downcrossings of the interval $(0,\zeta)$, which is more involved than handling the upcrossings. Our approach is intuitively based on a height decomposition in the following sense. Suppose that $\bP \in \sM_{\gamma,\delta}(\zeta)$, and let $H = \sup\{ m \in \N : \tau_{2^{m-1} \zeta} < \tau_0\}$, i.e.\ $H$ is the maximal ``height'' reached by $Y$ on a logarithmic scale before returning to $0$. Due to Proposition~\ref{prop_hittingprob_main}, provided $Y$ does not go \emph{too} high before returning to $0$, $H$ is well approximated by a $\mathrm{Geometric}(1/2)$ random variable. Thus, out of $n$ excursions started from $\zeta$, we expect roughly $n/2$ with height $1$, $n/4$ with height $2$, and so on, with approximately $n/2^k$ at height $k$, and the maximum height near $\log_2 n$. 

While conceptually it is convenient to argue with the height of downcrossings, we must control their durations. To do this, we decompose their durations into the time spent on each rung of the dyadic ladder. If a downcrossing started from $\zeta$ reaches maximal height $H = k$, we consider the time spent going from $\zeta$ to $2\zeta$, from $2 \zeta$ to $4 \zeta$, up to $2^{k-1} \zeta$, and finally the time spent returning from $2^{k-1} \zeta$ to zero. It follows from Proposition~\ref{prop_exit_exp} that the last contribution is of order $(2^k \zeta)^{2-2\gamma}$. The next result, which is the key result in our quantification of downcrossing times, establishes that, for a downcrossing attaining maximal height $H=k$, the combined time spent at lower levels is comparable in length to the final crossing from $2^{(k-1)}\zeta$ to $0$. In particular, we show that the probability that a downcrossing has duration of order $(2^k \zeta)^{2-2\gamma}$ is comparable to the probability that it reaches height $k$, which we know is approximately $2^{-k}$.


\begin{proposition} \label{prop_tau0_tail} There exist constants $C_6 \geq 1$ and $L_0 \in \N$ such that the following holds. For sufficiently small $\zeta > 0$, for any $\delta \in (0,1]$ and $\bP_\zeta \in \sM_{\gamma,\delta}(\zeta)$, for all $k \in \N$ such that $k \leq \log_2 (1/\zeta) - L_0$, 
\begin{equation}
\bP_\zeta(\tau_0 > C_6 (2^k \zeta)^{2-2\gamma}) \leq 2^{-k}.
\end{equation}
\end{proposition} 

\begin{proof}
Let $\zeta_0>0$ be sufficiently small so that the conclusions of Propositions~\ref{prop_hittingprob_main} and \ref{prop_exit_exp} hold for all $\zeta \in (0,\zeta_0]$ and $\delta \in (0,1]$. Throughout this proof, we will make use of the estimates from Propositions~\ref{prop_hittingprob_main} and \ref{prop_exit_exp} applied at various points $\zeta'$, in particular at $\zeta' = 2^i \zeta$ for $i = 1,\dots,k$ for some fixed $\zeta>0$. Since these estimates apply when $\zeta' \leq \zeta_0$, this imposes the constraint that $k \leq \log_2 (\zeta_0) + \log_2 (1/\zeta)$. Thus, we set $L_1 = -\log_2 (\zeta_0) \vee 1$ and work with $k \leq \log_2(1/\zeta)- L_1$, which ensures that the relevant estimates always apply. At the end of the proof, we increase $L_1$ to the larger value $L_0$ appearing in the statement.

Let $\bP \in \sM_{\gamma,\delta}(\zeta)$ and $k \leq \log(1/\zeta) - L_1$. We will begin by obtaining a bound on $\bP(\tau_0 > (2^k \zeta)^{2-2\gamma})$ which can later be manipulated to obtain the stated bound. 
Recall that $H = \sup\{ m \in \N : \tau_{2^{m-1} \zeta} < \tau_0\}$ is the maximal height reached by $Y$, on a logarithmic scale, before returning to $0$. We also remark that, under $\bP \in \sM_{\gamma,\delta}(\zeta)$, $\tau_\zeta = 0$ a.s. 
We then define
\[ T_i = \tau_{2^i \zeta}\wedge \tau_0 - \tau_{2^{i-1}\zeta} \wedge \tau_0 , \quad i \geq 1.\]
Note that for $m \in \N$, on $\{H = m\}$, $T_i = \tau_{2^i \zeta} - \tau_{2^{-i-1} \zeta}$ for $i = 1,\dots,m-1$ and $T_m =  \tau_0 - \tau_{2^{m-1} \zeta}$. These increments telescope, and on $\{H = m\}$ we have
\[\tau_0 = \sum_{i = 1}^{m} T_i.\]
Next, still on $\{H = m \}$, suppose that $T_i \leq (2^k \zeta)^{2-2\gamma} / 2^{m-i + 1}$ for each $i = 1,\dots,m$. Then we have
\[ \tau_0 \leq \sum_{i = 1}^{m} (2^k \zeta)^{2-2\gamma} / 2^{m-i + 1} \leq (2^k \zeta)^{2-2\gamma}.\]
In particular, this implies that 
\begin{align} \label{eq_prop_downcross_decomp}
&\bP( \tau_0 > (2^k \zeta)^{2-2\gamma})\notag
\\ &\hspace{1 cm} \leq\bP(H > k) + \sum_{m=1}^{k} \bP\left( H = m, T_i > (2^k \zeta)^{2-2\gamma} / 2^{m-i + 1} \text{ for some } i =1,\dots,m \right).
\end{align}
For the first term, we will ultimately show that 
\begin{equation} \label{eq_prop_downcross_heightbd}
\bP (H > k) \leq \prod_{j=1}^{k} \left( \frac 1 2 + 8(2^j \zeta)^{c_1} \right),
\end{equation}
where $c_1$ is the constant from Proposition~\ref{prop_hittingprob_main}.
In fact, we will derive this as a by-product of the proof that a similar upper bound can be obtained for the second term in \eqref{eq_prop_downcross_decomp}, which we now show.

Let $m \in \{1,\dots,k\}$. We consider the contribution to \eqref{eq_prop_downcross_decomp} from the event $\{H = m\}$. By a union bound, we have
\begin{align} \label{eq_heightm_totalbd}
&\bP(H = m, T_j > (2^k \zeta)^{2-2\gamma} / 2^{m-j + 1} \text{ for some } j =1,\dots,m) 
\\ &\hspace{1 cm} \leq \sum_{j=1}^m \bP(H= m, T_j > (2^k \zeta)^{2-2\gamma} / 2^{m-j + 1}). \notag
\end{align}
To obtain a suitable bound for the above, we will now prove the following:

\textbf{Claim.} For all $j \in \{1,\dots,m\}$, 
\begin{align} \label{eq_heightm_Tj_bd}
&\bP(H= m, T_j > (2^k \zeta)^{2-2\gamma} / 2^{m-j+ 1}) 
\\ &\hspace{1 cm} \leq C_1\prod_{i = 1}^{m-1} \left( \frac 1 2 + 8 (2^{i-1}\zeta)^{c_1} \right) \cdot  \exp \left( -c_2 \left( 2^{(2-2\gamma)(k-j-1) - (m-j+1)}  - 1 \right) \right), \notag
\end{align}
where $c_1$ is the constant from Proposition~\ref{prop_hittingprob_main} and $c_2$ and $C_1$ are the constants from Proposition~\ref{prop_exit_exp}.

Let us prove this claim. For $i,j \leq m$ we define events
\[ E_i =  \left\{\tau_{2^i \zeta} < \tau_0 \right\}\]
and
\[ F_j^m =  \{T_j > (2^k \zeta)^{2-2\gamma} / 2^{m-j+ 1}\}.\]
Then we remark that 
\begin{align} \label{eq_heightm_eventrep}
\left\{H = m, T_j > (2^k \zeta)^{2-2\gamma} / 2^{m-j+ 1} \right\} = E_{m}^c \cap \left( \cap_{i=1}^{m-1} E_i \right) \cap F^m_j.
\end{align}
It is now convenient to  subdivide the proof of \eqref{eq_heightm_Tj_bd} into two cases.\\

\noindent \textbf{Case 1 ($j = m$).}  By \eqref{eq_heightm_eventrep}, we have
\begin{align*}
 \bP(H= m, T_m & > (2^k \zeta)^{2-2\gamma} / 2^{m-j + 1}) = \bP\left(\left(\cap_{i=1}^{m-1} E_i\right) \cap E_m^c \cap F^m_m\right) 
\\ &= \bP\left(\left(\cap_{i=1}^{m-1} E_i\right) \cap \left\{ \tau_0 \wedge \tau_{2^m\zeta} - \tau_{2^{m-1}\zeta} > (2^k \zeta)^{2-2\gamma} / 2, \tau_0 < \tau_{2^m\zeta} \right\} \right) 
\\ &\leq \bP\left(\left(\cap_{i=1}^{m-1} E_i\right) \cap \left\{ \tau_0 \wedge \tau_{2^m\zeta} - \tau_{2^{m-1}\zeta} > (2^k \zeta)^{2-2\gamma} / 2 \right\} \right) 
\end{align*}
Since $\tau_0 \wedge \tau_{2^m\zeta} > \tau_{2^{m-1} \zeta}$ on $\cap_{i=1}^{m-1} E_i$, on this event we have
\[(\tau_0 \wedge \tau_{2^m\zeta})\circ \theta_{\tau_{2^{m-1}\zeta}}(\omega) = \tau_0 \wedge \tau_{2^m \zeta}(\omega) - \tau_{2^{m-1}\zeta}(\omega)\]
and $\bP^{\cap_{i=1}^{m-1} E_i} \circ \theta_{\tau_{2^{m-1}\zeta}} \in \sM_{\gamma,\delta}(2^{m-1}\zeta)$ by the PSMP. Combining these two facts and substituting them into the previous expression, we obtain
\begin{align} \label{eq_excursiontime1}
& \bP(H= m, T_m > (2^k \zeta)^{2-2\gamma} / 2)  \notag
\\ &\hspace{1.5 cm} \leq  \bP\left(\left(\cap_{i=1}^{m-1} E_i\right) \right) \cdot \bP^{\cap_{i=1}^{m-1} E_i} \circ \theta_{\tau_{2^{m-1}\zeta}} \left(  \tau_0 \wedge \tau_{2^m \zeta} >  (2^k \zeta)^{2-2\gamma} / 2  \right) \notag
\\ &\hspace{1.5 cm} \leq  \bP\left(\left(\cap_{i=1}^{m-1} E_i\right) \right) \cdot \bP^{\cap_{i=1}^{m-1} E_i} \circ \theta_{\tau_{2^{m-1}\zeta}} \left(  \tau_0 \wedge \tau_{2^m \zeta} >  \lfloor 2^{(2-2\gamma)(k-m)-1}\rfloor (2^m \zeta)^{2-2\gamma}\right)  \notag
\\ &\hspace{1.5 cm} \leq \bP\left(\left(\cap_{i=1}^{m-1} E_i\right) \right) \cdot C_1 \exp \left(-c_2 \left(2^{(2-2\gamma)(k-m)-1} - 1\right) \right),
\end{align}
where the last line follows from Proposition~\ref{prop_exit_exp} and the fact that $\bP^{\cap_{i=1}^{m-1} E_i} \circ \theta_{\tau_{2^{m-1}\zeta}} \in \sM_{\gamma,\delta}(2^{m-1}\zeta)$. 

To handle the remaining term, we remark that for each $i =2,\dots,m-1$, on $\cap_{l=1}^{i-1} E_l$ we have
\begin{align*} 
\tau_0 \circ \theta_{\tau_{2^{i-1}\zeta}}(\omega)  = \tau_0(\omega) - \tau_{2^{i-1}\zeta} (\omega) \quad \text{ and } \quad \tau_{2^i \zeta} \circ \theta_{\tau_{2^{i-1}\zeta}}(\omega)  = \tau_{2^i\zeta} (\omega) - \tau_{2^{i-1}\zeta} (\omega),
\end{align*}
which implies that the ordering of $\tau_{2^i \zeta}$ and $\tau_0$ is invariant under the shift $\theta_{\tau_{2^{i-1}\zeta}}$. Moreover, by the PSMP,
\[\bP^{\cap_{l=1}^{i-1} E_l} \circ \theta_{\tau_{2^{i-1}\zeta}} \in \sM_{\gamma,\delta}(2^{i-1}\zeta).\]
Together, these observations imply for each $i = 1,\dots,m-1$, for some $\bP^{(i)} \in \sM_{\gamma,\delta}(2^{i-1}\zeta)$,
\[ \bP^{\cap_{l=1}^{i-1} E_l} \circ \theta_{\tau_{2^{i-1}\zeta}}(E_i) = \bP^{(i)} ( \tau_{2^i \zeta} < \tau_0).\]
(This holds for $i=1$ as well, if we define the empty intersection to be the entire space and remark that $\theta_{\tau_\zeta} = \theta_0$ a.s.\ under $\bP \in \sM_{\gamma,\delta}(\zeta)$.) Starting at $i=m-1$, conditioning on $\cap_{l=1}^{i-2} E_l$, and repeating, we conclude that
\begin{align*}
\bP\left(\left(\cap_{i=1}^{m-1} E_i\right) \right) = \prod_{i=1}^{m-1} \bP^{(i)}( \tau_{2^i \zeta} < \tau_0) \leq \prod_{i=1}^{m-1} \left( \frac 1 2 + 8(2^{i-1} \zeta)^{c_1} \right).
\end{align*}
The final inequality follows from applying Proposition~\ref{prop_hittingprob_main} to each term. Substituting this bound into \eqref{eq_excursiontime1}, we have shown that,
\begin{align*}
 \bP(H= m, T_m > (2^k \zeta)^{2-2\gamma} / 2) &\leq C_1 \prod_{i=1}^{m-1} \left( \frac 1 2 + 8(2^{i-1} \zeta)^{c_1} \right) \cdot \exp \left(-c_2\left(2^{(2-2\gamma)(k-m)-1} - 1\right) \right),
\end{align*}
which is \eqref{eq_heightm_Tj_bd} for $j=m$.\\

\noindent \textbf{Case 2 ($j < m$).} The argument here is similar to the first case but with some reordering, so we just sketch the necessary changes. We begin with the representation \eqref{eq_heightm_eventrep} for our event. 
First, conditioning on every sub-event except for $E_m^c$ and shifting by $\tau_{2^{m-1}\zeta}$, we obtain a term $\bP'(\tau_0 < \tau_{2^m \zeta})$ with $\bP' \in \sM_{\gamma,\delta}(2^{m-1}\zeta)$. By Proposition~\ref{prop_hittingprob_main}, this is bounded above by $1/2$, and hence 
\begin{align} \label{}
\bP(H= m, T_j > (2^k \zeta)^{2-2\gamma} / 2^{m-i + 1}) &\leq \bP\left(\left(\cap_{i=1}^{m-1} E_i \cap F^m_j \right)\right) \cdot \frac 1 2.
\end{align}
For $i \in \{j + 1, m-1\}$, we proceed as in the final step of the previous case. That is, we condition on $\left(\cap_{l = 1}^{i-1} E_l \right) \cap F^m_{ j}$ and shift by $\tau_{2^{i-1}\zeta}$. Just like in the previous case, the conditioned and shifted law belongs to $\sM_{\gamma,\delta}(2^{i-1}\zeta)$, and the shifted event is $\{\tau_{2^i \zeta} < \tau_0\}$. Thus, subsequent applications of Proposition~\ref{prop_hittingprob_main} imply that
\begin{align*}
\bP(H= m, T_j > (2^k \zeta)^{2-2\gamma} / 2^{m-j + 1}) &\leq \bP\left(\left(\cap_{i=1}^{j} E_i \cap F^m_j \right)\right) \cdot \prod_{i=j+1}^{m-1} \left( \frac 1 2 + 8 (2^{i-1}\zeta)^{c_2} \right) \cdot \frac 1 2.
\end{align*}
For $i = j$, we remark that on the event $\cap_{i =1}^{j-1}E_i$, 
\[E_j \cap F^m_j \subset \left\{ \tau_{2^j \zeta} \wedge \tau_0 - \tau_{2^{j-1} \zeta} > (2^k \zeta)^{2-2\gamma} / 2^{m-j+1} \right\}.\]
Conditioned on $\cap_{i =1}^{j-1}E_i$, shifting by $\tau_{2^{j-1} \zeta}$ simply removes the term $- \tau_{2^{j-1}\zeta}$ from the event above, and $\bP^{\cap_{i =1}^{j-1}E_i} \circ \theta_{\tau_{2^{j-1} \zeta}} =: \bP^{(j)} \in \sM_{\gamma,\delta}(2^{j-1}\zeta)$ by the PSMP. It then follows from Proposition~\ref{prop_exit_exp} that
\begin{align*}
&\bP\left(\left(\cap_{i=1}^{j-1} E_i \cap F^m_j \right)\right) 
\\ &\hspace{1 cm}\leq \bP\left( \left( \cap_{i=1}^{j-1} E_i \right) \right) \cdot \bP^{(j)}(\tau_{2^j \zeta} \wedge \tau_0 > (2^k \zeta)^{2-2\gamma} / 2^{m-j + 1})
\\ &\hspace{1 cm}\leq  \bP\left( \left( \cap_{i=1}^{j-1} E_i \right) \right) \cdot \bP^{(j)}(\tau_{2^j \zeta} \wedge \tau_0 > \lfloor(2^{(2-2\gamma)(k-j-1) - (m-j+1)})\rfloor (2^{j-1} \zeta)^{2-2\gamma})
\\ &\hspace{1 cm}\leq \bP\left( \left( \cap_{i=1}^{j-1} E_i \right) \right) \cdot C_1 \exp \left( -c_2 \left( 2^{(2-2\gamma)(k-j-1) - (m-j+1)}  - 1 \right) \right). 
\end{align*}
The remaining terms are handled identically to the previous case using Proposition~\ref{prop_hittingprob_main}. We therefore obtain that for $j =1,\dots,m-1$, 
\begin{align*}
&\bP(H= m, T_j > (2^k \zeta)^{2-2\gamma} / 2^{m-j+ 1}) 
\\ &\hspace{1.5 cm} \leq \prod_{i  \in \{1,\dots,m-1\} \backslash \{j\} } \left( \frac 1 2 + 8 (2^{i-1}\zeta)^{c_1} \right) \cdot \frac 1 2\cdot  C_1 \exp \left( -c_2 \left( 2^{(2-2\gamma)(k-j-1) - (m-j+1)}  - 1 \right) \right).
\end{align*}
Bounding the $1/2$ above by the missing term in the product (corresponding to $i=j$), we have that \eqref{eq_heightm_Tj_bd} holds for $j \in \{0,\dots,m-1\}$. Combined with the previous case, this proves that $ \eqref{eq_heightm_Tj_bd}$ holds for all $j \in \{1,\dots,m\}$. 

We observe that $\{H > k\} = \cap_{i=1}^{k} E_i$, and that the probability of this event can be handled using the arguments above, which immediately yields \eqref{eq_prop_downcross_heightbd}. Let us now return to \eqref{eq_heightm_totalbd}. Applying \eqref{eq_heightm_Tj_bd} to each term in the right-hand side, we obtain
\begin{align*}
&\bP(H = m, T_j > (2^k \zeta)^{2-2\gamma} / 2^{m-j + 1} \text{ for some } j =1,\dots,m) 
\\ &\hspace{1 cm} \leq \prod_{i = 1}^{m-1} \left( \frac 1 2 + 8 (2^{i-1}\zeta)^{c_1} \right) \cdot C_1 \cdot \sum_{j=1}^m  \exp \left( -c_2 \left( 2^{(2-2\gamma)(k-j-1) - (m-j+1)}  - 1 \right) \right)\notag
\\ &\hspace{1 cm} \leq \prod_{i = 1}^{m-1} \left( \frac 1 2 + 8 (2^{i-1}\zeta)^{c_1} \right) \cdot C_1 e^{c_2}\cdot \sum_{l=0}^\infty  \exp \left( -c_2 \left( 2^{(2-2\gamma)(k-m-1) -1} \cdot 2^{l(1-2\gamma)} \right) \right),\notag
\end{align*}
where in the final line we substitute $l = m-j$ and bound the finite sum above by the infinite series. We use elementary methods to bound the series. Let $\lambda_m = c_2 2^{(2-2\gamma)(k-m-1)-1}$ and $b = \log 2^{1-2\gamma}$. Then it is straightforward to obtain the bound 
\begin{align*}
\sum_{l=0}^\infty  \exp \left( -c_2 \left( 2^{(2-2\gamma)(k-m-1) -1} \cdot 2^{l(1-2\gamma)} \right) \right) &\leq e^{-\lambda_m} +  \int_{0}^\infty e^{-\lambda_m e^{b x}} dx
\\ &= e^{-\lambda_m} + \frac 1 b \int_{1}^\infty \frac{e^{-\lambda_mu}}{u} du
\\ &\leq  e^{-\lambda_m} + \frac{1}{\lambda_m b}
\\ &\leq \frac{2}{\lambda_m b},
\end{align*}
where we use $e^{-x} \leq 1/x$ and $b < 1$ in the last line. Using this upper bound and substituting the values of $\lambda_m$ and $b$ into our previous upper bound, it now follows from \eqref{eq_heightm_totalbd} that
\begin{align*}
&\bP(H = m, T_j > (2^k \zeta)^{2-2\gamma} / 2^{m-j + 1} \text{ for some } j =1,\dots,m) 
\\ &\hspace{1 cm} \leq \left(\frac{4 C_1e^{c_2}}{c_2 (1-2\gamma)\log 2} \right) \cdot \prod_{i = 1}^{m-1} \left( \frac 1 2 + 8 (2^{i-1}\zeta)^{c_1} \right) \cdot 2^{-(2-2\gamma)(k-m-1)}.
\end{align*}
Denote the leading constant by $C>0$. We now reorder terms to obtain
\begin{align*}
&\bP(H = m, T_j > (2^k \zeta)^{2-2\gamma} / 2^{m-j + 1} \text{ for some } j =1,\dots,m) 
\\ &\hspace{1 cm} \leq C 2^{-(k-m-1)} \prod_{i = 1}^{m-1} \left( \frac 1 2 + 8 (2^{i-1}\zeta)^{c_1} \right) \cdot 2^{-(1-2\gamma)(k-m-1)}
\\ &\hspace{1 cm} \leq 4C \prod_{i = 1}^{k} \left( \frac 1 2 + 8 (2^{i-1}\zeta)^{c_1} \right) \cdot 2^{-(1-2\gamma)(k-m-1)}
.\end{align*}
The last line follows by simply adding extraneous errors to the factors of $1/2$, but the resulting expression is one we have to bound anyway, as it appears in \eqref{eq_prop_downcross_heightbd}. We now return to \eqref{eq_prop_downcross_decomp}. Using the bound above for $m=1,\dots,k$ along with \eqref{eq_prop_downcross_heightbd}, we obtain
\begin{align} \label{eq_downcross_assembly1}
\bP(\tau_0 > (2^k \zeta)^{2-2\gamma}) &\leq \left(1 + 4C \sum_{m=1}^{k} 2^{-(1-2\gamma)(k-m-1)} \notag
\right) \cdot \prod_{i=1}^{k} \left( \frac 1 2 + 8(2^{i}\zeta)^{c_1} \right) \notag
\\ &\leq \left( 1 + 4C \sum_{m=-1}^\infty 2^{-(1-2\gamma)m} \right)  \cdot \prod_{i=1}^{k}  \left( \frac 1 2 + 8(2^{i}\zeta)^{c_1} \right) \notag
\\ &\leq C  \prod_{i=1}^{k}  \left( \frac 1 2 + 8(2^{i}\zeta)^{c_1} \right),
\end{align}
where $C$ is increased in the last line in a way which only depends on $\gamma$. To complete the proof, we show that the above is at most a constant multiplied by $2^{-k}$. We have 
\begin{align*}
  \prod_{i=1}^{k}   \left( \frac 1 2 + 8(2^{i}\zeta)^{c_1} \right) 
   & = \exp \left( \sum_{i=1}^{k} \log   \left( \frac 1 2 + 8(2^{i}\zeta)^{c_1} \right)  \right)
  \\ & \leq \exp \left( \sum_{i=1}^{k} \left(\log \left(\frac 1 2 \right) +  16(2^{i} \zeta)^{c_1} \right) \right)
  \\ & = \left( \frac 1 2 \right)^{k} \cdot \exp \left(16 \zeta^{c_1} \cdot 2^{c_1} \frac{2^{c_1(k-2) }-1}{2^{c_1} -1}\right)
  \\ & \leq \left( \frac 1 2 \right)^{k} \cdot \exp \left(C (2^k \zeta)^{c_1}\right),
\end{align*}
  for a constant $C>0$ which depends only on $c_1$. Since $k \leq  \log_2 (1/\zeta ) - L_1$ the exponential is bounded above by a constant factor. In particular, it follows from the above and \eqref{eq_downcross_assembly1} that, for some constant $C'$ depending only on $\gamma$, for all $k \leq  \log_2(1/\zeta) - L_1$, 
  \begin{equation*}
\bP(\tau_0 > (2^k \zeta)^{2-2\gamma}) \leq C' 2^{-k}. 
\end{equation*}
Let $L_2$ be the smallest positive integer such that $C' 2^{-L_2} \leq 1$. Then the above applied with $k$ replaced by $k+L_2$ implies that, for all $k \leq \log_2(1/\zeta) - L_1 - L_2$,
\begin{equation*}
\bP(\tau_0 > 2^{(2-2\gamma)L_2} (2^k \zeta)^{2-2\gamma}) \leq 2^{-k}.
\end{equation*}
This implies the result with $C_6 = 2^{(2-2\gamma)L_2}$ and $L_0 = L_1 + L_2$.\end{proof}

We can now quantify the time spent on downcrossings. Roughly speaking, we show that for small $\delta t$, the total duration of $\delta t \zeta^{-1}$ downcrossings of $(0,\zeta)$ is of order $(\delta t)^{2-2\gamma}$. Because it is slightly more convenient for the proof of our main theorem, we show instead that for any $a >0$, with high probability as $\delta t \to 0$, the first $\delta t \zeta^{-1}$ downcrossings take at most a constant multiple of $(\delta t)^{2-2\gamma - a}$.

\begin{lemma} \label{lemma_downtimefinal} There exist constants $C_7, C_8 \geq 1$ depending only on $\gamma$ such that the following holds: for $t > 0$ and $\delta \in (0,1]$, for any $a \in (0,2-2\gamma)$, if $\delta t$ is sufficiently small, then for sufficiently small $\zeta>0$ and any $\bP \in \sM_{\gamma,\delta}(0)$,
\begin{equation} \bP\left( T_{\mathrm{down}}  \geq C_7(\delta t)^{2-2\gamma - a} \right) \leq C_8 (\delta t)^{a/(2-2\gamma)}, \end{equation}
where $T_{\mathrm{down}} = T_{\mathrm{down}} (\lceil \delta t \zeta^{-1} \rceil)$. 
\end{lemma}

\begin{proof}
Let $\zeta>0$ be sufficiently small so that the conclusions of Proposition~\ref{prop_tau0_tail} hold. Let $t>0$, $\delta\in (0,1]$, and let $n = \lceil \delta t \zeta^{-1} \rceil$ and $T_{\text{down}} = T_{\text{down}}(n)$. We begin by recalling that, for $i \in \N$,
\[D_i = (\gamma^i_\zeta - \tau^i_\zeta)\indc(\tau^i_\zeta < \infty).\]
We then define 
 \[J_0 = \left\{ i \in [n] : D_i \leq C_6 \zeta^{2-2\gamma} \right\},\]
and, for $k \in \N$, let $I_k = [C_6 (2^{k-1} \zeta)^{2-2\gamma}, C_6 (2^{k}\zeta)^{2-2\gamma}]$, for $C_6$ from Proposition~\ref{prop_tau0_tail}, and define 
\[J_k = \left\{ i \in [n] : D_i \in I_k \right\}.\]
Next, let $K \in \N$ and set $L = \lfloor \log_2 ( n) \rfloor + K$. The value of $K$ will be specified later. For the time being, we suppose that $K$ is such that $L \leq \log_2(1/\zeta) - L_0$, where $L_0$ is defined in Proposition~\ref{prop_tau0_tail}, so that we can apply the bound from that result for all $k \leq L$.

We fix $\eta \in (0,1-2\gamma)$ and set 
\begin{equation*}
E_k = \left\{ |J_k| \leq 2^{(1+\eta)(L-k)} \right\}. 
\end{equation*}
We also define
\begin{equation*}
F = \left\{ |J_k| = 0 \text{ for all } k > L \right\}.
\end{equation*}
We remark that
\begin{align} \label{eq:Tdownbd1}
T_{\text{down}} = \sum_{i=1}^{n} D_i &\leq C_6 \zeta^{2-2\gamma} \sum_{k=0}^\infty |J_k| \cdot 2^{(2-2\gamma)k} \notag
 \\ &\leq C_6 \zeta^{2-2\gamma} \left[n  + \sum_{k=1}^\infty |J_k| \cdot 2^{(2-2\gamma)k}\right],
\end{align}
where all but finitely many of the summands must be zero, and in the second line we use $|J_0| \leq n$. In particular, on $(\cap_{k=1}^{L} E_k ) \cap F$ we have
\begin{align*}
\sum_{k=1}^\infty |J_k| \cdot 2^{(2-2\gamma)(k+1)}  &\leq \sum_{k=1}^L 2^{(1+\eta)(L-k)} \cdot 2^{(2-2\gamma)(k+1)}
\\ &\leq  2^{(1+\eta) L}  \sum_{k=1}^L 2^{(1 - 2\gamma - \eta)k}
\\ &\leq C 2^{(2-2\gamma)L}
\\ &\leq C 2^{(2-2\gamma)K} n^{2-2\gamma} ,
\end{align*}
where the the constant depends only on $\gamma$ and our choice of $\eta$, and the last line uses the definition of $L$. Substituting the above into \eqref{eq:Tdownbd1}, we obtain that on $(\cap_{k=1}^{L} E_k) \cap F$, 
\begin{align} \label{eq:Tdownsizebd}
T_{\text{down}} &\leq C \zeta^{2-2\gamma}  \left[\lceil \delta t \zeta^{-1} \rceil +(\lceil \delta t \zeta^{-1} \rceil )^{2-2\gamma}2^{(2-2\gamma)K} \right]  \leq C' (\delta t)^{2-2\gamma} 2^{(2-2\gamma)K}. 
\end{align}
where we enlarge $C'$ in the second inequality, and $C'$ depends only on $\gamma$, $\eta$ and $C_6$.

It remains to bound the probability of the event on which the above fails. To this end, we first obtain a bound for $\bP(E_k^c)$ for $i = 1,\dots,L$. For $\bP' \in \sM_{\gamma,\delta}(\zeta)$, we note that
\[ \bP'(\tau_0 \in I_k) \leq \bP'(\tau_0 > C_6 (2^{k-1}\zeta)^{2-2\gamma}) \leq 2^{-(k-1)}\]
by Proposition~\ref{prop_tau0_tail}. By definition of $D_i$, we have
\[ D_i \circ \theta_{\tau^{i}_\zeta} (\omega) = \tau_0(\omega)\]
provided $\tau^i_\zeta < \infty$. The two facts above combined with a routine application of the PSMP imply that for $i \in [n]$ and $k \in [L]$,
\begin{align*}
\bP(\tau_\zeta^i < \infty, D_i \in I_k) &= \bP(\tau_\zeta^i < \infty) \cdot \bP(D_i \in I_k \, | \, \tau_\zeta^i < \infty)
\\ &= \bP(\tau_\zeta^i < \infty) \cdot \bP^{\{\tau^i_\zeta < \infty\}} \circ \theta_{\tau_\zeta^i}(\tau_0 \in I_k)
\\ &\leq 2^{-(k-1)}.
\end{align*}
Hence
\begin{align*}
\E[|J_k|]  &= \sum_{i=1}^{n} \bP(D_i \in I_k)
\\ &= \sum_{i=1}^{n}  \bP(\tau^i_\zeta < \infty, D_i \in I_k)
\\ &\leq 2n 2^{-k}.
\end{align*}
Hence, by Markov's inequality, 
\begin{align*}
\bP(E_k^c) = \bP(|J_k| > 2^{(1+\eta)(L-k)}) &\leq 2 n 2^{-k} 2^{-(1+\eta)(L-k)}.
\end{align*}
Since $n \leq 2^{\lfloor \log_2 n \rfloor + 1} = 2^{L-K + 1}$, we obtain from the above that
\begin{align*}
\bP(E_k^c) = \bP(|J_k| > 2^{(1+\eta)(L-k)}) &\leq 2^{-(K-2)} 2^{-\eta(L-k)},
\end{align*}
and hence
\begin{align*}
\bP\left( \cup_{k=1}^{L} E_k^c\right) &\leq 2^{-(K-2)}  \sum_{k=1}^{L} 2^{-\eta(L-k)} \leq C 2^{-K},
\end{align*}
where $C$ depends on $\eta$ only. We can similarly obtain from Markov's inequality that
\begin{align*}
\bP(F^c) \leq 2 n 2^{-L} \leq 2^{-(K-2)}.
\end{align*}
Combining \eqref{eq:Tdownsizebd} with the two inequalities above, we obtain that for a constant $C'' \geq 1$ depending only on $\gamma$ and $\eta$, we have
\begin{align} \label{eq:TdownboundK}
\bP\left( T_{\text{down}} > C' (\delta t)^{2-2\gamma} 2^{(2-2\gamma)K} \right) \leq C''2^{-K}.
\end{align}
Finally, we recall that we have proved the above under the assumption $L \leq \log_2(1/\zeta) - L_0$, where $L = \lceil \log_2(\delta t \zeta^{-1}) \rceil + K$. Let $a \in (0,2-2\gamma)$ and take $K = - \lfloor \frac{a}{2-2\gamma} \log_2 (\delta t) \rfloor$. Then
\begin{align*}
    L = \lceil \log_2(\delta t \zeta^{-1}) \rceil + K \leq 2 + \left(1 - \frac{a}{2-2\gamma}\right)\log_2(\delta t) + \log_2(1/\zeta).
\end{align*}
In particular, since $a < 2-2\gamma$, if $\delta t$ is sufficiently small then $L \leq \log_2(1/\zeta) - L_0$, as desired. Substituting $K$ into \eqref{eq:TdownboundK} then gives the desired result.
\end{proof}


\subsection{Proof of Proposition~\ref{prop_downcrosscount}}
Proposition~\ref{prop_downcrosscount} now follows from the results of the previous two subsections.

\begin{proof}[Proof of Proposition~\ref{prop_downcrosscount}]
    Let us prove part (a). Fix $t>0$ and let $\eps >0$. For $\delta > 0$ and $\zeta > 0$ let $n_0 = \lceil (1-2\eps)\delta t \zeta^{-1} \rceil $. Let $\bP \in \sM_{\gamma,\delta}(0)$. By \eqref{eq:downcount_stop}, \eqref{eq_UpDownDecomp} and \eqref{defTupdown}, 
    \begin{align*}
        \bP\left( N_\zeta(t) < n_0 \right) &= \bP\left(T_{\mathrm{up}}(n_0) + T_{\mathrm{down}}(n_0) > t  \right)
        \\ &\leq \bP\left(T_{\mathrm{up}}(n_0) > (1-\eps)t\right) + \bP\left(T_{\mathrm{down}}(n_1) > \eps t  \right),
    \end{align*}
    where $n_1 = \lceil \delta t \zeta^{-1} \rceil \geq n_0$. Let $a \in (0,2-2\gamma)$. Then for sufficiently small $\delta$, we have $C_7 (\delta t)^{2-2\gamma - a} < \eps t$, where $C_7$ is one of the constants from Lemma~\ref{lemma_downtimefinal}. Hence, for sufficiently small $\delta$ and $\zeta$, for any $\bP \in \sM_{\gamma,\delta}(0)$, by Lemma~\ref{lemma_uptime} and Lemma~\ref{lemma_downtimefinal}, 
     \begin{align*}
        \bP\left( N_\zeta(t) < n_0 \right) &\leq \frac{C_5}{\eps^2} \cdot \frac{\zeta}{\delta t} + C_8 (\delta t)^{a/(2-2\gamma)}.
    \end{align*}
    For sufficiently small $\delta$, the right-hand side of the above is smaller than $\eps$. This completes the proof. 

    The proof of part (b) follows from a simpler version of the arguments in this section. Let $\delta,t>0$. Using Lemma~\ref{lemma_upcrossing} and Proposition~\ref{prop_tau0_tail}, a simpler version of the arguments from Sections~\ref{s_up} and~\ref{s_down} can prove versions of Lemmas~\ref{lemma_uptime} and~\ref{lemma_downtimefinal} which state, respectively, that for every $\eps'>0$, we may take $\eps>0$ to be sufficiently small so that for any $\bP \in \sM_{\gamma,\delta}(0)$,
    \begin{equation} \label{eq:weakupdown}
        \bP(T_{\mathrm{up}}(\lceil\eps \zeta^{-1} \rceil) > t/2) < \eps'/2 \quad \text{and} \quad \bP(T_{\mathrm{down}}(\lceil\eps \zeta^{-1}\rceil)> t/2) < \eps'/2
    \end{equation}
    for all small enough $\zeta>0$. Indeed,  the above follow from a less tight version of the proofs of Lemmas~\ref{lemma_uptime} and~\ref{lemma_downtimefinal}. It is immediate from \eqref{eq:weakupdown} that $\bP(N_\zeta(t) < \lceil \eps \zeta^{-1} \rceil ) < \eps'$, for sufficiently small $\zeta, \eps >0$, which proves Proposition~\ref{prop_downcrosscount}(b).
\end{proof}

\section{Proof of main results} \label{s_proofs}
We now complete the proofs of our main result, Theorem~\ref{thm_main_compact}, using the results of the previous section and the coupling constructed in Lemma~\ref{lemma_coupling}. We also prove Theorem~\ref{thm_unbounded}. We consider a weak solution $X$ to \eqref{e_sdesystem} satisfying Assumptions~\ref{assumption1} and \ref{assumption2}. In keeping with Remark~\ref{remark_sigmaassumpt}, we can and will assume that $\sigma$ satisfies \eqref{eq_sigmaassumpstrong} in the sequel.


\begin{proof}[Proof of Theorem~\ref{thm_main_compact}]
    Suppose Assumptions~\ref{assumption1} and~\ref{assumption2} hold, in the second case also satisfying \eqref{eq_sigmaassumpstrong}. Let $X$ be a continuous $\ell^1_+$-valued solution to \eqref{e_sdesystem} such that $X_0$ is compactly supported. For $i \in \Z$ and $\delta>0$, recall the stopping times $T^i_\delta$ and $\tau^i_1$ from \eqref{def_T} and \eqref{def_taui1}. By Lemma~\ref{lemma_coupling} and Remark~\ref{remark_countablecoupling}, there exists a family of processes $\{(\bar{Y}_t(i,n))_{t \geq 0} : i \in \Z, n \in \N\}$ such that 
	\begin{equation} \label{eq_coupling_in}
		\text{For all $i \in\Z$ and $n \in \N$, }\, \bar{Y}_t(i,n) \geq Y_t(i) \, \text{ a.s.\ for all $t \in [0,T^i_{1/n} \wedge \tau^i_1]$},
	\end{equation}
	and the following holds: for each $i \in \Z$ and $n \in \N$, $\bar{Y}(i,n)$ has Doob-Meyer decomposition
	\[\bar{Y}_t(i,n) = Y_0(i) + \frac 1 n t + M_t(i,n), \quad t \geq 0,\]
	and with probability one,
	\begin{equation*} 
		\langle M(i,n) \rangle_t - \langle M(i,n) \rangle_s \geq \frac 1 2 \int_s^t \bar{Y}_u(i,n)^{2\gamma} du, \quad 0 \leq s \leq t.
	\end{equation*}
    It then follows from Definition~\ref{def_M} that (the law of) $\bar{Y}(i,n)$ is an element of $\sM_{\gamma,\frac 1 n}$.

	 Fix $t >0$. For $i \in \Z$ and $n \in \N$, we define
	 \begin{equation*}
	 	\mathcal{Z}_t(i) = \{s \in [0,t] : Y_s(i) = 0\}, \quad 	\tilde{\mathcal{Z}}_t(i,n) = \{s \in [0,t] : \tilde{Y}_s(i,n) = 0\}.
	 \end{equation*}
	 Then by \eqref{eq_coupling_in} and the fact that $Y(i)$ and $\tilde{Y}(i,n)$ are non-negative, it follows that for any $t >0$,
	 \begin{equation*}
	 	 |\mathcal{Z}_t(i)| \geq
        |\tilde{\mathcal{Z}}_t(i,n)| \,\, \text{ on } \,\left\{t \leq T^i_{1/n} \wedge \tau^i_1 \right\}.
	 \end{equation*}
	 Let $\eps > 0$. By the above and a union bound,
	 \begin{align} \label{eq_zeroset_unionbd}
	 	\bP(|\mathcal{Z}_t(i)| < (1-2\eps)t) \leq \bP(|\tilde{\mathcal{Z}}_t(i,n)| < (1-2\eps)t) + \bP(T^{1/n}_i \wedge \tau^i_1 < t).
	 \end{align}
Since $X_0$ has compact support, there exists some $R \in \Z$ such that if $i \geq R$, then $\tilde{Y}_0(i,n) = Y_0(i) = 0$, and so $\bar{Y}(i,n) \in \sM_{\gamma,\frac{1}{n}}(0)$. Hence, by Theorem~\ref{thm_Y_zeroset}, the first probability on the right-hand side of \eqref{eq_zeroset_unionbd} is smaller than $\eps$ for sufficiently large $n$. 
Let $n_1 \in \N$ be sufficiently large so that this holds. Then for every $i \geq R$,
	 	 \begin{align*}
	 	\bP(|\mathcal{Z}_t(i)| < (1-2\eps)t) < 2\eps + \bP(T_i^{1/n_1} \wedge \tau^i_1< t).
	 \end{align*}
	 Finally, by Lemma~\ref{lemma_Tbound}, we may let $i$ be large enough (depending on $n_1$) so that the second term is at most $\eps$. Hence, for any $\eps > 0$, for sufficiently large $i$ we have $\bP(|\mathcal{Z}_t(i)| < (1-2\eps)t) < 3\eps$. In particular, 
	 \begin{equation}
	 	\lim_{i \to \infty}  \bP(\mathcal{Z}_t(i) < (1-2\eps)t) = 0
	 \end{equation}
	 for every $\eps > 0$. Hence, by choosing a sequence which goes to infinity fast enough, we may conclude from Borel-Cantelli that
	 \begin{equation}
	 	\text{For every $\eps > 0$, there a.s.\ exists some $i_\eps \in \Z$ such that $|\mathcal{Z}_t(i_\eps)| \geq (1-2\eps)t$}.
	 \end{equation}
	 In particular, with probability one, $\cup_{n \in \N} \mathcal{Z}_t(i_{1/n})$ has Lebesgue measure equal to $t$. Moreover, for every $s \in \cup_{n \in \N} \mathcal{Z}_t(i_{1/n})$, it follows that $s \in \mathcal{Z}_t(i_{1/n})$ for some $n \in \N$, and hence $Y_s(i_{1/n}) = 0$, which implies that $X_s(j) = 0$ for all $j \in \Z^d$ with $j_1 \geq i_{1/n}$. This proves that with probability one, the support of $X_t$ is bounded to the right for almost all $t \in [0,T]$. Since the same argument can be applied to the left, and in both directions along the other $d-1$ coordinate axes, this implies that with probablity one, the support of $X_t$ is compact for a.e.\ $t \in [0,T]$. The proof is complete.
\end{proof}

We conclude the paper with a short proof of Theorem~\ref{thm_unbounded}.

\begin{proof}[Proof of Theorem~\ref{thm_unbounded}]
    Let $\cL = \Delta$, the discrete Laplacian on $\Z^d$, i.e.
    \[ \Delta g(i) = \frac{1}{2d} \sum_{j \sim i} (g(j) -g(i)),\]
    where $j \sim i$ if $|i - j| = 1$. We also set $f \equiv 0$, and assume that $\sigma$ satisfies Assumption~\ref{assumption1}(b). Let $X = (X_t)_{t\geq 0}$ be a continuous $\ell^1_+$-valued solution to \eqref{e_sdesystem} under these assumptions. We further assume that $X_0$ is not identically zero. 

    First, we observe that if $X_t(i) = 0$ for all $i \in \Z^d$ for some $t \geq 0$, then $X_s(i) = 0$ for all $i \in \Z^d$ for all $s \geq t$. To see this, we simply note from Lemma~\ref{lemma_bddweak} with $\phi \equiv 1$ that the total mass process is a non-negative continuous local martingale, and that zero is hence an absorbing state. In particular, if $X_t \equiv 0$ for some $t$, then $t \geq \eta$, where the extinction time $\eta$ is defined in \eqref{def_extinction}. Furthermore, since the total mass process is continuous and is positive at time zero, we may conclude that $\eta > 0$ with probability one.

    We recall the definition of $h$ from \eqref{eq_heaviside}. We note that if $Y_t(i) = 0$ for some $i \in \Z$ and $t \geq 0$, then
    \begin{equation} \label{eq_heaviside_null}
     \langle X_t , \Delta h(\cdot - \underline{i})\rangle  = \frac{1}{2d} \sum_{j \in \Z^d : j_1 = i - 1} X_t(j) \geq 0.
    \end{equation}
    Let $0 \leq a < b$ and $i \in \Z$ and suppose that $Y_t(i) = 0$ for all $t \in [a,b]$. Using Corollary~\ref{cor_Ystochastic}, it is straightforward to argue that the martingale term in \eqref{eq_Yibp} must be constant over $[a,b]$, and we conclude that 
    \begin{align*}
        \int_a^b \langle X_t , \Delta h(\cdot - \underline{i})\rangle dt = 0.
    \end{align*}
    Hence, by \eqref{eq_heaviside_null} and continuity we have
    \begin{align*}
         \sum_{j \in \Z^d : j_1 = i - 1} X_t(j) = 0 \quad \forall t \in [a,b].
    \end{align*} 
    In particular, since $Y_t(i) = 0$ for all $t \in [a,b]$, the above implies that $Y_t(i-1) = 0$ for all $t \in [a,b]$. Repeating the same argument, we conclude that $Y_t(j) = 0$ for all $t\in[a,b]$ for all $j \leq i$, which implies that $X_t \equiv 0$ for $t \in [a,b]$ a.s. In particular,
    \[ \bP(\eta > a \text{ and } Y_t(i) = 0 \, \forall \, t \in [a,b]) = 0.\]
    Hence, by a union bound,
    \begin{align}
        \bP\left( \exists \,a,b \in \Q \cap \R_+, i \in \Z : \eta > a \text{ and } \, Y_t(i) = 0 \, \forall t \in [a,b] \right) = 0.
    \end{align}
    Now fix $\omega \in \Omega$ outside of the null set above and let $0 \leq a \leq b$. Suppose that $\eta > a$, then there exists some $a' \in (a,\eta) \cap \Q$ and $b' \in (a',b) \cap \Q$, and by our choice of $\omega$ we have $Y_t(i) > 0$ for some $t \in (a',b') \subset (a,b)$. Thus, for such $\omega$, for every interval $[a,b] \subset \R_+$, $a < \eta$ implies that $Y_t(i) > 0 $ for some $t \in [a,b]$. Thus, for every $i \in \Z$, $\{t \in [0,\eta] : Y_t(i) > 0\}$ is $\bP$-a.s.\ a dense subset of $[0,\eta]$, and in particular, with probability one, $\{t \in [0,\eta] : Y_t(i) > 0\}$ is dense for all $i \in \Z$. Since $t \mapsto Y_t(i)$ is a.s.\ continuous, $\{t \in [0,\eta] : Y_t(i) > 0\}$ is a.s.\ open (with respect to the relative topology on $[0,\eta]$).
    
   Let $\mathcal{U}:= \cap_{i \in \Z} \{t \in [0,\eta] : Y_t(i) > 0\}$. By the above, with probability one,\, $\mathcal{U}$ is a countable intersection of dense open subsets of $[0,\eta]$, and hence\, $\mathcal{U}$ is a dense subset of $[0,\eta]$ by the Baire category theorem. A dense countable intersection of open sets of $[0,\eta]$ is also necessarily uncountable. For every $t \in \cU$, $Y_t(i) > 0$ for all $i \in \Z$, and hence for all $i \in \Z$ there exists $j \in \Z^d$ with $j_1 \geq i$ such that $X_t(j) > 0$. In particular, $X_t$ has unbounded support for all $t \in \cU$. This completes the proof. \end{proof}

\noindent \textbf{Acknowledgements.} The authors thank Gerónimo Uribe Bravo for some very helpful discussions while this work was in an early stage and Ed Perkins for suggesting an improved version of Theorem~\ref{thm_unbounded}. TH is supported by an enhancement award to Royal Society grant URF/R1/211179 and MO is partially supported by EPSRC research grant EP/X040089/1.

\bibliographystyle{alpha} 
\bibliography{support}

@book {Bass,
    AUTHOR = {Bass, Richard F.},
     TITLE = {Diffusions and elliptic operators},
    SERIES = {Probability and its Applications (New York)},
 PUBLISHER = {Springer-Verlag, New York},
      YEAR = {1998},
     PAGES = {xiv+232},
      ISBN = {0-387-98315-5},
   MRCLASS = {60J10 (35B99 35J99 35R60 60H07 60H10)},
  MRNUMBER = {1483890},
MRREVIEWER = {Elton\ Pei\ Hsu},
}

@article {BBC2007,
    AUTHOR = {Bass, Richard F. and Burdzy, Krzysztof and Chen, Zhen-Qing},
     TITLE = {Pathwise uniqueness for a degenerate stochastic differential
              equation},
   JOURNAL = {Ann. Probab.},
  FJOURNAL = {The Annals of Probability},
    VOLUME = {35},
      YEAR = {2007},
    NUMBER = {6},
     PAGES = {2385--2418},
      ISSN = {0091-1798,2168-894X},
   MRCLASS = {60H10 (60J60)},
  MRNUMBER = {2353392},
MRREVIEWER = {Isamu\ D\^oku},
       DOI = {10.1214/009117907000000033},
       URL = {https://doi.org/10.1214/009117907000000033},
}

@article {BG1997,
    AUTHOR = {Bertini, Lorenzo and Giacomin, Giambattista},
     TITLE = {Stochastic {B}urgers and {KPZ} equations from particle
              systems},
   JOURNAL = {Comm. Math. Phys.},
  FJOURNAL = {Communications in Mathematical Physics},
    VOLUME = {183},
      YEAR = {1997},
    NUMBER = {3},
     PAGES = {571--607},
      ISSN = {0010-3616,1432-0916},
   MRCLASS = {60K35 (60H15 82C22 82C24)},
  MRNUMBER = {1462228},
MRREVIEWER = {Ellen\ Saada},
       DOI = {10.1007/s002200050044},
       URL = {https://doi.org/10.1007/s002200050044},
}

@article {BMP2010,
    AUTHOR = {Burdzy, Krzysztof and Mueller, Carl and Perkins, Edwin},
     TITLE = {Nonuniqueness for nonnegative solutions of parabolic
              stochastic partial differential equations},
   JOURNAL = {Illinois J. Math.},
  FJOURNAL = {Illinois Journal of Mathematics},
    VOLUME = {54},
      YEAR = {2010},
    NUMBER = {4},
     PAGES = {1481--1507},
      ISSN = {0019-2082,1945-6581},
   MRCLASS = {60H15 (60G60 60H10)},
  MRNUMBER = {2981857},
MRREVIEWER = {Jie\ Xiong},
       URL = {http://projecteuclid.org/euclid.ijm/1348505538},
}

@article {C2015,
    AUTHOR = {Chen, Yu-Ting},
     TITLE = {Pathwise nonuniqueness for the {SPDE}s of some
              super-{B}rownian motions with immigration},
   JOURNAL = {Ann. Probab.},
  FJOURNAL = {The Annals of Probability},
    VOLUME = {43},
      YEAR = {2015},
    NUMBER = {6},
     PAGES = {3359--3467},
      ISSN = {0091-1798,2168-894X},
   MRCLASS = {60H15 (35A02 35R60 60J68)},
  MRNUMBER = {3433584},
MRREVIEWER = {Hui\ He},
       DOI = {10.1214/14-AOP962},
       URL = {https://doi.org/10.1214/14-AOP962},
}

@article {D1978,
    AUTHOR = {Dawson, Donald A.},
     TITLE = {Geostochastic calculus},
   JOURNAL = {Canad. J. Statist.},
  FJOURNAL = {The Canadian Journal of Statistics. La Revue Canadienne de
              Statistique},
    VOLUME = {6},
      YEAR = {1978},
    NUMBER = {2},
     PAGES = {143--168},
      ISSN = {0319-5724,1708-945X},
   MRCLASS = {60J60 (60G44 60H05 92A09)},
  MRNUMBER = {532855},
MRREVIEWER = {Yu.\ L.\ Daletski\u i},
       DOI = {10.2307/3315044},
       URL = {https://doi.org/10.2307/3315044},
}

@article {DP1998,
    AUTHOR = {Dawson, Donald A. and Perkins, Edwin},
     TITLE = {Long-time behavior and coexistence in a mutually catalytic
              branching model},
   JOURNAL = {Ann. Probab.},
  FJOURNAL = {The Annals of Probability},
    VOLUME = {26},
      YEAR = {1998},
    NUMBER = {3},
     PAGES = {1088--1138},
      ISSN = {0091-1798,2168-894X},
   MRCLASS = {60K35 (60H10 60H15 60J80)},
  MRNUMBER = {1634416},
MRREVIEWER = {Klaus\ Fleischmann},
       DOI = {10.1214/aop/1022855746},
       URL = {https://doi.org/10.1214/aop/1022855746},
}

@book {DellacherieMeyerA,
    AUTHOR = {Dellacherie, Claude and Meyer, Paul-Andr\'e},
     TITLE = {Probabilit\'es et potentiel},
    SERIES = {Publications de l'Institut de Math\'ematique de
              l'Universit\'e{} de Strasbourg},
    VOLUME = {No. XV},
      NOTE = {Chapitres I \`a{} IV,
              \'Edition enti\`erement refondue,
              Actualit\'es Scientifiques et Industrielles, No. 1372.
              [Current Scientific and Industrial Topics]},
 PUBLISHER = {Hermann, Paris},
      YEAR = {1975},
     PAGES = {x+291},
   MRCLASS = {60-02 (28-02 28A05 31C15)},
  MRNUMBER = {488194},
MRREVIEWER = {M.\ G.\ Shur},
}

@article {EP2014,
    AUTHOR = {Engelbert, Hans-J\"urgen and Peskir, Goran},
     TITLE = {Stochastic differential equations for sticky {B}rownian
              motion},
   JOURNAL = {Stochastics},
  FJOURNAL = {Stochastics. An International Journal of Probability and
              Stochastic Processes},
    VOLUME = {86},
      YEAR = {2014},
    NUMBER = {6},
     PAGES = {993--1021},
      ISSN = {1744-2508,1744-2516},
   MRCLASS = {60H10 (60H20 60J50 60J55 60J60 60J65)},
  MRNUMBER = {3271518},
MRREVIEWER = {Chang-Song\ Deng},
       DOI = {10.1080/17442508.2014.899600},
       URL = {https://doi.org/10.1080/17442508.2014.899600},
}

@article {ES1985,
    AUTHOR = {Engelbert, Hans-J\"urgen and Schmidt, Wolfgang},
     TITLE = {On solutions of one-dimensional stochastic differential
              equations without drift},
   JOURNAL = {Z. Wahrsch. Verw. Gebiete},
  FJOURNAL = {Zeitschrift f\"ur Wahrscheinlichkeitstheorie und Verwandte
              Gebiete},
    VOLUME = {68},
      YEAR = {1985},
    NUMBER = {3},
     PAGES = {287--314},
      ISSN = {0044-3719},
   MRCLASS = {60H10},
  MRNUMBER = {771468},
MRREVIEWER = {M.\ M\'etivier},
       DOI = {10.1007/BF00532642},
       URL = {https://doi.org/10.1007/BF00532642},
}

@article {EP1991,
    AUTHOR = {Evans, Steven and Perkins, Edwin},
     TITLE = {Absolute continuity results for superprocesses with some
              applications},
   JOURNAL = {Trans. Amer. Math. Soc.},
  FJOURNAL = {Transactions of the American Mathematical Society},
    VOLUME = {325},
      YEAR = {1991},
    NUMBER = {2},
     PAGES = {661--681},
      ISSN = {0002-9947,1088-6850},
   MRCLASS = {60G30 (60J80)},
  MRNUMBER = {1012522},
MRREVIEWER = {Klaus\ Fleischmann},
       DOI = {10.2307/2001643},
       URL = {https://doi.org/10.2307/2001643},
}

@article {HKY2023,
    AUTHOR = {Han, Beom-Seok and Kim, Kunwoo and Yi, Jaeyun},
     TITLE = {The compact support property for solutions to the stochastic
              partial differential equations with colored noise},
   JOURNAL = {SIAM J. Math. Anal.},
  FJOURNAL = {SIAM Journal on Mathematical Analysis},
    VOLUME = {55},
      YEAR = {2023},
    NUMBER = {6},
     PAGES = {7665--7703},
      ISSN = {0036-1410,1095-7154},
   MRCLASS = {60H15 (35D30 35R60)},
  MRNUMBER = {4665042},
       DOI = {10.1137/23M1557283},
       URL = {https://doi.org/10.1137/23M1557283},
}

@article {HKY2024,
    AUTHOR = {Han, Beom-Seok and Kim, Kunwoo and Yi, Jaeyun},
     TITLE = {On the support of solutions to nonlinear stochastic heat
              equations},
   JOURNAL = {Electron. J. Probab.},
  FJOURNAL = {Electronic Journal of Probability},
    VOLUME = {29},
      YEAR = {2024},
     PAGES = {Paper No. 201, 29},
      ISSN = {1083-6489},
   MRCLASS = {60H15 (35K10 35K55 35R60)},
  MRNUMBER = {4842851},
       DOI = {10.1214/24-ejp1261},
       URL = {https://doi.org/10.1214/24-ejp1261},
}

@article {HM2025,
    AUTHOR = {Hong, Jieliang and Mytnik, Leonid},
     TITLE = {Exceptional times for the instantaneous propagation of
              superprocess},
   JOURNAL = {Trans. Amer. Math. Soc. Ser. B},
  FJOURNAL = {Transactions of the American Mathematical Society. Series B},
    VOLUME = {12},
      YEAR = {2025},
     PAGES = {470--515},
      ISSN = {2330-0000},
   MRCLASS = {60J68 (60G17)},
  MRNUMBER = {4887195},
MRREVIEWER = {Xiaowen\ Zhou},
       DOI = {10.1090/btran/226},
       URL = {https://doi.org/10.1090/btran/226},
}

@article {H2025,
    AUTHOR = {Hughes, Thomas},
     TITLE = {The compact support property for solutions to stochastic heat
              equations with stable noise},
   JOURNAL = {Electron. J. Probab.},
  FJOURNAL = {Electronic Journal of Probability},
    VOLUME = {30},
      YEAR = {2025},
     PAGES = {Paper No. 106, 60},
      ISSN = {1083-6489},
   MRCLASS = {60H15 (60G17 60G52)},
  MRNUMBER = {4926009},
       DOI = {10.1214/25-ejp1350},
       URL = {https://doi.org/10.1214/25-ejp1350},
}

@article {I1988,
    AUTHOR = {Iscoe, Ian},
     TITLE = {On the supports of measure-valued critical branching
              {B}rownian motion},
   JOURNAL = {Ann. Probab.},
  FJOURNAL = {The Annals of Probability},
    VOLUME = {16},
      YEAR = {1988},
    NUMBER = {1},
     PAGES = {200--221},
      ISSN = {0091-1798,2168-894X},
   MRCLASS = {60G57 (34B15 60J80)},
  MRNUMBER = {920265},
MRREVIEWER = {Ralf\ Manthey},
       URL =
              {http://links.jstor.org/sici?sici=0091-1798(198801)16:1<200:OTSOMC>2.0.CO;2-Z&origin=MSN},
}

@article {KS1988,
    AUTHOR = {Konno, Norio and Shiga, Tokuzo},
     TITLE = {Stochastic partial differential equations for some
              measure-valued diffusions},
   JOURNAL = {Probab. Theory Related Fields},
  FJOURNAL = {Probability Theory and Related Fields},
    VOLUME = {79},
      YEAR = {1988},
    NUMBER = {2},
     PAGES = {201--225},
      ISSN = {0178-8051,1432-2064},
   MRCLASS = {60H15 (60J60)},
  MRNUMBER = {958288},
MRREVIEWER = {Ali\ S\"uleyman\ \"Ust\"unel},
       DOI = {10.1007/BF00320919},
       URL = {https://doi.org/10.1007/BF00320919},
}

@article {K1997,
    AUTHOR = {Krylov, Nicolai V.},
     TITLE = {On a result of {C}. {M}ueller and {E}. {P}erkins},
   JOURNAL = {Probab. Theory Related Fields},
  FJOURNAL = {Probability Theory and Related Fields},
    VOLUME = {108},
      YEAR = {1997},
    NUMBER = {4},
     PAGES = {543--557},
      ISSN = {0178-8051,1432-2064},
   MRCLASS = {60H15},
  MRNUMBER = {1465641},
MRREVIEWER = {Ya.\ \=I.\ B\=\i lopol\cprime s\cprime ka},
       DOI = {10.1007/s004400050120},
       URL = {https://doi.org/10.1007/s004400050120},
}

@article {MF2014,
    AUTHOR = {Moreno Flores, Gregorio R.},
     TITLE = {On the (strict) positivity of solutions of the stochastic heat
              equation},
   JOURNAL = {Ann. Probab.},
  FJOURNAL = {The Annals of Probability},
    VOLUME = {42},
      YEAR = {2014},
    NUMBER = {4},
     PAGES = {1635--1643},
      ISSN = {0091-1798,2168-894X},
   MRCLASS = {60H15 (60K35 60K37)},
  MRNUMBER = {3262487},
MRREVIEWER = {Xinghua\ Zheng},
       DOI = {10.1214/14-AOP911},
       URL = {https://doi.org/10.1214/14-AOP911},
}

@article {M1990,
    AUTHOR = {Mueller, Carl},
     TITLE = {On the support of solutions to the heat equation with noise},
   JOURNAL = {Stochastics Stochastics Rep.},
  FJOURNAL = {Stochastics and Stochastics Reports},
    VOLUME = {37},
      YEAR = {1991},
    NUMBER = {4},
     PAGES = {225--245},
      ISSN = {1045-1129},
   MRCLASS = {60H15 (35R60)},
  MRNUMBER = {1149348},
MRREVIEWER = {A.\ Badrikian},
       DOI = {10.1080/17442509108833738},
       URL = {https://doi.org/10.1080/17442509108833738},
}

@article {M2000,
    AUTHOR = {Mueller, Carl},
     TITLE = {The critical parameter for the heat equation with a noise term
              to blow up in finite time},
   JOURNAL = {Ann. Probab.},
  FJOURNAL = {The Annals of Probability},
    VOLUME = {28},
      YEAR = {2000},
    NUMBER = {4},
     PAGES = {1735--1746},
      ISSN = {0091-1798,2168-894X},
   MRCLASS = {60H15 (35B40 35K05 35R60 60H40)},
  MRNUMBER = {1813841},
MRREVIEWER = {Richard\ B.\ Sowers},
       DOI = {10.1214/aop/1019160505},
       URL = {https://doi.org/10.1214/aop/1019160505},
}

@article {MMP2014,
    AUTHOR = {Mueller, Carl and Mytnik, Leonid and Perkins, Edwin},
     TITLE = {Nonuniqueness for a parabolic {SPDE} with
              {$\frac{3}{4}-\epsilon$}-{H}\"older diffusion coefficients},
   JOURNAL = {Ann. Probab.},
  FJOURNAL = {The Annals of Probability},
    VOLUME = {42},
      YEAR = {2014},
    NUMBER = {5},
     PAGES = {2032--2112},
      ISSN = {0091-1798,2168-894X},
   MRCLASS = {60H15 (35A02 35K15 35K91 35R60)},
  MRNUMBER = {3262498},
MRREVIEWER = {Stefano\ Bonaccorsi},
       DOI = {10.1214/13-AOP870},
       URL = {https://doi.org/10.1214/13-AOP870},
}

@article {MMR2021,
    AUTHOR = {Mueller, Carl and Mytnik, Leonid and Ryzhik, Lenya},
     TITLE = {The speed of a random front for stochastic reaction-diffusion
              equations with strong noise},
   JOURNAL = {Comm. Math. Phys.},
  FJOURNAL = {Communications in Mathematical Physics},
    VOLUME = {384},
      YEAR = {2021},
    NUMBER = {2},
     PAGES = {699--732},
      ISSN = {0010-3616,1432-0916},
   MRCLASS = {35K57 (60H15 92D10)},
  MRNUMBER = {4259374},
       DOI = {10.1007/s00220-021-04084-0},
       URL = {https://doi.org/10.1007/s00220-021-04084-0},
}

@article {MP1992,
    AUTHOR = {Mueller, Carl and Perkins, Edwin},
     TITLE = {The compact support property for solutions to the heat
              equation with noise},
   JOURNAL = {Probab. Theory Related Fields},
  FJOURNAL = {Probability Theory and Related Fields},
    VOLUME = {93},
      YEAR = {1992},
    NUMBER = {3},
     PAGES = {325--358},
      ISSN = {0178-8051,1432-2064},
   MRCLASS = {60H15},
  MRNUMBER = {1180704},
MRREVIEWER = {Ya.\ \=I.\ B\=\i lopol\cprime s\cprime ka},
       DOI = {10.1007/BF01193055},
       URL = {https://doi.org/10.1007/BF01193055},
}

@article {M1998,
    AUTHOR = {Mytnik, Leonid},
     TITLE = {Weak uniqueness for the heat equation with noise},
   JOURNAL = {Ann. Probab.},
  FJOURNAL = {The Annals of Probability},
    VOLUME = {26},
      YEAR = {1998},
    NUMBER = {3},
     PAGES = {968--984},
      ISSN = {0091-1798,2168-894X},
   MRCLASS = {60H15 (35K05 35R60)},
  MRNUMBER = {1634410},
MRREVIEWER = {Fred\ Espen\ Benth},
       DOI = {10.1214/aop/1022855740},
       URL = {https://doi.org/10.1214/aop/1022855740},
}

@article {P1990,
    AUTHOR = {Perkins, Edwin},
     TITLE = {Polar sets and multiple points for super-{B}rownian motion},
   JOURNAL = {Ann. Probab.},
  FJOURNAL = {The Annals of Probability},
    VOLUME = {18},
      YEAR = {1990},
    NUMBER = {2},
     PAGES = {453--491},
      ISSN = {0091-1798,2168-894X},
   MRCLASS = {60G17 (31C15 60G57 60J45)},
  MRNUMBER = {1055416},
MRREVIEWER = {Steven\ N.\ Evans},
       URL =
              {http://links.jstor.org/sici?sici=0091-1798(199004)18:2<453:PSAMPF>2.0.CO;2-3&origin=MSN},
}

@incollection {Perkins,
    AUTHOR = {Perkins, Edwin},
     TITLE = {Dawson-{W}atanabe superprocesses and measure-valued
              diffusions},
 BOOKTITLE = {Lectures on probability theory and statistics
              ({S}aint-{F}lour, 1999)},
    SERIES = {Lecture Notes in Math.},
    VOLUME = {1781},
     PAGES = {125--324},
 PUBLISHER = {Springer, Berlin},
      YEAR = {2002},
      ISBN = {3-540-43736-3},
   MRCLASS = {60G57 (60H15 60J80 60K35)},
  MRNUMBER = {1915445},
MRREVIEWER = {Anton\ Wakolbinger},
}

@article {R1989,
    AUTHOR = {Reimers, Mark},
     TITLE = {One-dimensional stochastic partial differential equations and
              the branching measure diffusion},
   JOURNAL = {Probab. Theory Related Fields},
  FJOURNAL = {Probability Theory and Related Fields},
    VOLUME = {81},
      YEAR = {1989},
    NUMBER = {3},
     PAGES = {319--340},
      ISSN = {0178-8051,1432-2064},
   MRCLASS = {60H15 (03H05 35R60)},
  MRNUMBER = {983088},
       DOI = {10.1007/BF00340057},
       URL = {https://doi.org/10.1007/BF00340057},
}

@book {RevuzYor,
    AUTHOR = {Revuz, Daniel and Yor, Marc},
     TITLE = {Continuous martingales and {B}rownian motion},
    SERIES = {Grundlehren der mathematischen Wissenschaften [Fundamental
              Principles of Mathematical Sciences]},
    VOLUME = {293},
   EDITION = {Third},
 PUBLISHER = {Springer-Verlag, Berlin},
      YEAR = {1999},
     PAGES = {xiv+602},
      ISBN = {3-540-64325-7},
   MRCLASS = {60G44 (60G07 60H05)},
  MRNUMBER = {1725357},
       DOI = {10.1007/978-3-662-06400-9},
       URL = {https://doi.org/10.1007/978-3-662-06400-9},
}

@book {RogersWilliams2,
    AUTHOR = {Rogers, L. C. G. and Williams, David},
     TITLE = {Diffusions, {M}arkov processes, and martingales. {V}ol. 2},
    SERIES = {Cambridge Mathematical Library},
      NOTE = {It\^o{} calculus,
              Reprint of the second (1994) edition},
 PUBLISHER = {Cambridge University Press, Cambridge},
      YEAR = {2000},
     PAGES = {xiv+480},
      ISBN = {0-521-77593-0},
   MRCLASS = {60J60 (60G07 60H05 60J25)},
  MRNUMBER = {1780932},
       DOI = {10.1017/CBO9781107590120},
       URL = {https://doi.org/10.1017/CBO9781107590120},
}

@article {S1994,
    AUTHOR = {Shiga, Tokuzo},
     TITLE = {Two contrasting properties of solutions for one-dimensional
              stochastic partial differential equations},
   JOURNAL = {Canad. J. Math.},
  FJOURNAL = {Canadian Journal of Mathematics. Journal Canadien de
              Math\'ematiques},
    VOLUME = {46},
      YEAR = {1994},
    NUMBER = {2},
     PAGES = {415--437},
      ISSN = {0008-414X,1496-4279},
   MRCLASS = {60H15 (35R60)},
  MRNUMBER = {1271224},
MRREVIEWER = {Ralf\ Manthey},
       DOI = {10.4153/CJM-1994-022-8},
       URL = {https://doi.org/10.4153/CJM-1994-022-8},
}

@article {SS1980,
    AUTHOR = {Shiga, Tokuzo and Shimizu, Akinobu},
     TITLE = {Infinite-dimensional stochastic differential equations and
              their applications},
   JOURNAL = {J. Math. Kyoto Univ.},
  FJOURNAL = {Journal of Mathematics of Kyoto University},
    VOLUME = {20},
      YEAR = {1980},
    NUMBER = {3},
     PAGES = {395--416},
      ISSN = {0023-608X},
   MRCLASS = {60H10 (60J60 60K35 82A05 92A10)},
  MRNUMBER = {591802},
MRREVIEWER = {Yu.\ L.\ Daletski\u i},
       DOI = {10.1215/kjm/1250522207},
       URL = {https://doi.org/10.1215/kjm/1250522207},
}

\end{document}